\documentclass[11pt,a4paper,reqno]{amsart}

\usepackage[utf8]{inputenc}
\usepackage[T1]{fontenc} 
\usepackage[english]{babel} 
\usepackage{amsmath,amssymb,amsthm}
\usepackage{graphicx}
\usepackage{geometry}
\usepackage{xcolor}
\usepackage{enumitem}
\usepackage{hyperref}
\usepackage{braket,stix}

\usepackage{amssymb,amsthm,amsmath,amsxtra,dsfont,bbm,stix,color}
\usepackage{hyperref,cleveref} 

\usepackage{environ} 
\usepackage{pgfkeys} 

\usepackage{mathtools} 
\usepackage{stmaryrd} 
\usepackage{braket} 
\usepackage{enumitem}
\usepackage{mathrsfs}

\allowdisplaybreaks

\numberwithin{equation}{section}
\makeatletter

\makeatletter
\def\@tocline#1#2#3#4#5#6#7{\relax
  \ifnum #1>\c@tocdepth 
  \else
    \par \addpenalty\@secpenalty\addvspace{#2}%
    \begingroup \hyphenpenalty\@M
    \@ifempty{#4}{%
      \@tempdima\csname r@tocindent\number#1\endcsname\relax
    }{%
      \@tempdima#4\relax
    }%
    \parindent\z@ \leftskip#3\relax \advance\leftskip\@tempdima\relax
    \rightskip\@pnumwidth plus4em \parfillskip-\@pnumwidth
    #5\leavevmode\hskip-\@tempdima
      \ifcase #1
       \or\or \hskip 1em \or \hskip 2em \else \hskip 3em \fi%
      #6\nobreak\relax
      \dotfill
      \hbox to\@pnumwidth{\@tocpagenum{#7}}
    \par
    \nobreak
    \endgroup
  \fi}
\makeatother

\newcommand{\tr}[1]{\textnormal{Tr}\left[ #1 \right]}
\newcommand{\ps}[2]{\left\langle #1 , #2 \right\rangle}
\newcommand{\dgamma}[1]{\textnormal{d}\Gamma \left( #1 \right)}
\newcommand{\wick}[1]{:\!#1\!:}
\newcommand{\gF}{\mathfrak{F}}
\newcommand{\gH}{\mathfrak{H}}
\newcommand{\bT}{\mathbb{T}}
\newcommand{\cN}{\mathcal{N}}
\newcommand{\sym}{\mathrm{sym}}
\newcommand{\1}{\mathbb{1}}

\usepackage{graphicx}
\usepackage{lipsum}

\calclayout

\newtheorem{theorem}{Theorem}[section]

\newtheorem{proposition}[theorem]{Proposition}
\newtheorem{remark}[theorem]{Remark}
\newtheorem{lemma}[theorem]{Lemma}
\newtheorem{definition}[theorem]{Definition}

\numberwithin{equation}{section}
\numberwithin{figure}{section}

\title{$\Phi^4_2$ theory limit of a many-body bosonic free energy}

\author[L. Jougla]{Lucas Jougla}
\address{Constructor university Bremen}
\email{ljougla@constructor.university}

\author[N. Rougerie]{Nicolas Rougerie}
\address{Ecole Normale Sup\'erieure de Lyon \& CNRS,  UMPA (UMR 5669)}
\email{nicolas.rougerie@ens-lyon.fr}

\date{May 2026}

\begin{document}

\maketitle

\begin{abstract}
We consider the quantum Gibbs state of an interacting Bose gas on the 2D torus $\mathbb{T}^2$. We set temperature, chemical potential and coupling constant in a regime where classical field theory gives leading order asymptotics. In the same limit, the repulsive interaction potential is set to be short-range: it converges to a Dirac delta function with a rate depending polynomially on the other scaling parameters. We prove that the free-energy of the interacting Bose gas (counted relatively to the non-interacting one) converges to the free energy of the $\Phi^4_2$ non-linear Schr\"odinger-Gibbs measure, thereby revisiting recent results and streamlining proofs thereof. We combine the variational method of Lewin-Nam-Rougerie to connect, with controled error, the quantum free energy to a classical Hartree-Gibbs one with smeared non-linearity. The convergence of the latter to the $\Phi^4_2$ free energy then follows from arguments of Fr\"ohlich-Knowles-Schlein-Sohinger. This derivation parallels recent results of Nam-Zhu-Zhu.
\end{abstract}

\tableofcontents

\section{Introduction}
The $\Phi^4_2$ measure on the torus $\mathbb{T}^2$ is formally defined as
\begin{equation} \label{eq: phi42 def}
    \textnormal{d}\mu(u) = \frac{1}{\mathcal{Z}} \exp \left( -\int_{\mathbb{T}^2} \left(|\nabla u(x)|^2 + |u(x)|^2 + \frac{1}{2} |u(x)|^4 \right) \textnormal{d}x  \right) \textnormal{d}u \ .
\end{equation}
Defining rigorously such an object is non-trivial, for every relevant term in \eqref{eq: phi42 def} is infinite almost surely on the support of the measure. The precise definition (recalled below) goes back to Nelson~\cite{Nelson-66} and was a landmark of the constructive field theory program, see e.g.~\cite{GliJaf-87,Simon-74,Hairer-16} for review. \\

One can see the above measure as the formal limit when $\varepsilon \rightarrow 0^+$ of the nonlinear Gibbs measure
\begin{equation} \label{eq:nonlinear gibbs def}
    \textnormal{d}\mu^\varepsilon(u) = \frac{1}{\mathcal{Z}^\varepsilon} \exp \left( -\int_{\mathbb{T}^2} \left(|\nabla u(x)|^2 + |u(x)|^2\right)\textnormal{d}x - \frac{1}{2} \iint_{\mathbb{T}^2 \times \mathbb{T}^2} w^\varepsilon(x-y) |u(x)|^2|u(y)|^2\textnormal{d}x\textnormal{d}y  \right) \textnormal{d}u \ , \\
\end{equation}
where
\begin{equation}
    w^\varepsilon(x) = \frac{1}{\varepsilon^2}w\left(\frac{x}{\varepsilon}\right)
\end{equation}
is an approximation of the Dirac delta $\delta_0$, and $w$ is sufficiently nice. In the above definitions, $\mathcal{Z}$ and $\mathcal{Z}^\varepsilon$ are normalization constants.  This is also formal, for the same reasons  as in \eqref{eq: phi42 def}. The above measures can be defined as being absolutely continuous with respect to the Gaussian measure 
\begin{equation}
    \textnormal{d}\mu_0(u) = \frac{1}{\mathcal{Z}_0}\exp \left( -\int_{\mathbb{T}^2} \left(|\nabla u(x)|^2 + |u(x)|^2\right)\textnormal{d}x \right) \textnormal{d}u
\end{equation}
with covariance $h^{-1}$, on the appropriate Sobolev space. Here,  
$$h = - \Delta+1$$ 
on $\mathfrak{H} = L^2(\mathbb{T}^2)$, see e.g.~\cite[Lemma~5.1]{LewNamRou-20}. Due to the support properties of this free Gibbs measure, there are still infinite terms in~\eqref{eq:nonlinear gibbs def}, hence the interaction term has to be ``Wick-renormalized''. This is fairly classical, we refer to~\cite{FroKnoSchSoh-20,LewNamRou-20} for rigorous definitions in a context tailored for our needs below. After this Wick renormalization, one finds that the interaction term
\begin{equation}
    \textnormal{Int}_\varepsilon (u)= \frac{1}{2} \iint_{\mathbb{T}^2 \times \mathbb{T}^2} w^\varepsilon(x-y) |u(x)|^2|u(y)|^2\textnormal{d}x\textnormal{d}y 
\end{equation}
can be rigorously defined as the limit $K\to \infty$  (e.g. in $L^2(\textnormal{d}\mu_0)$) of its truncated  counterpart
\begin{equation}
    \frac{1}{2}\iint_{\mathbb{T}^2 \times \mathbb{T}^2} w^\varepsilon(x-y) \wick{|P_Ku(x)|^2}\wick{|P_Ku(y)|^2} \textnormal{d}x\textnormal{d}y - \tau^\varepsilon \int_{\mathbb{T}^2}\wick{|P_Ku(x)|^2} \textnormal{d}x - E^\varepsilon \ , \\
\end{equation}
where $P_K$ is the orthogonal projection onto the first $K$ eigenmodes of $h$, 
$$\wick{|P_Ku(x)|^2}= |P_Ku(x)|^2 - \left\langle |P_Ku(x)|^2 \right\rangle_{\mu_0} $$
with $\left\langle \, . \,  \right\rangle_{\mu_0}$ the expectation value in the gaussian measure and $\tau^\varepsilon, E^\varepsilon > 0$ are counterterms that come from the Wick ordering of the quartic term $|u(x)|^2|u(y)|^2$. Rigorously linking \eqref{eq:nonlinear gibbs def} and \eqref{eq: phi42 def} turns out to be highly  non trivial. See~\cite{FroKnoSchSoh-22} for our case of dimension $d=2$, and~\cite{NamZhuZhu-25} for another approach to $d=2$ that extends to $d=3$. We will take this derivation for granted in the sequel. \\

Our main goal is to derive the $\Phi^4_2$ measure as a ``limit'' in a suitable sense of a quantum Gibbs state describing an homogeneous system of bosons. We will focus for simplicity on the derivation of the free-energy, i.e. minus the logarithm of the constant normalising $\mu$ relatively to $\mu_0$. This has been obtained first in~\cite{FroKnoSchSoh-22}, in a limit where $\varepsilon \to 0$ very slowly (logarithmically) as a function of the other scaling parameters of the quantum theory (temperature, chemical potential, coupling constant). Adapting tools from~\cite{LewNamRou-20} and combining them with the study~\cite{FroKnoSchSoh-22} of the limit from~\eqref{eq:nonlinear gibbs def} to~\eqref{eq: phi42 def} we obtain the same result, but allowing a faster (polynomial) rate $\varepsilon \to 0$. The point here, even if we do not quite reach it in this paper, is that a physically natural rate would be to have $\varepsilon$ much smaller than the typical inter-particle distance (dilute regime, typical in cold atoms physics, see~\cite{Rougerie-EMS} for a general discussion). A polynomial rate of $\varepsilon \to 0$ has been simultaneously obtained in~\cite{NamZhuZhu-25} with related methods. The main focus of the latter paper is however the 3D case, which goes far beyond the scope of the present text. The derivation of non-linear Gibbs measures from many-body quantum mechanics is a relatively recent endeavor~\cite{AmmRat-21,FroKnoSchSoh-16,FroKnoSchSoh-17,FroKnoSchSoh-20,FroKnoSchSoh-20b,FroKnoSchSoh-22,DinRou-24,LewNamRou-15,LewNamRou-17,LewNamRou-18_2D,LewNamRou-20,Sohinger-22,RouSoh-22,RouSoh-23,NamLuZhu-26,NamZhuZhu-25,NamZhuZhu-26} and our hope is that the present text can help in explaining some of the relevent methods. \\


Our goal is thus to obtain the convergence -- in a sense that will be made precise later -- of the grand-canonical many-body quantum Gibbs state
\begin{equation} \label{eq:generic Gibbs state}
        \Gamma_\lambda = \frac{e^{-\lambda\mathbb{H}}}{\textnormal{Tr}_{\mathfrak{F}}\left[ e^{-\lambda\mathbb{H}} \right]}
\end{equation}
as $\lambda \rightarrow 0^+$, where 
$$\mathbb{H} = \bigoplus_{n\geq 0} H_{\lambda,\nu,n}$$
is given as the lifting to the bosonic Fock space 
$$ \gF:= \bigoplus_{n\geq 0} L^2_{\sym} \left(\bT^{2n}\right)$$
of the usual $n$-body Schr\"odinger Hamiltonian
\begin{equation} \label{eq: generic nbody Hamiltonian}
    H_{\lambda,\nu,n} = \sum_{j=1}^n \left( - \Delta_{x_j} + 1 - \nu \right) + \lambda \sum_{1 \leq j < \ell \leq n}w^\varepsilon (x_j - x_\ell) \ 
\end{equation}
 acting on $L^2_{\sym} \left(\bT^{2n}\right)$. Here, $-\Delta_{x_j}$ is the Laplacian acting on the $j^{\textnormal{th}}$ coordinate, $w^\varepsilon$ is a scaled two-body multiplication operator as above, and $\nu$ is the \textit{chemical potential}, which sets the average particle number and will be crucial in order to renormalize the theory. \\
 

We study the following interacting and non-interacting partition functions
\begin{equation}
    \begin{split}
        Z(\lambda) & = 1 + \sum_{n \geq 1}\textnormal{Tr} \left[e^{-\lambda H_{\lambda,\nu,n}} \right] = \textnormal{Tr}_{\mathfrak{F}}\left[ e^{-\lambda\mathbb{H}} \right]\\
        Z_0(\lambda) & = 1 + \sum_{n \geq 1}\textnormal{Tr} \left[e^{-\lambda H_{\lambda=0,\nu,n}} \right].
    \end{split}
\end{equation}
We will consider the ratio 
\begin{equation}
    \frac{Z(\lambda)}{Z_0(\lambda)}
\end{equation}
whose logarithm gives the \textit{relative free energy} of the quantum model. We prove that it converges, in the limit $\lambda \rightarrow 0, \varepsilon \to 0$, towards the relative partition function of the classical model, formally given by $\mathcal{Z}$ in \eqref{eq: phi42 def}. We generalize estimates from ~\cite{LewNamRou-20}, making them quantitative as a function of $\varepsilon$ in order to allow it to converge to $0$ as fast as possible. Again, an important challenge is to allow for a physically relevant $\varepsilon$, meaning smaller as the typical distance between the particles, i.e
\begin{equation} \label{eqcondition epsilon}
    \varepsilon \ll N^{-1/d} \mbox{ where } N:= \tr{\cN \Gamma_\lambda}
\end{equation}
with the particle number operator 
$$\cN = \bigoplus_{n\geq 0} n \1_{L^2_{\sym} \left(\bT^{2n}\right)}.$$
In fact we will manage to get $\varepsilon \sim N^{-1/24}$ essentially, with the other parameters of the Hamiltonian scaled as in~\eqref{eq:renorm param}-\eqref{eq:phys param}. This still brings us closer to the relevant regime~\eqref{eqcondition epsilon} than the previous work~\cite{FroKnoSchSoh-22}. A truly dilute regime however remains a challenging open problem.

    \bigskip
    
    \noindent \textbf{Acknowledgments.} After starting this project (which is the master memoir of the first author), we learned from P.T. Nam of the ongoing work that would culminate in~\cite{NamZhuZhu-25}. Many thanks to the authors of this paper for their friendly attitude towards the present work, in particular in the places where it overlaps with theirs. Several discussions with P.T. Nam, M. Lewin and A. Knowles were aslo helpful.

\section{Definitions and main result}
We give here the main definitions that we will use, and we formulate our main result, as well as the proof strategy, which will is inspired by that outlined in~\cite[Section~4]{LewNamRou-20}.

\subsection{Fock space formalism, Quantum and Classical models}

\noindent\textbf{Hilbert space.} We consider a system of many interacting \textit{bosons} on the $2\textnormal{D}$ torus $\mathbb{T}^2 = \left[0,2 \pi\right]^2$. The Hilbert space for one particle is 
\begin{equation}
    \mathfrak{H} = L^2\left(\mathbb{T}^2\right)\ . \\
\end{equation}
We work grand-canonically, i.e we do not fix the number of particles. The appropriate many-body Hilbert space is thus the \textit{bosonic Fock space}
\begin{equation}
    \mathfrak{F}(\mathfrak{H}) = \mathbb{C} \oplus \bigoplus_{n=1}^\infty \mathfrak{H}^{\otimes_s n} \ , \\
\end{equation}
where $\mathfrak{H}^{\otimes_s n}$ is the $n-$fold symmetric tensor product. An operator $A_k$ acting on the $k$-particles sector $\mathfrak{H}^{\otimes_s k}$ can be lifted to an operator on $\mathfrak{F}(\mathfrak{H})$ as follows :
\begin{equation}
    \mathbb{A}_k= 0 \oplus \dots \oplus 0 \bigoplus_{n=k}^\infty \left( \sum_{1 \leq i_1 < \dots < i_k \leq n} (A_k)_{i_1, \dots, i_k}\right) \ , \\
\end{equation}
which we call the \textit{second quantization} of $A_k$. When $k=1$, the second quantization of a one-body operator $h$ is denoted by $\dgamma{h}$. 

\medskip

\noindent\textbf{States, density matrices.} We define the set of \textit{states} on the Fock space as follows
\begin{equation}
    \mathcal{S}(\mathfrak{F}(\mathfrak{H})) = \left\{ \Gamma \textnormal{ self-adjoint operator on } \mathfrak{F}(\mathfrak{H}) \ \vert \ \Gamma \geq 0 , \ \textnormal{Tr}_{\mathfrak{F}(\mathfrak{H})} \left[ \Gamma \right] = 1 \right\} \ . \\
\end{equation}
We will be interested in a particular class of states, called \textit{quantum Gibbs states}, that are of the form
\begin{equation}
    \Gamma = \frac{e^{-\mathbb{H}}}{\textnormal{Tr}_{\mathfrak{F}(\mathfrak{H})} \left[ e^{-\mathbb{H}} \right]} \ , \\
\end{equation}
where $\mathbb{H}$ is a self-adjoint operator on the Fock space such that $\textnormal{Tr}_{\mathfrak{F}(\mathfrak{H})} \left[ e^{-\mathbb{H}} \right] < \infty$. We will also make use of the $k$-body reduced density matrix $\Gamma^{(k)}$ of a state $\Gamma \in \mathcal{S}(\mathfrak{F}(\mathfrak{H}))$ which is defined by duality by setting 
\begin{equation}
    \textnormal{Tr}_{\mathfrak{H}^{\otimes_s k}} \left[A_k \Gamma^{(k)}\right] = \textnormal{Tr}_{\mathfrak{F}(\mathfrak{H})} \left[\mathbb{A}_k \Gamma\right]
\end{equation}
for every bounded operator $A_k$ on $\mathfrak{H}^{\otimes_s k}$. If $\Gamma$ is a diagonal state, i.e
\begin{equation*}
    \Gamma = \bigoplus_{n=0}^\infty \Gamma_n \ , \\
\end{equation*}
we have the equivalent formulation
\begin{equation*}
    \Gamma^{(k)} = \sum_{n\geq k} {n \choose k} \textnormal{Tr}_{k+1 \rightarrow n} \left[ \Gamma_n\right] \ , \\
\end{equation*}
where the partial trace notation $\textnormal{Tr}_{k+1 \rightarrow n}$ corresponds to tracing out $n-k$ variables. \\

\medskip

\noindent\textbf{Creation/annihilation operators.} Consider any $f \in \mathfrak{H}$. We define respectively $a^\dagger(f)$ and $a(f)$ as operators on the Fock space, whose actions on any $k$-particle sector, $k \geq 1$ is given by
\begin{equation}
    \begin{split}
        a^\dagger(f) \ &\ \mathfrak{H}^{\otimes_s k} \longrightarrow \mathfrak{H}^{\otimes_s k+1} \\
        a(f) \ &\ \mathfrak{H}^{\otimes_s k} \longrightarrow \mathfrak{H}^{\otimes_s k-1}
    \end{split}
\end{equation}
with
\begin{equation}
    \begin{split}
        a^\dagger(f) u_1 \otimes_s \dots \otimes_s u_k & = \sqrt{k+1}f \otimes_su_1 \otimes_s \dots \otimes_s u_k \\
        a(f) u_1 \otimes_s \dots \otimes_s u_k & = \frac{1}{\sqrt{k}} \sum_{j=1}^{k} \ps{f}{u_j}u_1 \otimes_s \dots u_{j-1} \otimes_s u_{j+1} \otimes_s \dots \otimes_s u_k \ . \\
    \end{split}
\end{equation}
Here, $\otimes_s$ stands for the symmetric tensor product of vectors. Note that $a(f)$ and $a^\dagger(f)$ are formal adjoints. They satisfy the Canonical Commutation Relations (CCR) : for all $f, g \in \mathfrak{H}$,
\begin{equation} \label{CCR}
    \left[a(f),a(g) \right] = \left[a^\dagger(f),a^\dagger(g) \right] = 0 \ , \quad \left[a(f),a^\dagger(g) \right] = \ps{f}{g} \ . \\
\end{equation}
Moreover, we can define the operator-valued distributions $a_x$ and $a^\dagger_x$ with
\begin{equation}
    a(f) = \int \overline{f(x)}a_x \textnormal{d}x \ , \quad a^\dagger(f) = \int f(x) a^\dagger_x \textnormal{d}x \ . \\
\end{equation}
The creation and annihilation operators allow to rewrite the second quantization of any $k$-body operator with a specified orthonormal basis for $\mathfrak{H}$. For example, if $h$ is a self-adjoint one-body operator on $\mathfrak{H}$, and $\left\{u_n \right\}_{n\ge 1}$ is an orthonormal basis for $\mathfrak{H}$, we have 
\begin{equation}
    \dgamma{h} = \sum_{m,n \geq 1} \ps{u_m}{hu_n}a^\dagger_m a_n \ , \\
\end{equation}
where $a^\dagger_m = a^\dagger(u_m)$, $a_n = a(u_n)$. 

\medskip

\noindent \textbf{Quantum Hamiltonian}. Following \cite{LewNamRou-20}, we first define our Hamiltonian with a chemical potential, that we will fix later in order to obtain a renormalization related to that of the classical theory. We consider 
$$h = - \Delta + 1$$
acting on $\mathfrak{H}$, and define
\begin{equation} \label{eqHamiltonian def with CP}
    \boxed{\mathbb{H}_\lambda^\varepsilon = \dgamma{h} + \lambda \mathbb{W_\varepsilon} - \nu(\lambda,\varepsilon)\mathcal{N} + E_0(\lambda,\varepsilon)}
\end{equation}
where
\begin{equation} \label{eqnon renormalized quantum interaction}
    \begin{split}
        \mathbb{W}_\varepsilon &= 0 \oplus0 \oplus \bigoplus_{n=2}^\infty \left( \sum_{1 \leq j < \ell \leq n} w^\varepsilon(x_j-x_\ell) \right) \\
        & = \frac{1}{2} \sum_{k \in 2\pi\mathbb{Z}^2} \widehat{w^\varepsilon}(k) \left \lvert \dgamma{e^{ik \cdot x}} \right \lvert^2 - \frac{w^\varepsilon(0)}{2}\mathcal{}
    \end{split}
\end{equation}
is the second quantization of the multiplication operator by $w^\varepsilon(x-y)$ acting on $\mathfrak{H}^{\otimes_s2}$. Here, $e^{ik \cdot x}$ is identified with the corresponding multiplication operator on $\mathfrak{H}$. Moreover, $\nu(\lambda,\varepsilon)$ is the \textit{chemical potential}, and $E_0(\lambda,\varepsilon)$ is a constant energy shift. Both have to be appropriately defined in parallel with the renormalized classical theory. 

\medskip

\noindent \textbf{Quantum Gibbs state.} This is the unique minimizer of the free-energy functional associated with~\eqref{eqHamiltonian def with CP} :
\begin{equation*}
    \mathcal{F}_\lambda : \Gamma \in \mathcal{S}(\mathfrak{F}(\mathfrak{H})) \mapsto \tr{\mathbb{H}^\varepsilon_\lambda \Gamma} + \lambda^{-1} \tr{\Gamma \log \Gamma} \ . \\
\end{equation*}
Here, the temperature $T$ is given by 
\begin{equation}
    T = \frac{1}{\lambda}
\end{equation}
The Gibbs state $\Gamma_\lambda$ is explicitely given by 
\begin{equation}
    \boxed{\Gamma_\lambda = \frac{1}{Z(\lambda)}e^{-\lambda\mathbb{H}^\varepsilon_\lambda} \ , \quad Z(\lambda) = \textnormal{Tr}_{\mathfrak{F}(\mathfrak{H})} \left[e^{-\lambda\mathbb{H}^\varepsilon_\lambda}\right]} \ . \\
\end{equation}
It can be checked that we indeed have $\Gamma_\lambda \in \mathcal{S}(\mathfrak{F}(\mathfrak{H}))$. 

\medskip

\noindent \textbf{Classical model}. In order to appropriately define what will be the limiting measure in our result, we first recall the definition of the free Gibbs measure, and refer to \cite[Sections~2 and~5]{LewNamRou-20} for details. We first write the spectral decomposition of $h = -\Delta + 1$ :
\begin{equation}
    h = \sum_{k \in 2\pi\mathbb{Z}^2} \lambda_k \ket{e_k}\bra{e_k} \ , \\
\end{equation}
where 
\begin{equation}
    \lambda_k = |k|^2 + 1 \ , \quad e_k = e^{ik \cdot x} \ , \quad k \in 2\pi\mathbb{Z}^2 \ , \\
\end{equation}
and define the Sobolev-type spaces
\begin{equation} \label{eq : eigenset of the laplacian}
    \mathfrak{H}^s= \left\{u = \sum_k \alpha_k e_k \ , \quad \|u\|_{\mathfrak{H}^s}^2 = \sum_k \lambda_k^s |\alpha_k|^2 < \infty \right\} \ . \\
\end{equation}
Then, the free Gibbs measure $\mu_0$ associated to $h$ is the unique probability (see \cite[Section~3]{LewNamRou-15}), such that for all $K \geq 1$, its cylindrical projection on $V_K = \textnormal{span} \left\{e_1 \dots, e_K \right\}$ is given by 
\begin{equation}
    \textnormal{d}\mu_{0,K}(u) = \prod_{k=1}^K \left( \frac{\lambda_k}{\pi}e^{-\lambda_k |\alpha_k|^2} \textnormal{d}\alpha_k\right) \ , \\
\end{equation}
where $\alpha_k = \ps{e_k}{u}$, and $\textnormal{d}\alpha_k$ is the Lebesgue measure on $\mathbb{C} \simeq \mathbb{R}^2$. Note that the measure $\mu_0$ lives on $\mathfrak{H}^{-\delta}$ for any $\delta > 0$, but not $\delta=0$. Hence it fails to charge $L^2\left(\mathbb{T}^2\right)$. This is precisely why we need to renormalize the mass in the interaction terms, in order to properly define a nonlinear Gibbs measure. To implement such a renormalization, we use \textit{Wick ordering}, as presented in~\cite[Appendix A]{FroKnoSchSoh-22}, and~\cite[Lemma 5.2]{LewNamRou-20}. This is the ``usual'' renormalization procedure that has been implemented in the context of Constructive Quantum Field Theory, with ideas originated by Edward Nelson, see \cite{Nelson-66}. Consider $P_K$ the orthogonal projection on $V_K$, and define
\begin{equation} \label{eqinteraction term in classical model}
    \begin{split}
        W^\varepsilon_K[u] & = \frac{1}{2}\iint_{\mathbb{T}^2 \times \mathbb{T}^2} w^\varepsilon(x-y) \wick{|P_Ku(x)|^2} \wick{|P_Ku(y)|^2}\textnormal{d}x\textnormal{d}y- \tau^\varepsilon \int_{\mathbb{T}^2}\wick{|P_Ku(x)|^2} \textnormal{d}x - E^\varepsilon
    \end{split}
\end{equation}
where $\tau^\varepsilon, E^\varepsilon > 0$ are constants that will be defined below, and
\begin{equation}
    \wick{|P_Ku(x)|^2} = |P_Ku(x)|^2 - \left \langle |P_Ku(x)|^2 \right\rangle_{\mu_0}
\end{equation}
denotes the  Wick ordering with respect to the free measure $\mu_0$. The precise meaning of \eqref{eqinteraction term in classical model} will be explained below.  With arguments similar to those of~\cite[Lemma 5.3]{LewNamRou-20}, we can show that the sequence $(W^\varepsilon_K)_{K \geq 1}$ is a Cauchy sequence in $L^2(\textnormal{d}\mu_0)$, and thus converges strongly to a limit $W^\varepsilon$ in $L^2(\textnormal{d}\mu_0)$. This is also explained in \cite{FroKnoSchSoh-22}. This allows us to define the nonlinear Gibbs measure
\begin{equation} \label{eqnonlinear gibbs meas def in classical model}
    \boxed{\textnormal{d}\mu^\varepsilon(u) = \frac{1}{\mathcal{Z}^\varepsilon}e^{-W^\varepsilon[u]}\textnormal{d}\mu_0(u)} \ . \\
\end{equation}
With a similar procedure, one can define rigorously the renormalized $\Phi^4_2$ measure
\begin{equation} \label{eqphi42 meas def in classical model}
    \boxed{\textnormal{d}\mu(u) = \frac{1}{\mathcal{Z}}e^{-V[u]}\textnormal{d}\mu_0(u)} \ , \\
\end{equation}
where $V[u]$ is defined as the $L^2(\textnormal{d}\mu_0)$-limit of its truncated counterpart
\begin{equation}
    V_K[u] = \frac{1}{2} \int_{\mathbb{T}^2} \wick{|P_Ku(x)|^4} \textnormal{d}x \ , \\
\end{equation}
where 
\begin{equation}
    \wick{|P_Ku(x)|^4} = |P_Ku(x)|^4 - 4 \left \langle |P_Ku(x)|^2\right \rangle_{\mu_0}|P_Ku(x)|^2 + 2\left \langle |P_Ku(x)|^2\right \rangle_{\mu_0}^2 \ . \\
\end{equation}
For details on the computation of the Wick ordering for monomials, see for example \cite[Lemma A.1]{FroKnoSchSoh-22}. \\

\subsection{Quantum / classical correspondence}
Here we give some formal links between the quantum and the classical theory, that will motivate why the $\Phi^4_2$ measure appears naturally as the limit of a many-body quantum Gibbs state. When there is no need for a renormalization, the formal link between the nonlinear Gibbs measure \eqref{eq:nonlinear gibbs def} and the Gibbs state \eqref{eq:generic Gibbs state} is quite straightforward. When renormalizing the interaction term from the classical theory, the obtained counterterms have to be appropriately incorporated in the quantum theory. In \cite{LewNamRou-20}, the ``easiest'' renormalization implemented proceeds by replacing the quadratic terms $|u(x)|^2$ and $|u(y)|^2$ by their Wick ordered version with respect to the free measure $\mu_0$, which is formally defined as
\begin{equation} \label{eq: wick order def}
    \wick{|u(x)|^2} \ = \ |u(x)|^2 - \left \langle |u(x)|^2 \right \rangle_{\mu_0} \ . \\
\end{equation}
Hence, the obtained renormalized interaction is formally given by
\begin{equation} \label{eqformal interaction from LNR}
    \textnormal{Int}^{\textnormal{Ren}}(u) = \frac{1}{2} \iint_{\mathbb{T}^2 \times \mathbb{T}^2} w(x-y) \wick{|u(x)|^2}\wick{|u(y)|^2}\textnormal{d}x\textnormal{d}y \ . \\
\end{equation}
Every term in \eqref{eq: wick order def} is infinite $\mu_0$-almost surely, but their difference turns out to be well defined, which essentially means that the fluctuations of $|u(x)|^2$ around its (infinite) mean are finite. Hence \eqref{eqformal interaction from LNR} has to be understood as an appropriate limit of a truncated version, in an appropriate space, see \cite[Lemma 5.3]{LewNamRou-20} for details.

Here we consider a local version with $w^\varepsilon \rightarrow \delta_0$, the limiting measure will be the $\Phi^4_2$ measure \eqref{eq: phi42 def}, with interaction term
\begin{equation*}
    \frac{1}{2} \int_{\mathbb{T}^2} |u(x)|^4 \textnormal{d}x \ . \\
\end{equation*}
Consequently, the most natural renormalization would be
\begin{equation*}
    \frac{1}{2} \int_{\mathbb{T}^2} \wick{|u(x)|^4} \textnormal{d}x \ . \\
\end{equation*}
It would then be more natural to replace the renormalization from~\eqref{eqformal interaction from LNR} by
\begin{equation}
    \textnormal{Int}_\varepsilon^{\textnormal{Ren}}(u) = \frac{1}{2} \iint_{\mathbb{T}^2 \times \mathbb{T}^2} w^\varepsilon(x-y) \wick{|u(x)|^2|u(y)|^2}\textnormal{d}x\textnormal{d}y \ , \\
\end{equation}
again, as seen as the limit of an appropriately truncated term. This can be done in the same spirit as \cite[Lemma 5.3]{LewNamRou-20}. Now, in order to make a link with the quantum model, the quartic term $\wick{|u(x)|^2|u(y)|^2}$ can be expanded (see for example \cite[Appendix A]{FroKnoSchSoh-22}), and we can express the renormalized interaction as
\begin{equation} \label{eqint ren with extra term I}
    \begin{split} \textnormal{Int}_\varepsilon^{\textnormal{Ren}}(u) & = \frac{1}{2} \iint_{\mathbb{T}^2 \times \mathbb{T}^2} w^\varepsilon(x-y) \wick{|u(x)|^2} \wick{|u(y)|^2}\textnormal{d}x\textnormal{d}y - \tau^\varepsilon \int_{\mathbb{T}^2} \wick{|u(x)|^2} \textnormal{d}x - E^\varepsilon \\
    & \quad - \iint_{\mathbb{T}^2 \times \mathbb{T}^2} w^\varepsilon(x-y)G(x-y) \left(  \wick{\overline{u(x)}u(y)}  - \wick{|u(x)|^2} \right) \textnormal{d}x\textnormal{d}y \ . \\
    \end{split}
\end{equation}
Here $\tau^\varepsilon$ and $\ E^\varepsilon$ are given in terms of the Green function 
\begin{equation} \label{eqgreen function laplacian}
    G(x,y) = \left \langle u(x) \overline{u(y)} \right \rangle_{\mu_0} = \sum_{k} \lambda_k^{-1}e_k(x) \overline{e_k(y)} \\
\end{equation}
as follows
\begin{equation}\label{eq:renorm param}
    \begin{split}
        \tau^\varepsilon & = \int_{\mathbb{T}^2} w^\varepsilon(z)G(z) \textnormal{d}z \underset{\varepsilon \rightarrow 0}{\simeq} |\log \varepsilon| \\ 
        \ E^\varepsilon &= \int_{\mathbb{T}^2} w^\varepsilon(z)G(z)^2 \textnormal{d}z \underset{\varepsilon \rightarrow 0}{\simeq} |\log \varepsilon|^2 \ . \\
    \end{split}
\end{equation}
Note that since we are working on the torus, $G$ is translation invariant, i.e $G(x,y) = G(x-y)$. Now, following closely the formal quantum / classical correspondence from \cite{LewNamRou-20}, we would like to write down the corresponding interacting Hamiltonian for our system. Let us first note that the term
\begin{equation}\label{eq:neglect term}
    I^\varepsilon = \iint_{\mathbb{T}^2 \times \mathbb{T}^2} w^\varepsilon(x-y)G(x-y) \left(  \wick{\overline{u(x)}u(y)}  - \wick{|u(x)|^2} \right) \textnormal{d}x\textnormal{d}y 
\end{equation}
does not make sense in a physical setting, because one cannot find a quantum counterpart that can be added to the Hamiltonian by only altering the chemical potential. Moreover, some insights from the classical theory that can be found in \cite[Section 4]{FroKnoSchSoh-22}, tell us that 
\begin{equation} \label{eqI epsilon goes to zero}
    \left \lVert I^\varepsilon \right \lVert_{L^2(\textnormal{d}\mu_0)} \underset{\varepsilon \rightarrow 0}{\longrightarrow} 0 \ , \\
\end{equation}
which is intuitive since $w^\varepsilon \rightarrow \delta_0$. This implies that~\eqref{eq:neglect term} will not contribute in the limit $\varepsilon \to 0$. More precisely, the following holds
\begin{equation} \label{eqlimit epsilon classical interaction}
    \textnormal{Int}_\varepsilon^{\textnormal{Ren}}(u) \underset{\varepsilon \rightarrow 0^+}{\longrightarrow} \frac{1}{2} \int_{\mathbb{T}^2} \wick{|u(x)|^4} \textnormal{d}x
\end{equation}
strongly in $L^2(\textnormal{d}\mu_0)$. \\

With these renormalized quantities at hand, we can make a formal link between the two following (non renormalized) objects 
\begin{equation}
    \Gamma_\lambda = Z(\lambda)^{-1}e^{-\lambda \mathbb{H}} \quad \longleftrightarrow \quad \textnormal{``} \ \textnormal{d}\mu(u) = \mathcal{Z}^{-1}e^{-\mathcal{E}(u)} \textnormal{d}u \ \textnormal{''} \ . \\
\end{equation}
Here,
\begin{equation}
    \begin{split}
        \lambda \mathbb{H} & = \lambda \int_{\mathbb{T}^2}a^\dagger_xh_xa_x \textnormal{d}x + \frac{\lambda^2}{2}\iint_{\mathbb{T}^2 \times \mathbb{T}^2} a^\dagger_x a_xw^\varepsilon(x-y)a^\dagger_ya_y \textnormal{d}x \textnormal{d}y \\
        \mathcal{E}(u) & = \int_{\mathbb{T}^2} \overline{u(x)} (h u)(x) \textnormal{d}x + \frac{1}{2} \iint_{\mathbb{T}^2 \times \mathbb{T}^2} |u(x)|^2w^\varepsilon(x-y) |u(y)|^2\textnormal{d}x\textnormal{d}y \ , \\
    \end{split}
\end{equation}
the classical objects being formal. In this rewriting, we replace quantum fields $a^\dagger_x a_x$ by functions $|u(x)|^2$ in the interaction term of $\lambda \mathbb{H}$, in order to obtain our renormalized Hamiltonian. By looking at \eqref{eqint ren with extra term I} and taking \eqref{eqI epsilon goes to zero} into account, we would obtain
\begin{equation}
    \begin{split}
        \lambda\mathbb{H}_\lambda^\varepsilon & = \lambda \int_{\mathbb{T}^2}a^\dagger_xh_xa_x \textnormal{d}x + \frac{\lambda^2}{2} \iint_{\mathbb{T}^2 \times \mathbb{T}^2} \left(a^\dagger_x a_x - N_0 \right)w^\varepsilon(x-y) \left( a^\dagger_ya_y - N_0 \right) \textnormal{d}x \textnormal{d}y \\
        & \quad \quad - \lambda \tau^\varepsilon \int_{\mathbb{T}^2} \left(a^\dagger_x a_x - N_0 \right) \textnormal{d}x - E^\varepsilon \ , \\
    \end{split}
\end{equation}
where 
\begin{equation}
    N_0 = \int \tr{a^\dagger_x a_x \Gamma_0} \textnormal{d}x= \tr{\mathcal{N} \Gamma_0} = \sum_{k \in 2\pi\mathbb{Z}^2} \frac{1}{e^{\lambda(|k|^2+1)}-1} \ . \\
\end{equation}
Notice that, by translation invariance 
$$N_0(x,x)=\tr{a^\dagger_x a_x \Gamma_0} $$ 
is constant. Eventually, starting from \eqref{eqnon renormalized quantum interaction} and using the same computations as in \cite[Equation (2.31)]{LewNamRou-20} we find that the renormalized Hamiltonian must have the following form
\begin{equation} \label{eqHamiltonian with fourier}
    \boxed{\lambda \mathbb{H}^\varepsilon_\lambda = \lambda \dgamma{h} + \frac{\lambda^2}{2} \sum_{k \in 2\pi\mathbb{Z}^2} \widehat{w^\varepsilon}(k) \left \lvert \dgamma{e_k} - \left \langle \dgamma{e_k} \right \rangle_{\Gamma_0}\right \lvert ^2 - \lambda \tau^\varepsilon(\mathcal{N}-N_0) - E^\varepsilon}
\end{equation}
where $e_k$ is the multiplication operator by $e^{ik \cdot x}$. Comparing the above with the initial form \eqref{eqHamiltonian def with CP}, we find the following choices for the chemical potential $\nu(\lambda, \varepsilon)$ and energy reference $E(\lambda, \varepsilon)$ to be natural:
\begin{equation}\label{eq:phys param}
    \boxed{
    \begin{split}
    \nu(\lambda,\varepsilon) &= \lambda \widehat{w^\varepsilon}(0)\,N_0 - \lambda \frac{\widehat{w^\varepsilon}(0)}{2} + \tau^\varepsilon,\\
    E_0(\lambda,\varepsilon) &= \lambda\frac{\widehat{w^\varepsilon}(0)}{2}\,N_0^2 + \tau^\varepsilon N_0 - \lambda^{-1}E^\varepsilon \ . \\
    \end{split}}
\end{equation}
The rewriting \eqref{eqHamiltonian with fourier} is very useful, especially since assuming $\widehat{w} \geq 0$ will make the whole interacting term 
\begin{equation} \label{eqinteractin term}
    \mathbb{W}^{\textnormal{Ren}}_\varepsilon= \frac{1}{2}\sum_{k \in 2\pi\mathbb{Z}^2} \widehat{w^\varepsilon}(k) \left \lvert \dgamma{e_k} - \left \langle \dgamma{e_k} \right \rangle_{\Gamma_0}\right \lvert ^2
\end{equation}
positive. We will use this expression frequently throughout the proof, and it also motivates the use of correlation inequalities from~\cite{DeuNamNap-25,LewNamRou-20} for variance terms of the form 
\begin{equation}
    \left \langle \left \lvert A - \left \langle A \right \rangle_{\Gamma_0}\right \lvert ^2 \right \rangle_{\Gamma_\lambda} \ . \\
\end{equation}
Here we used the notation $\langle \cdot \rangle_{\Gamma_\lambda} = \textnormal{Tr} \left[\cdot \Gamma_\lambda \right]$, which we will adopt for the rest of the paper. When no confusion is possible, we will also write $\langle \cdot \rangle_{\lambda}$. \\

We are now ready to state our main result.

\subsection{Main result}

We recall the definitions or the free Gibbs state, or Gaussian quantum state,
\begin{equation} \label{eqfree gibbs state}
    \Gamma_0 = Z_0(\lambda)^{-1} e^{-\lambda \dgamma{h}} \ , \quad Z_0(\lambda) = \textnormal{Tr}_{\mathfrak{F}(\mathfrak{H})} \left[ e^{-\lambda \dgamma{h}} \right] \\
\end{equation}
and of the interacting Gibbs state
\begin{equation} \label{eqquantum gibbs state}
    \Gamma_\lambda = Z(\lambda)^{-1}e^{-\lambda \mathbb{H}^\varepsilon_\lambda} \ , \quad Z(\lambda) = \textnormal{Tr}_{\mathfrak{F}(\mathfrak{H})} \left[ e^{-\lambda \mathbb{H}^\varepsilon_\lambda} \right] \ . \\
\end{equation}
We set the interaction potential $w \in L^1(\mathbb{T}^2)$, such that
\begin{equation} \label{eqassumptions on w}
    \widehat{w}(0) = 1 \ , \ \ \widehat{w}(k) \geq 0 \ , \ \ \sum_{k \in 2\pi\mathbb{Z}^2}\widehat{w}(k)(1 + |k|^2) < + \infty \ . \\
\end{equation}
Moreover, for any $\varepsilon > 0$, we define
\begin{equation}
    w^\varepsilon = \varepsilon^{-2}w(\varepsilon^{-1}\cdot) \ , \\
\end{equation}
which is an approximation of the Dirac delta $\delta_0$ when $\varepsilon \rightarrow 0$. To consider it as a periodic function on $\mathbb{T}^2$, we identify it with its periodized version
\begin{equation} \label{eq : periodization w}
    w^\varepsilon(x) = \sum_{k \in 2\pi\mathbb{Z}^2}\widehat{w}(\varepsilon k)e^{ik \cdot x} \ . \\
\end{equation}
Assumptions \eqref{eqassumptions on w} and \eqref{eq : periodization w} are rather classical in the many-body literature, and include for example Bessel or Gaussian interactions. \\

 Our main result is as follows:

\begin{theorem}[\textbf{Convergence of the relative free-energy}]\label{thmConvergence of the relative free-energy}\mbox{}\\
    Let $\lambda > 0$ and $\eta \in (0,1/24)$. Set $\varepsilon = \lambda^\eta$. Consider the renormalized interaction of the $\Phi^4_2$ theory
    \begin{equation}
        V(u) = \frac{1}{2} \int_{\mathbb{T}^2} \wick{|u(x)|^4} \textnormal{d}x \ , \\
    \end{equation}
    which is well-defined as a $L^2(\textnormal{d}\mu_0)$-limit of its truncated counterpart. Define the classical partition function
    \begin{equation}
        \mathcal{Z} = \int e^{-V(u)} \textnormal{d}\mu_0(u) \ . \\
    \end{equation}
    Then, setting the parameters of~\eqref{eqHamiltonian with fourier} as in~\eqref{eq:renorm param}, we have as $\lambda \rightarrow 0^+$,
    \begin{equation} \label{eqrelative free energy convergence}
        \log \frac{Z(\lambda)}{Z_0(\lambda)} \underset{\lambda \rightarrow 0^+}{\longrightarrow} \log \mathcal{Z} \ . \\
    \end{equation}
\end{theorem}

This result is an extension of~\cite[Theorem 3.1]{LewNamRou-20} and parallel results in~\cite{FroKnoSchSoh-20}, which cover the case where $\varepsilon$ is fixed. We also extend~\cite[Theorem 2.1,~Equation (2.23)]{FroKnoSchSoh-22}, because the condition on $\varepsilon$ is relaxed. Indeed, as the authors explain in the remark following the statement of the theorem, their result is based on a functional integral representation of the interacting Gibbs state, thus leading them to control oscillatory integrals throughout their proof. As a resulting issue, they get exponentially diverging factors of the form $e^{(\tau^\varepsilon)^2/2 + E^\varepsilon}$, which, in view of the asymptotic behaviour of the relevant quantities, forces them to require the condition
\begin{equation*}
    \varepsilon \geq \exp(- (\log \lambda^{-1})^{1/2-c})
\end{equation*}
for some constant $c>0$. We can bypass this constraint because our approach here is variational and does not rely on functional integral representations. \\

To compare our scaling of $\varepsilon$ with the diluteness condition~\eqref{eqcondition epsilon} we note that the expected number of particles in the free state is given by
\begin{equation}
    \begin{split}
        \tr{\mathcal{N}\Gamma_0} & = \sum_{k \in 2\pi\mathbb{Z}^2} \frac{1}{e^{\lambda( |k|^2 + 1)} - 1} \\
        &= - \frac{1}{4 \pi} \lambda^{-1} \log \lambda + \mathcal{O}(\lambda^{-1}) 
    \end{split}
\end{equation}
see for example \cite[Lemma B.1]{LewNamRou-20}. Our choice of $\varepsilon$ hence essentially corresponds to $\varepsilon \ll N^{-1/24}$. A similar scaling is obtained in~\cite[Section~12]{NamZhuZhu-25}. \\

\begin{remark}[Convergence of states]
When deriving nonlinear Gibbs measures from many-body quantum mechanics, a typical additional result, which is a consequence of the convergence of the relative free energy, is the convergence of the reduced density matrices (see for example \cite{LewNamRou-20} and \cite{FroKnoSchSoh-20}, or any related work from the two groups). In our setting, if one wants to straightforwardly adapt \cite[Lemma 11.3]{LewNamRou-20}, the condition $\varepsilon = \lambda^\eta$ raises issues because of some a priori bounds which lead to exponentially diverging factors. In \cite{NamZhuZhu-25}, the authors manage to bypass this issue, the main expense being that the convergence is obtained in a weaker topology, namely a weak convergence against $L^2$ functions with compactly supported Fourier transform rather than a strong convergence in Hilbert-Schmidt norm. It could be interesting to try to obtain this stronger convergence, and generalize it to the more challenging three dimensional case. \\
\end{remark}

\subsection{Variational proof strategy}
Let us now explain the main steps of our proof. As in~\cite{LewNamRou-15,LewNamRou-20,NamZhuZhu-25}, the whole argument is based on quantum and classical formulations of the \textit{Gibbs variational principle}. The Gibbs state $\Gamma_\lambda$ is indeed the unique minimizer for the variational problem 
\begin{equation} \label{eqvariational problem 1}
    - \log Z(\lambda) = \inf_{\Gamma \in \mathcal{S}(\mathfrak{F}(\mathfrak{H}))} \left\{ \lambda \text{Tr}_{\mathfrak{F}(\mathfrak{H})} \left[\mathbb{H}^\varepsilon_\lambda \Gamma \right] + \text{Tr}_{\mathfrak{F}(\mathfrak{H})} \left[\Gamma \log \Gamma\right] \right\} \ . \\
\end{equation}
Using a similar formulation for $\Gamma_0$  we also have~\cite[Equation (1.9)]{LewNamRou-15} that $\Gamma_\lambda$ is the unique minimizer to the relative variational problem
\begin{equation} \label{eqvariational problem 2}
    - \log \frac{Z(\lambda)}{Z_0(\lambda)} = \inf_{\Gamma \in \mathcal{S}(\mathfrak{F}(\mathfrak{H}))} \left\{ \mathcal{H}(\Gamma, \Gamma_0) + \lambda^2 \text{Tr}_{\mathfrak{F}(\mathfrak{H})} \left[ \mathbb{W}^\text{Ren}_\varepsilon \Gamma\right] - \lambda \tau^\varepsilon \text{Tr}_{\mathfrak{F}(\mathfrak{H})} \left[ \left( \mathcal{N} - N_0\right) \Gamma\right] - E^\varepsilon \right\} \ , \\
\end{equation}
where 
\begin{equation}
    \mathcal{H}(\Gamma, \Gamma') = \tr{\Gamma \left( \log \Gamma - \log \Gamma'\right)} \geq 0 \\
\end{equation}
is the von Neumann relative entropy. Both \eqref{eqvariational problem 1} and \eqref{eqvariational problem 2} will be useful for us, depending on whether we are trying to establish a lower or an upper bound on the relative free energy. \\

 On the classical side, the nonlinear Gibbs measure $\mu^\varepsilon$ defined in \eqref{eqnonlinear gibbs meas def in classical model} is the unique minimizer to the following variational problem
\begin{equation} \label{eqclassical variational problem 1}
    - \log \mathcal{Z}^\varepsilon = \inf_{\nu \ \ll \ \mu_0} \left\{ \mathcal{H}_{\textnormal{cl}}(\nu, \mu_0) + \int W^\varepsilon[u] \textnormal{d}\nu(u) \right\}
\end{equation}
where $\mathcal{H}_{\textnormal{cl}}$ is the classical relative entropy between two probability measures
\begin{equation} \label{eqclassical relative entropy def}
    \mathcal{H}_{\textnormal{cl}}(\mu,\mu')= \int \frac{\textnormal{d}\mu}{\textnormal{d}\mu'}(u) \log \frac{\textnormal{d}\mu}{\textnormal{d}\mu'}(u) \textnormal{d}\mu'(u) \ . \\
\end{equation}
We will be using \eqref{eqvariational problem 2} and \eqref{eqclassical variational problem 1} as crucial arguments in our proofs, as well as some finite dimensional analogues. \\

In order to prove \eqref{eqrelative free energy convergence}, we derive an upper bound and a lower bound to the relative free energy. Both will have as starting point the quantum variational problem \eqref{eqvariational problem 2}. More precisely, everything will be done in a finite dimensional setting by introducing a frequency cutoff $\Lambda > 0$, which will be taken to infinity afterwards. The lower bound is proved in Proposition \ref{propfree energy lower bound}, and the upper bound in Proposition \ref{propFree energy upper bound}. Combining these results, we will obtain a projected version of \eqref{eqrelative free energy convergence}, and Proposition \ref{propConvergence regularized renormalized partition function} will allow us to conclude.

\section{A priori estimates for the relative free energy}

We start by deriving some useful estimates on the relative free energy :

\begin{lemma}[\textbf{A-priori estimates for the relative free energy}]\label{lem:apriori}\mbox{}\\
Let $\lambda >0$, and $\varepsilon = \lambda^\eta$ for some $\eta \in (0,1/2)$. Then we have
\begin{equation} \label{eqa priori estimate relative free energy}
    \left \lvert  \log \frac{Z(\lambda)}{Z_0(\lambda)} \right \lvert \lesssim |\log \varepsilon|^2 \ . \\
\end{equation}
In particular, the Gibbs state $\Gamma_\lambda$ satisfies
\begin{equation}
   \mathcal{H}(\Gamma_\lambda, \Gamma_0) +\lambda^2  \tr{\left( \mathcal{N}-N_0\right)^2\Gamma_\lambda} \lesssim |\log\varepsilon|^2 \ . \\
\end{equation}
Moreover, the reduced density matrices $\Gamma_\lambda^{(1)}$ and $\Gamma_0^{(1)}$ satisfy 
\begin{equation} \label{eqa priori bound on RDM}
    \lambda \left \lVert \sqrt{h} \left(\Gamma_\lambda^{(1)} - \Gamma_0^{(1)} \right) \sqrt{h} \right \lVert_{\textnormal{HS}} \lesssim |\log \varepsilon|^2 \ . \\
\end{equation}
\end{lemma}
\begin{proof}
    Taking $\Gamma_0$ as a trial state in \eqref{eqvariational problem 2}, we find 
\begin{equation} \label{eq:vppletrialgamma0}
    \begin{split}
        - \log \frac{Z(\lambda)}{Z_0(\lambda)} & \leq \lambda^2 \text{Tr}_{\mathfrak{F}(\mathfrak{H})} \left[ \mathbb{W}^\text{Ren}_\varepsilon \Gamma_0\right] - E^\varepsilon \\
        & = \frac{\lambda^2}{2} \sum_{k\neq0} \widehat{w^\varepsilon}(k) \tr{|\dgamma{e_k}|^2 \Gamma_0} + \frac{\lambda^2}{2} \tr{\left( \mathcal{N}-N_0\right)^2\Gamma_0} - E^\varepsilon \ . \\
    \end{split}
\end{equation}
By \cite[Lemma~5.13]{LewNamRou-20}, we can bound the second term of the right hand side as follows
\begin{equation}
    \lambda^2\tr{\left( \mathcal{N}-N_0\right)^2\Gamma_0} \lesssim \tr{h^{-2}} \ . \\
\end{equation}
which is just a finite constant in our case. In order to deal with the first term in~\eqref{eq:vppletrialgamma0}, we follow a variance computation from~\cite[Lemma~5.11]{LewNamRou-20}. We write
\begin{equation} \label{eq:dgammaAca}
    \dgamma{A} = \sum_{m,n \geq 1}\ps{u_m}{Au_n}a^\dagger_ma_n \ . \\
\end{equation}
for any one body operator $A$ on $\mathfrak{H}$, 
where $\left\{ u_n \right\}_n$ is an orthonormal basis for $\mathfrak{H}$, and $a^\dagger_n = a^\dagger(u_n)$, $a_n = a(u_n)$ the corresponding creation and annihilation operators. Then,
\begin{equation}
    \begin{split}
        \tr{|\dgamma{e_k}|^2 \Gamma_0} & = \sum_{m,n,p,q \geq 1} \overline{\ps{u_m}{e_ku_n}}\ps{u_p}{e_ku_q} \tr{\left( a^\dagger_ma_n\right)^*a^\dagger_pa_q  \Gamma_0} \\
        & = \sum_{m,n,p,q \geq 1} \overline{\ps{u_m}{e_ku_n}}\ps{u_p}{e_ku_q} \left \langle a^\dagger_na_ma^\dagger_pa_q \right \rangle_{0} \ . \\
    \end{split}
\end{equation}
Using the CCR \eqref{CCR} and Wick's theorem (see \cite[Section 5.4]{namMQMII2020}), we have the following expression
\begin{equation}
    \left \langle a^\dagger_na_ma^\dagger_pa_q \right \rangle_{0} = \left \langle a^\dagger_na_m\right \rangle_{0}\left \langle a^\dagger_pa_q \right \rangle_{0} + \left \langle a^\dagger_na_q \right \rangle_{0} \delta_{m,p} + \left \langle a^\dagger_na_q\right \rangle_{0}\left \langle a^\dagger_pa_m \right \rangle_{0} \ . \\
\end{equation}
Hence
\begin{equation}
    \begin{split}
        \tr{|\dgamma{e_k}|^2 \Gamma_0} & = \left \langle \sum_{m,n \geq 1} \overline{\ps{u_m}{e_ku_n}} \left(a^\dagger_ma_n\right)^* \right \rangle_0 \left \langle \sum_{p,q \geq 1} \ps{u_p}{e_ku_q} a^\dagger_pa_q \right \rangle_0\\
        & \quad + \sum_{n,q \geq 1} \sum_{m\geq 1} \ps{e_k u_n}{u_m} \ps{u_m}{e_ku_q}\left \langle a^\dagger_n a_q \right \rangle_0 \\
        & \quad + \sum_{m,n,p,q \geq 1} \ps{e_k u_n}{u_m} \ps{u_p}{e_ku_q}\left \langle a^\dagger_n a_q \right \rangle_0 \left \langle a^\dagger_p a_m \right \rangle_0 \\
        & = \left \lvert \left \langle \dgamma{e_k} \right \rangle_0 \right \lvert^{2} + \tr{|e_k|^2 \Gamma_{0}^{(1)}} + \tr{e_k^* \Gamma_{0}^{(1)} e_k \Gamma_{0}^{(1)}} \ . \\
    \end{split}
\end{equation}
We used for the second term the fact that $\ps{e_k u_n}{u_m} \ps{u_m}{e_ku_q} = \bra{e_k u_n} (\ket{u_m} \bra{u_m}) \ket{e_ku_q}$, and the resolution of the identity $\sum_{m \geq 1}\ket{u_m} \bra{u_m} = \mathbf{1}_{\mathfrak{H}}$, and for the third term the fact that reduced density matrices can be defined with creation and annihilation operators~\cite[Equation~(6.1)]{LewNamRou-15}, giving in particular 
\begin{equation} \label{eq:reddensmatca}
    \ps{u_m}{\Gamma_0^{(1)}u_q} = \tr{a^\dagger_qa_m\Gamma_0} \ . \\
\end{equation}
Now, using \eqref{eq:dgammaAca}, \eqref{eq:reddensmatca} and the fact that $\Gamma_0^{(1)} = \frac{1}{e^{\lambda h}-1} = \sum_{p \in 2\pi\mathbb{Z}^2}\frac{1}{e^{\lambda(|p|^2+1)}-1}\ket{u_p}\bra{u_p}$  we can compute directly the following 
\begin{equation} 
    \begin{split}
        \left \langle \dgamma{e_k} \right \rangle_0 & = \tr{\sum_{p \in 2\pi\mathbb{Z}^2}a^\dagger_q a_{q-k} \Gamma_0} \\
        & = \sum_{p \in 2\pi\mathbb{Z}^2} \ps{u_q}{\Gamma_0^{(1)}u_{q-k}} \\
        & =  \sum_{p,r \in 2\pi\mathbb{Z}^2} \frac{1}{e^{\lambda(|r|^2+1)}-1} \delta_{q,r} \delta_{r,q-k} \\
        & = 0 \ , \\
    \end{split}
\end{equation}
since the sum in \eqref{eq:vppletrialgamma0} is taken over $k \neq 0$. For the second term, we just have
\begin{equation}
    \tr{|e_k|^2 \Gamma_{0}^{(1)}} = \tr{\mathcal{N}\Gamma_0} \lesssim - \frac{\log \lambda}{\lambda} \ . \\
\end{equation}
Hence, using that $\sum_{k\neq0} \widehat{w^\varepsilon}(k) \lesssim \varepsilon^{-2}$, we obtain that
\begin{equation}
    \frac{\lambda^2}{2}\sum_{k\neq0} \widehat{w^\varepsilon}(k) \tr{|e_k|^2 \Gamma_{0}^{(1)}} \lesssim -\lambda \varepsilon^{-2}\log \lambda \ . \\
\end{equation}
Now, let us use the operator inequality
\begin{equation}
    \frac{1}{e^{\lambda h}-1} \leq \frac{1}{\lambda h}
\end{equation}
and the fact that for any operators $A,B,C$, we have 
\begin{equation}
    A \leq B \implies C^* AC \leq C^*BC
\end{equation}
as operators. By cyclicity of the trace, we obtain
\begin{equation}
    \begin{split}
        \lambda^2\tr{e_k^* \Gamma_{0}^{(1)} e_k \Gamma_{0}^{(1)}} & = \lambda^2 \tr{e_k^* \Gamma_{0}^{(1)} e_k \Gamma_{0}^{(1)}} \\
        & = \lambda^2 \tr{\sqrt{\Gamma_{0}^{(1)}}e_k^* \Gamma_{0}^{(1)} e_k \sqrt{\Gamma_{0}^{(1)}}} \\
        & \leq \lambda \tr{\sqrt{\Gamma_{0}^{(1)}}e_k^* h^{-1}e_k \sqrt{\Gamma_{0}^{(1)}}} \\
        & = \lambda \tr{h^{-1/2}e_k^* \Gamma_{0}^{(1)}e_k h^{-1/2}} \\
        & \leq \tr{e_k^* h^{-1} e_k h^{-1}} \\
        & = \iint_{\mathbb{T}^2 \times \mathbb{T}^2}e^{ik\cdot(x-y)} G(x-y)^2\text{d}x \text{d}y \ . \\
    \end{split}
\end{equation}
Here $G$ is the Green kernel of $h$, and we used the fact that it is translation invariant on the torus. Moreover we have that $\int_{\mathbb{T}^2}\text{d}x = (2\pi)^2 < \infty$, and thus
\begin{equation}
    \iint_{\mathbb{T}^2 \times \mathbb{T}^2}e^{ik\cdot(x-y)} G(x-y)^2\text{d}x \text{d}y \simeq \widehat{G^2}(k) \ . \\
\end{equation}
In addition we have that
\begin{equation} \label{eq:greenkernel2d}
    G(x,y) \underset{|x-y| \rightarrow 0}{\sim} - \frac{1}{2 \pi} \log |x-y| \ , \\
\end{equation}
see e.g. \cite[Lemma 5.4]{RouSer-14}, which here will reduce to $G(r)^2 \simeq (\log r)^2$. Hence we obtain 
\begin{equation}
    \begin{split}
        \frac{\lambda^2}{2} \sum_{k\neq0} \widehat{w^\varepsilon}(k)\tr{e_k^* \Gamma_{0}^{(1)} e_k \Gamma_{0}^{(1)}} &\lesssim \sum_{k\neq0} \widehat{w^\varepsilon}(k)\widehat{G^2}(k) \\
        & \leq \sum_{k\in 2\pi\mathbb{Z}^2} \widehat{w^\varepsilon}(k)\widehat{G^2}(k) \\
        & = \int_{\mathbb{T}^2}w^\varepsilon(x) G(x)^2\text{d}x \\
        & = \int_{\mathbb{T}^2}w(y) G(\varepsilon y)^2\text{d}y \\
        & \simeq |\log \varepsilon|^2 \ . \\
    \end{split}
\end{equation}
Here in the second line, we used the identity
\begin{equation}
\widehat{G^2}(0) = \sum_{k \in 2 \pi \mathbb{Z}^2} \frac{1}{(|k|^2+1)^2}
\end{equation}
which can be obtained from~\eqref{eqgreen function laplacian} for example. Together with the assumption $\widehat{w} \geq 0$, this means that $\widehat{w^\varepsilon}(0)\widehat{G^2}(0) \geq 0$. In the third line, we used Parseval's identity and \eqref{eq:greenkernel2d} in the last line. Combining these three estimates with the variational principle and the fact that $E^\varepsilon \simeq |\log \varepsilon|^2$, we find that
\begin{equation}
    \begin{split}
        - \log \frac{Z(\lambda)}{Z_0(\lambda)} & \lesssim -\lambda \varepsilon^{-2}\log \lambda + |\log \varepsilon|^2 \\
        & \lesssim |\log\varepsilon|^2 \quad \text{if and only if } \eta < 1/2 \ . \\
    \end{split}
\end{equation}
The other side of \eqref{eqa priori estimate relative free energy} is obtained by writing the variational principle \eqref{eqvariational problem 2} as 
\begin{equation}
    \begin{split}
        - \log \frac{Z(\lambda)}{Z_0(\lambda)} & = \mathcal{H}(\Gamma_\lambda, \Gamma_0) + \frac{\lambda^2}{2} \sum_{k\neq0} \widehat{w^\varepsilon}(k) \tr{|\dgamma{e_k}|^2 \Gamma_\lambda} + \frac{\lambda^2}{2} \tr{\left( \mathcal{N}-N_0\right)^2\Gamma_\lambda} \\
        & \quad \quad - \lambda \tau^\varepsilon \text{Tr} \left[ \left( \mathcal{N} - N_0\right) \Gamma_\lambda\right] - E^\varepsilon \\
        & \geq \frac{\lambda^2}{2} \tr{\left( \mathcal{N}-N_0\right)^2\Gamma_\lambda} - \lambda \tau^\varepsilon \text{Tr} \left[ \left( \mathcal{N} - N_0\right) \Gamma_\lambda\right] - E^\varepsilon \ . \\
    \end{split}
\end{equation}
Indeed, the Von-Neumann relative entropy is always positive, and $\widehat{w} \geq 0$ by assumption. Then, using Young's inequality, we have that, as operators,
\begin{equation}
    \lambda \tau^\varepsilon \left( \mathcal{N} - N_0\right) \leq \frac{(\tau^\varepsilon)^2}{2} + \frac{\lambda^2}{2} \left( \mathcal{N} - N_0\right)^2 \ . \\
\end{equation}
Hence 
\begin{equation}
    \begin{split}
        \frac{\lambda^2}{2} \tr{\left( \mathcal{N}-N_0\right)^2\Gamma_\lambda} - \lambda \tau^\varepsilon \text{Tr} \left[ \left( \mathcal{N} - N_0\right) \Gamma_\lambda\right] & \geq -\frac{(\tau^\varepsilon)^2}{2} \\
        & \gtrsim - |\log \varepsilon|^2
    \end{split}
\end{equation}
which gives the desired result~\eqref{eqa priori estimate relative free energy}. We also obtain 
\begin{equation}
    \mathcal{H}(\Gamma_\lambda, \Gamma_0) +\lambda^2 \tr{\left( \mathcal{N}-N_0\right)^2\Gamma_\lambda} \lesssim |\log\varepsilon|^2 \ . \\
\end{equation}
Furthermore, applying \cite[Theorem~6.1]{LewNamRou-20} to the operator $\lambda h$, we find
\begin{equation}
    \begin{split}
        \lambda^2 \tr{\left \lvert \sqrt{h} \left(\Gamma_\lambda^{(1)} - \Gamma_0^{(1)} \right) \sqrt{h} \right \lvert^2} & \leq 4 \mathcal{H}(\Gamma_\lambda, \Gamma_0) \left(\sqrt{2} + \sqrt{\mathcal{H}(\Gamma_\lambda, \Gamma_0)} \right)^2 \\
        & \lesssim |\log \varepsilon|^4 \ . \\
    \end{split}
\end{equation}
In the last line we used $\mathcal{H}(\Gamma_\lambda, \Gamma_0) \lesssim |\log \varepsilon|^2$ which comes from the a priori estimate. This allows us to conclude.
\end{proof}

\section{Correlation estimate for high momenta}
We now seek a quantitative estimate for the interaction term~\eqref{eqinteractin term}, or more precisely its projected version defined in~\eqref{eqvar of projected interaction term} below. We tackle each Fourier mode separately, and the sought-after bound thus takes the form of a variance estimate for a one-body operator. This is where considerations from~\cite[Section~7]{LewNamRou-20} come into play, to deduce such bounds from the a priori information~\eqref{eqa priori bound on RDM}. Our goal here is to keep track of the $\varepsilon$ dependence along the proof. We shall rely on substantial improvements on the method of~\cite[Section~7]{LewNamRou-20} which were recently provided in~\cite[Section~7]{DeuNamNap-25}, using Stahl's theorem~\cite{Stahl-13}. \\

As mentioned, we first localize the problem to finitely many dimensions, using a cutoff $\Lambda$, that will be chosen so that $\Lambda = \lambda^{-\nu}$, for some $\nu >0$. We define the orthogonal projector
\begin{equation}
P := \mathbf{1}_{h \leq \Lambda} \ . \\
\end{equation}
Note that there holds
\begin{equation}
    \begin{split}
        P = \sum_{k \in 2\pi\mathbb{Z}^2} \mathbf{1}_{|k|^2+1 \leq \Lambda} \ket{e_k}\bra{e_k} = \sum_{j=1}^K \ket{e_j}\bra{e_j} \ , \\
    \end{split}
\end{equation}
where $K = \text{Tr}[P]$. We now set $$e_k = e_k^+ + e_k^- \mbox{ with } e_k^- = P e_k P \ . $$
In order to estimate the interaction term, we would like to control the variance (the expectations are in the interacting Gibbs state)
\begin{equation} \label{eqvar of projected interaction term}
    \lambda^2 \left \langle {\left \lvert \text{d}\Gamma(e_k) - \left \langle \text{d}\Gamma(e_k) \right \rangle_0 \right \lvert}^2 \right \rangle_\lambda \ . \\
\end{equation}
Using the Cauchy-Schwarz and Young inequalities, we see that, for every $\delta > 0$,
\begin{multline} \label{eq:icsvarterms}
    \left \langle {\left \lvert \text{d}\Gamma(e_k^-) - \left \langle \text{d}\Gamma(e_k^-) \right \rangle_0 \right \lvert}^2 \right \rangle_\lambda \leq (1+\delta)\left \langle {\left \lvert \text{d}\Gamma(e_k) - \left \langle \text{d}\Gamma(e_k) \right \rangle_0 \right \lvert}^2 \right \rangle_\lambda \\
    + (1+\delta^{-1})\left \langle {\left \lvert \text{d}\Gamma(e_k^+) - \left \langle \text{d}\Gamma(e_k^+) \right \rangle_0 \right \lvert}^2 \right \rangle_\lambda \ . \\
\end{multline}
Hence, we have to control the projected terms, that is those involving $e_k^+$ and $e_k^-$. The following result gives a control on the first term of the above inequality. 
\begin{theorem}[\textbf{Correlation estimate for high momenta}]\label{thm:corr}\mbox{}\\
    Let $\lambda >0$, $\varepsilon = \lambda^\eta$ and $\Lambda = \lambda^{-\nu}$ for some $\eta \in (0,1/2)$ and $\nu > 0$. We have
    \begin{equation} \label{eq:quantitativeesti in thm}
    \lambda^2 \left \langle {\left \lvert \textnormal{d}\Gamma(e_k^+) - \left \langle \textnormal{d}\Gamma(e_k^+) \right \rangle_0 \right \lvert}^2 \right \rangle_\lambda \lesssim |\log \varepsilon|^2 \Lambda^{-1/2+} - \lambda^2 \log \lambda (|k|^2 + |k| \Lambda^{1/2}) + \varepsilon^{-2} \lambda^2 |\log \lambda|^2 \ . \\
\end{equation}
\end{theorem}

\begin{proof}
    We use the new correlation inequality derived in \cite[Theorem 2]{DeuNamNap-25}. In order to apply it, we need to define a family of perturbed Gibbs states. Let $t\in [-1,1]$, and let
\begin{equation}
    \begin{split}
        \Gamma_{\lambda,t} & = Z_{\lambda,t}^{-1}e^{-\lambda \mathbb{H}^\varepsilon_\lambda + t \mathbb{B}} \quad , \quad Z_{\lambda,t} = \text{Tr}\left[ e^{-\lambda \mathbb{H}^\varepsilon_\lambda + t \mathbb{B}} \right] \\
        \Gamma_{0,t} & = Z_{0,t}^{-1}e^{-\lambda \text{d}\Gamma(h) + t \mathbb{B}} \quad , \quad Z_{0,t} = \text{Tr}\left[ e^{-\lambda \text{d}\Gamma(h) + t \mathbb{B}} \right] \ , \\
    \end{split}
\end{equation}
where 
\begin{equation}
    \mathbb{B} = \frac{\lambda}{4} \left( \text{d}\Gamma(f_k^+) - \left \langle \text{d}\Gamma(f_k^+) \right \rangle_0 \right) \ . \\
\end{equation}
Then, \cite[Theorem 2]{DeuNamNap-25} states that if
\begin{equation}
    \sup_{t \in [-1,1]} \left \lvert \text{Tr}\left[ \mathbb{B} \Gamma_{\lambda,t} \right] \right \lvert \leq a \\
\end{equation}
for some $a > 0$, then
\begin{equation}
    \text{Tr}\left[ \mathbb{B}^2 \Gamma_{\lambda,0} \right] \leq a e^a +  \text{Tr} \left[ \left[\mathbb{B}, \left[ \lambda \mathbb{H}^\varepsilon_\lambda, \mathbb{B} \right] \right] \Gamma_{\lambda,0} \right] \ . \\
\end{equation}
In fact, since $\Gamma_{\lambda,0} = \Gamma_\lambda$, we will exactly recover on the RHS the desired term. Note also that we have defined by linearity $f_k \in \left\{\cos(k\cdot x), \sin(k \cdot x) \right\}$ which allows us to consider only self-adjoint operators here, which will simplify things a bit. Now, using the fact that $\Gamma_{\lambda,t}$ and $\Gamma_{0,t}$ satisfy a similar variational principle as $\Gamma_\lambda$ does, we can obtain, from the a priori bound \eqref{eqa priori bound on RDM}, the following estimate  
\begin{equation}
    \lambda \left \lVert \sqrt{h} \left(\Gamma_{\lambda,t}^{(1)} -  \Gamma_{0,t}^{(1)} \right) \sqrt{h}\right \lVert_{\text{HS}} \lesssim |\log \varepsilon|^2 \ . \\
\end{equation}
Next, we compute
\begin{equation}\label{eq:compute}
    \begin{split}
        4\text{Tr} \left[ \mathbb{B} \Gamma_{\lambda,t} \right] & = \text{Tr} \left[ \lambda \left( \text{d}\Gamma(f_k^+) - \left \langle \text{d}\Gamma(f_k^+) \right \rangle_0 \right) \Gamma_{\lambda,t} \right] \\
        & = \lambda \text{Tr} \left[ \text{d}\Gamma(f_k^+) \Gamma_{\lambda,t}\right] - \lambda \text{Tr} \left[ \text{d}\Gamma(f_k^+) \Gamma_{0}\right] \\
        & = \lambda \text{Tr} \left[ f_k^+ \Gamma_{\lambda,t}^{(1)}\right] - \lambda \text{Tr} \left[ f_k^+ \Gamma_{0}^{(1)}\right] \\
        & = \lambda \text{Tr} \left[ f_k^+ \left( \Gamma_{\lambda,t}^{(1)} - \Gamma_{0,t}^{(1)} \right)\right] + \lambda \text{Tr} \left[ f_k^+ \left( \Gamma_{0,t}^{(1)} - \Gamma_{0}^{(1)} \right)\right] \ . \\
    \end{split}
\end{equation}
We start by controlling the second term. To do so, we use \cite[Lemma 6.3]{LewNamRou-20}, which states that for any self-adjoint operator $A$ satisfying $A \leq c h$ for some $c \in(0,1)$, we have 
\begin{equation}
    0 \leq \text{Tr} \left[ A \left( \frac{1}{e^{h-A}-1} - \frac{1}{e^h - 1}\right) \right] \leq \frac{1}{1-c} \text{Tr} \left[h^{-1}A h^{-1}A\right] \ . \\
\end{equation}
To apply this, notice that from the definition of $\mathbb{B}$, we can alternatively write
\begin{equation}
    \Gamma_{0,t} = {\Tilde{Z}_{0,t}}^{-1}e^{-\lambda \dgamma{h_t}} \ . \\
\end{equation}
Hence 
\begin{equation}
    \Gamma_{0,t}^{(1)} = \frac{1}{e^{\lambda h_t}-1} \ , \\
\end{equation}
where $h_t = h - \frac{t}{4}f_k^+$. Now let $A_t = \lambda \frac{t}{4} f_k^+$. We have $A_t \lesssim \lambda h$ as operators. Indeed $\|f_k^+\| \leq 1$, so that $\|h - h_t \| \leq \frac{1}{4}$, and thus $h \lesssim h_t \lesssim h$ as operators. Hence $A_t = \lambda(h-h_t) \leq C \lambda h$, $C \in (0,1)$ . Then, one can write
\begin{equation*}
    \begin{split}
        \lambda \text{Tr} \left[ f_k^+ \left( \Gamma_{0,t}^{(1)} - \Gamma_{0}^{(1)} \right)\right] & = \frac{4}{t} \text{Tr} \left[ A_t \left( \frac{1}{e^{\lambda h - A_t}-1} - \frac{1}{e^{\lambda h}-1} \right)\right] \\
        & \leq \frac{4}{t} \frac{1}{1-C} \text{Tr} \left[ (\lambda h)^{-1}A_t (\lambda h)^{-1}A_t \right] \\
        & \lesssim \text{Tr}\left[ h^{-1}f_k^+h^{-1}f_k^+ \right] \\
        & = \left \lVert h^{-1/2} f_k^+ h^{-1/2} \right \lVert_{\text{HS}}^2 \ . \\
    \end{split}
\end{equation*}
We used the cyclicity of the trace in the last line. For the first term in~\eqref{eq:compute}, we use again the cyclicity of the trace, and the Cauchy-Schwarz inequality as follows 
\begin{equation*}
    \begin{split}
        \text{Tr} \left[ f_k^+ \left( \Gamma_{\lambda,t}^{(1)} - \Gamma_{0,t}^{(1)} \right)\right] & = \text{Tr} \left[h^{-1/2} f_k^+ h^{-1/2} \sqrt{h}\left( \Gamma_{\lambda,t}^{(1)} - \Gamma_{0,t}^{(1)} \right)\sqrt{h}\right] \\
        & \leq \sqrt{\text{Tr} \left[ \left \lvert h^{-1/2} f_k^+ h^{-1/2} \right \lvert^2 \right]} \sqrt{\text{Tr} \left[\left \lvert \sqrt{h}\left( \Gamma_{\lambda,t}^{(1)} - \Gamma_{0,t}^{(1)} \right)\sqrt{h} \right \lvert^2\right]} \ . \\
    \end{split}
\end{equation*}
Combining these estimates, we find, for all $t \in [-1,1]$, 
\begin{equation}
    \begin{split}
        \left \lvert \text{Tr} \left[ \mathbb{B} \Gamma_{\lambda,t} \right] \right \lvert & \lesssim \lambda \left \lVert \sqrt{h}\left( \Gamma_{\lambda,t}^{(1)} - \Gamma_{0,t}^{(1)} \right)\sqrt{h} \right \lVert_{\text{HS}}\left \lVert h^{-1/2} f_k^+ h^{-1/2} \right \lVert_{\text{HS}} + \left \lVert h^{-1/2} f_k^+ h^{-1/2} \right \lVert_{\text{HS}}^2 \\
        & \lesssim | \log \varepsilon|^2\left \lVert h^{-1/2} f_k^+ h^{-1/2} \right \lVert_{\text{HS}} \ . \\
    \end{split}
\end{equation}
Here in the last line, we used the fact that 
\begin{equation}\label{eq:Schatten resol}
\mbox{with } h = -\Delta + 1, \mbox{ we have } \text{Tr} \left[h^{-p} \right] < + \infty \mbox{ for all } p > d/2= 1 \ . \\
\end{equation}
This implies
\begin{equation}
    \text{Tr} \left[ \frac{1}{h^{1+\delta}} \right] < + \infty \quad \forall \delta>0 \ . \\
\end{equation}
Hence, observing first that we can decompose
\begin{equation}
    f_k^+ = Pf_k Q + Qf_kP + Qf_kQ \\
\end{equation}
with $Q = \mathbf{1} - P$, we can write 
\begin{equation*}
    \begin{split}
        \left \lVert h^{-1/2} f_k^+ h^{-1/2} \right \lVert_{\text{HS}} & = \left \lVert f_k^+ h^{-1} \right \lVert_{\text{HS}} \\
        & \leq \left \lVert Pf_k Q h^{-1} \right \lVert_{\text{HS}} + \left \lVert Qf_k P h^{-1} \right \lVert_{\text{HS}} + \left \lVert Qf_k Q h^{-1} \right \lVert_{\text{HS}} \\
        & \leq \left \lVert Q h^{-1} \right \lVert_{\text{HS}} \ . \\
    \end{split}
\end{equation*}
Next,
\begin{equation*}
    \begin{split}
        \left \lVert Q h^{-1} \right \lVert_{\text{HS}}^2 & = \text{Tr} \left[ h^{-2} \mathbf{1}_{h > \Lambda} \right] \\
        & = \text{Tr} \left[ \frac{1}{h^{1+\delta}} \frac{1}{h^{1-\delta}} \mathbf{1}_{h > \Lambda} \right] \quad \forall \delta > 0 \\
        & \leq \frac{1}{\Lambda^{1-\delta}} \text{Tr}\left[ \frac{1}{h^{1+\delta}} \right] \ . \\
    \end{split}
\end{equation*}
Thus we find 
$$\left \lVert Q h^{-1} \right \lVert_{\text{HS}} \lesssim \Lambda^{-1/2+} \ .$$
On the one hand, this provides us the final bound for our uniform estimate on $\left \lvert \text{Tr} \left[ \mathbb{B} \Gamma_{\lambda,t} \right] \right \lvert$, and on the other hand, this shows that 
$$\left \lVert h^{-1/2} f_k^+ h^{-1/2} \right \lVert_{\text{HS}} \rightarrow 0 \ , $$
and that $\left \lVert h^{-1/2} f_k^+ h^{-1/2} \right \lVert_{\text{HS}}^2$ is of  higher order and can be neglected in our estimate. Finally, we have 
\begin{equation}
    \left \lvert \text{Tr} \left[ \mathbb{B} \Gamma_{\lambda,t} \right] \right \lvert \leq C |\log \varepsilon |^2 \Lambda^{-1/2+} =: a \\
\end{equation}
uniformly in $t \in [-1,1]$. Now remark that with our current choice for $\Lambda$ and $\varepsilon$, a short calculation gives us $a = C \eta^2 |\log \lambda| \lambda^{\nu/2 -}$ which goes to $0$ when $\lambda \rightarrow 0^+$. Hence, $ae^a \sim a$. We thus obtain the estimate
\begin{equation}
    \text{Tr} \left[ \mathbb{B}^2 \Gamma_\lambda \right] \lesssim |\log\varepsilon|^2\Lambda^{-1/2+} + \text{Tr}\left[ \left[\mathbb{B}, \left[ \lambda \mathbb{H}^\varepsilon_\lambda, \mathbb{B} \right] \right] \Gamma_{\lambda} \right] \ . \\
\end{equation}
We now have to analyse the commutator $\left[\mathbb{B}, \left[ \lambda \mathbb{H}^\varepsilon_\lambda, \mathbb{B} \right] \right]$. A computation shows that
\begin{equation} \label{eqcommutators}
    4\left[\mathbb{B}, \left[ \lambda \mathbb{H}^\varepsilon_\lambda, \mathbb{B} \right] \right] = \lambda^3 \left[\text{d}\Gamma(f_k^+), \left[ \text{d}\Gamma(h), \text{d}\Gamma(f_k^+) \right] \right] + \lambda^2\left[\mathbb{B}, \left[\widetilde{\mathbb{W}}_\varepsilon, \mathbb{B} \right] \right] \ , \\
\end{equation}
where
\begin{equation*}
    \widetilde{\mathbb{W}}_\varepsilon = \mathbb{W}^{\text{Ren}}_\varepsilon - \tau^\varepsilon\lambda^{-1}(\mathcal{N} - N_0) \ . \\
\end{equation*}
Let us deal with the first term in \eqref{eqcommutators}. Using that for any operators $X, Y$, 
$$\left[\text{d}\Gamma(X),\text{d}\Gamma(Y)\right] = \text{d}\Gamma\left(\left[X,Y\right]\right) \ ,$$
we first find that 
$$\left[\text{d}\Gamma(f_k^+), \left[ \text{d}\Gamma(h), \text{d}\Gamma(f_k^+) \right] \right] = \text{d}\Gamma\left(\left[f_k^+, \left[ h,f_k^+ \right] \right]\right) \ .$$
Then, an estimate from the proof of~\cite[Theorem 6.2]{NamZhuZhu-25}, which is independant of the dimension, shows that
\begin{equation*}
    \left[f_k^+, \left[ h,f_k^+ \right] \right] \lesssim |k|^2 + |k| \Lambda^{1/2} \ . \\
\end{equation*}
Hence, we get for the first term in~\eqref{eqcommutators}
\begin{equation}
    \lambda^3 \text{Tr} \left[ \left[\text{d}\Gamma(f_k^+), \left[ \text{d}\Gamma(h), \text{d}\Gamma(f_k^+) \right] \right] \Gamma_\lambda \right] \lesssim \lambda^3 (|k|^2 + |k| \Lambda^{1/2}) \text{Tr} \left[ \mathcal{N} \Gamma_\lambda \right] \ . \\
\end{equation}
In order to estimate $\text{Tr} \left[ \mathcal{N} \Gamma_\lambda \right]$, we use that, from the a priori bounds of Lemma~\ref{lem:apriori} we have 
\begin{equation}
    \lambda^2 \text{Tr}\left[ \left( \mathcal{N} - N_0\right)^2 \Gamma_\lambda \right] \lesssim |\log \varepsilon|^2 \ . \\
\end{equation}
Hence, by the Cauchy-Schwarz inequality, it holds that
\begin{equation*}
    \begin{split}
        \text{Tr}\left[ \left(\mathcal{N} - N_0\right) \Gamma_\lambda \right] & \leq \sqrt{\text{Tr}\left[ \left( \mathcal{N} - N_0\right)^2 \Gamma_\lambda \right]} \\
        & \lesssim |\log \varepsilon| \lambda^{-1} \ . \\
    \end{split}
\end{equation*}
Now, recall that in dimension $d=2$, we have in the limit $\lambda \rightarrow 0^+$ (see \cite{LewNamRou-20} for example)
\begin{equation}
    N_0 \lesssim - \frac{\log \lambda}{\lambda} \ . \\
\end{equation}
Thus, using our assumption on $\varepsilon$, 
\begin{equation*}
    \begin{split}
        \text{Tr} \left[ \mathcal{N} \Gamma_\lambda \right] & \lesssim \frac{|\log \varepsilon|}{\lambda} - \frac{\log \lambda}{\lambda} \lesssim - \frac{\log \lambda}{\lambda} \ . \\
    \end{split}
\end{equation*}
This leads to
\begin{equation}
    \lambda^3 \text{Tr} \left[ \left[\text{d}\Gamma(f_k^+), \left[ \text{d}\Gamma(h), \text{d}\Gamma(f_k^+) \right] \right] \Gamma_\lambda \right] \lesssim - \lambda^2 \log \lambda (|k|^2 + |k| \Lambda^{1/2}) \ . \\
\end{equation}
Let us now deal with the second term in~\eqref{eqcommutators} namely 
$$\lambda^2\left[\mathbb{B}, \left[\widetilde{\mathbb{W}}_\varepsilon, \mathbb{B} \right] \right] = \lambda^2\left[\mathbb{B}, \left[\mathbb{W}^{\text{Ren}}_\varepsilon, \mathbb{B} \right] \right] \ .$$
Since we have
\begin{equation}
    \mathbb{W}^{\text{Ren}}_\varepsilon = \frac{1}{2} \left(\mathcal{N}-N_0 \right)^2 + \frac{1}{2}\sum_{k\neq0}\widehat{w^\varepsilon}(k)\left \lvert \text{d}\Gamma(e_k) \right \lvert^2 \ ,
\end{equation}
we find 
\begin{equation}
    \begin{split}
        \lambda^2\left[\mathbb{B}, \left[\widetilde{\mathbb{W}}_\varepsilon, \mathbb{B} \right] \right] & = \frac{\lambda^4}{2} \sum_{\ell \neq 0}\widehat{w^\varepsilon}(\ell)\left[\text{d}\Gamma(f_k^+), \left[\left \lvert \text{d}\Gamma(e_\ell) \right \lvert^2, \text{d}\Gamma(f_k^+) \right] \right] \\
        & \lesssim \lambda^4 \sum_{\ell \neq 0} \widehat{w^\varepsilon}(\ell) \mathcal{N}^2 \\
        & \lesssim \lambda^4 \varepsilon^{-2}\mathcal{N}^2 \ . \\
    \end{split}
\end{equation}
Here in the last line we used Fourier inversion : 
\begin{equation*}
    w^\varepsilon(0) = \sum_{k \in 2 \pi \mathbb{Z}^2}\widehat{w^\varepsilon}(k) 
\end{equation*}
and the definition of $w^\varepsilon$. Now, using $\text{Tr}\left[\mathcal{N}^2 \Gamma_\lambda\right]| \lesssim \lambda^{-2}|\log\lambda|^2$, we find 
\begin{equation}
    \lambda^2 \text{Tr} \left[ \left[\mathbb{B}, \left[\widetilde{\mathbb{W}}_\varepsilon, \mathbb{B} \right] \right] \Gamma_\lambda \right] \lesssim \varepsilon^{-2} \lambda^2 |\log \lambda|^2 \ . \\
\end{equation}
Putting all this together, we finally arrive at
\begin{equation} \label{eq:quantitativeesti}
    \lambda^2 \left \langle {\left \lvert \text{d}\Gamma(e_k^+) - \left \langle \text{d}\Gamma(e_k^+) \right \rangle_0 \right \lvert}^2 \right \rangle_\lambda \lesssim |\log \varepsilon|^2 \Lambda^{-1/2+} - \lambda^2 \log \lambda (|k|^2 + |k| \Lambda^{1/2}) + \varepsilon^{-2} \lambda^2 |\log \lambda|^2 \\
\end{equation}
which is the desired result. \\
\end{proof}

\section{Convergence of the relative free energy}
Now that we have a quantitative estimate on the high momenta variance term, we can focus on the low momenta one, and apply some finite dimensional semiclassics that have already been used in \cite{LewNamRou-15}, \cite{LewNamRou-20} and \cite{NamZhuZhu-25}. 

\medskip

\noindent \textbf{Fock space localization}. 
We recall here the Fock space localization procedure for bosons, see e.g.~\cite{DerGer-13,Rougerie-LMU} for review. Consider $\mathfrak{H}_1$, $\mathfrak{H}_2$ two Hilbert spaces, and denote by $P_i$ the orthogonal projection of $\mathfrak{H}_1 \oplus \mathfrak{H}_2$ onto $\mathfrak{H}_i$, for $i = 1,2$. There exists a unitary
\begin{equation} \label{eq:unitary fock spaces}
    \mathcal{U}\mathfrak{F} \left( \mathfrak{H}_1 \oplus \mathfrak{H}_2 \right) \longrightarrow \mathfrak{F}(\mathfrak{H}_1) \otimes \mathfrak{F}(\mathfrak{H}_2)
\end{equation}
such that $\mathcal{U} \mathcal{U}^* = \mathbf{1}$, and 
\begin{equation} \label{eq:action of the unitary on creators and annihilators}
    \mathcal{U}a^{\sharp}(f) \mathcal{U}^* = a^\sharp(P_1f) \otimes \mathbf{1} + \mathbf{1} \otimes a^\sharp(P_2f) \quad, \quad \forall f \in \mathfrak{H}_1 \oplus \mathfrak{H}_2
\end{equation}
where $a^\sharp$ is either $a$ or $a^\dagger$. With this unitary $\mathcal{U}$ in hand, we can define the $P$-localization of a state. From now on, we will always consider $\mathfrak{H}_1 = P \mathfrak{H}$ and $\mathfrak{H}_2 = Q \mathfrak{H}$, so that $\mathfrak{H}_1 \oplus \mathfrak{H}_2 = \mathfrak{H}$, for a given Hilbert space $\mathfrak{H}$, and hence we have the corresponding factorization property of the Fock space
\begin{equation}
    \mathfrak{F}(\mathfrak{H}) \simeq \mathfrak{F}(P\mathfrak{H}) \otimes \mathfrak{F}(Q\mathfrak{H}) \ . \\
\end{equation}

\begin{definition}[\textbf{$P$-localization of a state}]\label{def:Ploc}\mbox{}\\
    Let $\mathfrak{H}$ be a separable complex Hilbert space, and let $\mathfrak{F}(\mathfrak{H})$ the associated bosonic Fock space. For any $\Gamma \in \mathcal{S}(\mathfrak{F}(\mathfrak{H}))$ and any orthogonal projector $P$, we define its \textnormal{$P-$localization} $\Gamma_P \in \mathcal{S}(\mathfrak{F}(P\mathfrak{H}))$ as the partial trace
    \begin{equation} \label{eq:P loc}
        \Gamma_P = \textnormal{Tr}_{\mathfrak{F}(Q\mathfrak{H})} \left[ \Gamma  \right] \ , \\
    \end{equation}
    i.e.
    \begin{equation} \label{eq: equivalent def P-loc}
        \textnormal{Tr}_{\mathfrak{F}(P\mathfrak{H})} \left[ A \Gamma_P \right] = \textnormal{Tr}_{\mathfrak{F}(\mathfrak{H})} \left[ A \otimes \mathbf{1}_{\mathfrak{F}(Q\mathfrak{H})} \Gamma \right] \\
    \end{equation}
    for any bounded  operator $A$ on $\mathfrak{F}(P\mathfrak{H})$. Moreover, the density matrices of $\Gamma_P$ can be shown to satisfy 
    \begin{equation}
        \Gamma_P^{(k)} = P^{\otimes k} \Gamma^{(k)} P^{\otimes k} \quad , \quad \forall k \geq 1 \ . \\
    \end{equation}
\end{definition}

\noindent Let us now recall some facts and results that we will use throughout the proof. First of all, let us give the definition of a \textit{lower symbol} at scale $\lambda >0$, using coherent states.

\begin{definition}[\textbf{Bosonic coherent states}]\label{def:coh}\mbox{}\\
    Let $\mathfrak{H}$ be a Hilbert space which is not necessarily finite dimensional. Define, for any $u \in \mathfrak{H}$ the \textnormal{Weyl operator} as
    \begin{equation}
        W(u) = \exp\left(a^\dagger(u) - a(u) \right) \ . \\
    \end{equation}
    Then, a \textnormal{coherent state} is a Weyl-rotation of the vacuum $\ket{0}$ in the Fock space $\mathfrak{F}(\mathfrak{H})$
    \begin{equation} \label{eq: coherent state}
    \begin{split}
        \xi(u) &:= W(u) \ket{0} \\
        & \ = e^{- \frac{\|u\|^2}{2}}e^{a^\dagger(u)}\ket{0} \\
        & \ = e^{- \frac{\|u\|^2}{2}} \bigoplus_{n=0}^\infty \frac{u^{\otimes n}}{\sqrt{n!}} \ . \\
    \end{split}
    \end{equation}
\end{definition}
\noindent The equivalence between the definitions can be seen through the CCR and the Baker-Campbell-Hausdorff formula (see e.g. \cite[Definition 3.11]{Rougerie-EMS} for more details). The $k$-particle density matrix of a coherent state is given by
\begin{equation}
    \ket{\xi(u)} \bra{\xi(u)}^{(k)} = \frac{1}{k!} \ket{u^{\otimes k}} \bra{u^{\otimes k}} \ . \\ 
\end{equation}

\begin{definition}[\textbf{Lower symbol/Husimi function}]\label{def:lowsymb}\mbox{}\\
    Let $\Gamma \in \mathcal{S}(\mathfrak{F}(\mathfrak{H}))$, and $\lambda > 0$. The lower symbol of $\Gamma$ on $P \mathfrak{H}$ at scale $\lambda$ is the \textnormal{probability measure} which is absolutely continuous with respect to the Lebesgue measure $\textnormal{d}u$ on $P \mathfrak{H} \simeq \mathbb{C}^{\tr{P}}$, defined by
    \begin{equation} \label{eq: lower symbol def}
        \textnormal{d}\mu_{P, \Gamma}^\lambda (u) = (\lambda \pi)^{-\tr{P}} \ps{\xi(u/ \sqrt{\lambda})}{\Gamma_P\xi(u/ \sqrt{\lambda})}_{\mathfrak{F}(P \mathfrak{H})} \textnormal{d}u \ . \\
    \end{equation}
\end{definition}

We mentioned that $\mu_{P, \Gamma}^\lambda$ is a probability measure. This is justified by the coherent state resolution of the identity that we recall in~\eqref{eq: cs partition unity} below. Now, we have an important result, that gives a link between the density matrices of the projected state, ie $\Gamma_P^{(k)}$, and the density matrices of the state
\begin{equation}
    \int_{P \mathfrak{H}} \ket{\xi(u/\sqrt{\lambda})} \bra{\xi(u/\sqrt{\lambda})} \textnormal{d}\mu_{P, \Gamma}^\lambda (u)
\end{equation}
which, as we will see, will naturally arise when considering a trial state for our upper bound. The result is the following

\begin{theorem}[\textbf{Lower symbols as de Finetti measures, \cite{LewNamRou-15}}] \label{thmLower symbols as de Finetti measures}\mbox{}\\
    Let $k \geq 1$. We have, on $\bigotimes_s^kP \mathfrak{H}$,
    \begin{equation}
        \int_{P \mathfrak{H}} \ket{u^{\otimes k}} \bra{u^{\otimes k}} \textnormal{d}\mu_{P, \Gamma}^\lambda (u) = k! \lambda^{k} \sum_{\ell = 0}^k \binom{k}{\ell} \Gamma_P^{(\ell)} \otimes_s \mathbf{1}_{\bigotimes_s^{k-\ell}P \mathfrak{H}} \ . \\
    \end{equation}
    In particular, we have the following estimate
    \begin{equation} \label{eq: lower symbols as de Finetti measures estimate}
        \left \lVert k! \lambda^k \Gamma_P^{(k)} -  \int_{P \mathfrak{H}} \ket{u^{\otimes k}} \bra{u^{\otimes k}} \textnormal{d}\mu_{P, \Gamma}^\lambda (u)\right \lVert_{\mathfrak{S}^1(\bigotimes_s^kP \mathfrak{H})} \leq \lambda^k \sum_{\ell = 0}^k {\binom{k}{\ell}}^2 \frac{(k - \ell + K - 1)!}{(K-1)!} \tr{\mathcal{N}^\ell \Gamma_P} \ , \\
    \end{equation}
    where $K = \tr{P}$. \\
\end{theorem}

Finally, we will connect quantum and classical entropies in our variational principle via the following 

\begin{theorem}[\textbf{Relative entropy : quantum to classical, \cite{LewNamRou-15}}]\label{thm:Relative entropyquantum to classical}\mbox{}\\
    Let $\Gamma, \Gamma' \in \mathcal{S}(\mathfrak{F}(\mathfrak{H}))$. Let $\mu_{P, \Gamma}^\lambda, \mu_{P, \Gamma'}^\lambda$ the associated lower symbols. Then
    \begin{equation}
        \mathcal{H}(\Gamma, \Gamma') \geq \mathcal{H}(\Gamma_P, \Gamma_P') \geq \mathcal{H}_{\textnormal{cl}}(\mu_{P, \Gamma}^\lambda, \mu_{P, \Gamma'}^\lambda) \ . \\
    \end{equation} \\
\end{theorem}

\subsection{Free energy lower bound}
We now have the necessary ingredients to prove the free energy lower bound:

\begin{proposition}[\textbf{Free energy lower bound}] \label{propfree energy lower bound}\mbox{}\\
    Let $\lambda > 0$, $\eta \in (0,1/24)$, and $\nu \in (8 \eta, 1/3)$. Let $\varepsilon = \lambda^\eta$, $\Lambda = \lambda^{-\nu}$. Let $P = \mathbf{1}_{h \leq \Lambda} = \sum_{k=1}^K \ket{e_k}\bra{e_k}$ where $K = \tr{P}$. Consider the partition functions of the free and interacting Gibbs states 
    \begin{equation*}
        Z(\lambda) = \textnormal{Tr}_{\mathfrak{F}(\mathfrak{H})} \left[e^{-\lambda\mathbb{H}^\varepsilon_\lambda} \right] \ , \quad Z_0(\lambda) = \textnormal{Tr}_{\mathfrak{F}(\mathfrak{H})} \left[e^{-\lambda \dgamma{h}} \right] \ , \\
    \end{equation*}
    where $\mathbb{H}^\varepsilon_\lambda$ is given by \eqref{eqHamiltonian with fourier}. Consider the renormalized regularized interaction 
    \begin{equation} \label{eq: regularized renormalized Int}
    W^\varepsilon_K[u]  = \frac{1}{2}\iint_{\mathbb{T}^2 \times \mathbb{T}^2} w^\varepsilon(x-y) \wick{|P_Ku(x)|^2} \wick{|P_Ku(y)|^2}\textnormal{d}x\textnormal{d}y- \tau^\varepsilon \int_{\mathbb{T}^2}\wick{|P_Ku(x)|^2} \textnormal{d}x - E^\varepsilon \ , \\
    \end{equation}
    and the cylindrical projection of the Gaussian measure $\mu_0$ on $P\mathfrak{H}$
    \begin{equation}
        \textnormal{d}\mu_{0,K}(u) = \frac{1}{z_{0,K}}e^{-\ps{u}{hu}} \textnormal{d}u \ . \\
    \end{equation}
    Then, we have
    \begin{equation} \label{eqlower bound in prop}
        -\log \frac{Z(\lambda)}{Z_0(\lambda)} \geq - \log \left( \int_{P \mathfrak{H}}e^{-W^\varepsilon_K[u]} \textnormal{d}\mu_{0,K}(u) \right) + o_{\lambda \rightarrow0^+}(1) \ . \\
    \end{equation}
\end{proposition}

\begin{proof}
    In view of the Cauchy-Schwarz inequality \eqref{eq:icsvarterms} and the previously derived correlation inequality \eqref{eq:quantitativeesti in thm}, we have to study the following term
\begin{equation}
    \left \langle {\left \lvert \text{d}\Gamma(e_k^-) - \left \langle \text{d}\Gamma(e_k^-) \right \rangle_0 \right \lvert}^2 \right \rangle_\lambda \ . \\
\end{equation}
To do so, we expand it and control separately each contribution, and we will see that the Wick-ordered interaction terms arise naturally, relying on Theorem~\ref{thmLower symbols as de Finetti measures}. To begin with, we expand 
\begin{equation} \label{eq:expanding the square in trace}
    \begin{split}
        \left \langle {\left \lvert \text{d}\Gamma(e_k^-) - \left \langle \text{d}\Gamma(e_k^-) \right \rangle_0 \right \lvert}^2 \right \rangle_\lambda & = \left \langle {\left \lvert \text{d}\Gamma(e_k^-)  \right \lvert}^2 \right \rangle_\lambda - 2 \Re \left\{ \overline{\left \langle \text{d}\Gamma(e_k^-)   \right \rangle_\lambda} \left \langle \text{d}\Gamma(e_k^-)   \right \rangle_0 \right\} + {\left \lvert \left \langle \text{d}\Gamma(e_k^-)   \right \rangle_0 \right \lvert}^2 \ . \\
    \end{split}
\end{equation}
Let us first deal with the term $\left \langle \text{d}\Gamma(e_k^-)   \right \rangle_0$. We have  
\begin{equation}
    \begin{split}
        \lambda \left \langle \text{d}\Gamma(e_k^-)   \right \rangle_0 & = \lambda\text{Tr} \left[ \text{d}\Gamma(e_k^-) \Gamma_0 \right] \\
        & = \lambda\text{Tr} \left[e_k^- \Gamma_0^{(1)} \right] \\
        & = \lambda\text{Tr} \left[e_k^- \frac{1}{e^{\lambda h} -1}\right] \\
        & = \tr{e_k^-h^{-1}} + \lambda \tr{e_k^-\left( \frac{1}{e^{\lambda h}-1} - \frac{1}{\lambda h}\right)} \\
        & \leq\tr{e_k^-h^{-1}} + \frac{\lambda}{2} \tr{Pe_kP} \ . \\
    \end{split}
\end{equation}
In the last line, we used the operator inequality 
\begin{equation}
    \left \lvert \frac{1}{e^{\lambda h}-1} - \frac{1}{\lambda h} \right \lvert \leq \frac{1}{2} \ . \\
\end{equation}
Now, using \cite[Lemma 5.1]{LewNamRou-20}, we have 
\begin{equation}
    h^{-1} = \int \ket{u}\bra{u} \text{d}\mu_0(u) \ ,
\end{equation}
and thus
\begin{equation}
    \begin{split}
        \tr{e_k^-h^{-1}} &= \int \tr{\ket{e_k^- u} \bra{u}} \text{d}\mu_0(u) \\ 
        & =\int \ps{u}{e_k^- u}\text{d}\mu_0(u) \ . \\
    \end{split}
\end{equation}
Furthermore, we have 
\begin{equation}
    \begin{split}
        \tr{Pe_kP} & \leq \tr{\mathbf{1}_{h\leq \Lambda}} \\
        & \leq \tr{\left( \frac{\Lambda}{h} \right)^{1+}} \\
        & \lesssim \Lambda^{1+} \ , \\
    \end{split}
\end{equation}
leading to
\begin{equation} \label{eqlambda mean dgamma e_k^-0}
    \lambda \left \langle \text{d}\Gamma(e_k^-)   \right \rangle_0 \leq \int \ps{u}{e_k^- u}\text{d}\mu_0(u) + C\lambda \Lambda^{1+}.
\end{equation}
Let us now deal with the term $\left \langle \text{d}\Gamma(e_k^-)   \right \rangle_\lambda$. We write :
\begin{equation} \label{eqlambda mean dgamma e_k^-}
    \lambda \left \langle \text{d}\Gamma(e_k^-)   \right \rangle_\lambda = \lambda \tr{e_k^- \left( \Gamma_\lambda^{(1)} - \Gamma_{0}^{(1)} \right)} + \lambda\tr{e_k^- \Gamma_{0}^{(1)}} \ . \\
\end{equation}
Another useful bound comes from \cite[Lemma 2.10]{LewNamRou-20} : for all $k \geq 1$,
\begin{equation}
    \Gamma_0^{(k)} = \left( \frac{1}{e^{\lambda h}-1} \right)^{\otimes k} \leq \lambda^{-k} \left(h^{-1} \right)^{\otimes k} \ . \\
\end{equation}
Here, we get 
\begin{equation}
    \begin{split}
        \lambda \tr{e_k^- \Gamma_{0}^{(1)}} & \leq \tr{Ph^{-1}} \\
        & = \tr{\mathbf{1}_{h \leq \Lambda}h^{-1}} \\
        & \leq \tr{\left( \frac{\Lambda}{h} \right)^{0+}h^{-1}} \\
        & \lesssim \Lambda^{0+} \ . \\
    \end{split}
\end{equation}
Now, with similar arguments, and using the a priori bound from~Lemma~\ref{lem:apriori}, we estimate
\begin{equation}
    \begin{split}
        \left \lvert \lambda \tr{e_k^- \left( \Gamma_\lambda^{(1)} - \Gamma_0^{(1)} \right)} \right \lvert & \leq \lambda \left \lVert \sqrt{h}\left( \Gamma_\lambda^{(1)} - \Gamma_0^{(1)} \right)\sqrt{h}  \right \lVert_\text{HS} \left \lVert h^{-1/2}e_k^- h^{-1/2} \right \lVert_{\text{HS}} \\
        & \lesssim |\log \varepsilon |^2 \left \lVert Ph^{-1} \right \lVert_{\text{HS}} \ . \\
    \end{split}
\end{equation}
Next, we write 
\begin{equation}
    \begin{split}
        \left \lVert Ph^{-1} \right \lVert_{\text{HS}} ^2 & = \tr{h^{-1}PPh^{-1}} \\
        & = \tr{\mathbf{1}_{h \leq \Lambda}h^{-2}} \\
        & \lesssim \Lambda^{0+} \ , \\
    \end{split}
\end{equation}
and thus find
\begin{equation}
     \begin{split}
     \left \lvert \lambda \left \langle \text{d}\Gamma(e_k^-)   \right \rangle_\lambda \right \lvert & \lesssim |\log \varepsilon |^2 \Lambda^{0+} + \Lambda^{0+} \\
     & \lesssim |\log \varepsilon |^2 \Lambda^{0+} \\
     \end{split}
\end{equation}
for $\lambda$ small enough. Now, since $(e_k)^* = e_{-k}$, we see that 
\begin{equation*}
        (e_k^-)^*  = \left(Pe_kP \right)^* = P (e_k)^*P = Pe_{-k}P = e^{-}_{-k} \ . \\
\end{equation*}
Hence, using the two estimates we just derived, we get
\begin{equation}
    \begin{split}
        \lambda^2 \Re \left\{ \overline{\left \langle \text{d}\Gamma(e_k^-)   \right \rangle_\lambda} \left \langle \text{d}\Gamma(e_k^-)   \right \rangle_0 \right\} & = \lambda \Re \left\{ \tr{e_{-k}^-\Gamma_\lambda^{(1)}} \lambda  \left \langle \text{d}\Gamma(e_k^-)   \right \rangle_0\right\} \\
        & \leq \lambda \Re \left\{ \tr{e_{-k}^-\Gamma_\lambda^{(1)}}  \int \ps{u}{e_k^- u}\text{d}\mu_0(u)\right\} + C \lambda\Lambda^{1+}\Re \left\{ \tr{e_{-k}^-\Gamma_\lambda^{(1)}} \right\} \\
        & \leq \lambda \Re \left\{ \tr{e_{-k}^-\Gamma_\lambda^{(1)}}  \int \ps{u}{e_k^- u}\text{d}\mu_0(u)\right\} + C \lambda\Lambda^{1+}|\log \varepsilon |^2 \ . \\
    \end{split}
\end{equation}
Let us now look at the two remaining terms in~\eqref{eq:expanding the square in trace}. First, we expand $\left \lvert \text{d}\Gamma(e_k^-) \right \lvert^2$ as follows 
\begin{equation}
    \begin{split}
        \left \lvert \text{d}\Gamma(e_k^-) \right \lvert^2 &= 2 \bigoplus_{n = 2}^\infty \left( \sum_{1 \leq j < \ell \leq n} \left(e_{-k}^- \right)_j \otimes \left(e_{-k} \right)_\ell \right) + \text{d}\Gamma\left(|e_k^-|^2\right) \\
        & = 2 \dgamma{e_{-k}^- \otimes e_k^-} + \dgamma{|e_k^-|^2} \\
        & \geq 2 \dgamma{e_{-k}^- \otimes e_k^-} \\
    \end{split}
\end{equation}
as operators. Then, using \eqref{eqlambda mean dgamma e_k^-} and the fact that $\left(e^{\lambda h}-1\right)^{-1} - (\lambda h)^{-1} \gtrsim -1$ we have  
\begin{equation}
    \begin{split}
        \left \lvert \lambda \left \langle \text{d}\Gamma(e_k^-)   \right \rangle_0 \right \lvert^2 & = \left \lvert \tr{e_k^- h^{-1}} \right \lvert^2 + 2 \lambda \Re \left\{ \tr{e_k^- h^{-1}} \tr{e_k^-\left( \frac{1}{e^{\lambda h}-1} - \frac{1}{\lambda h}\right)} \right\} \\
        &\quad \quad \quad \quad \quad \quad \quad \ +\lambda^2 \left \lvert  \tr{e_k^-\left( \frac{1}{e^{\lambda h}-1} - \frac{1}{\lambda h}\right)} \right \lvert^2 \\
        & \geq \left \lvert \tr{e_k^- h^{-1}} \right \lvert^2 - C\lambda \Re \left\{ \tr{Ph^{-1}} \tr{Pe_kP}\right\} \\
        & \geq \left \lvert \int \ps{u}{e_k^- u}\text{d}\mu_0(u) \right \lvert^2 - C \lambda \Lambda^{1+} \ . \\
    \end{split}
\end{equation}
Gathering all these expressions, we find
\begin{equation} \label{eq:ek-}
    \begin{split}
        \frac{\lambda^2}{2}\left \langle {\left \lvert \text{d}\Gamma(e_k^-) - \left \langle \text{d}\Gamma(e_k^-) \right \rangle_0 \right \lvert}^2 \right \rangle_\lambda & \geq \lambda^2 \left \langle \dgamma{e_{-k}^- \otimes e_k^-} \right \rangle_\lambda - \lambda \Re \tr{e_{-k}^- \Gamma_\lambda^{(1)}} \int_{P \mathfrak{H}} \ps{u}{e_k u}\text{d}\mu_{0,K}(u)  \\
        & \quad + \frac{1}{2} \left \lvert \int_{P \mathfrak{H}} \ps{u}{e_k u}\text{d}\mu_{0,K}(u) \right \lvert^2 - C \left( \lambda \Lambda^{1+} + \lambda \Lambda^{1+} |\log \varepsilon|^2 \right) \ . \\
    \end{split}
\end{equation}

Here we used that, by definition, 
$$\int \ps{u}{e_k^- u}\text{d}\mu_0(u) = \int_{P \mathfrak{H}} \ps{u}{e_k u}\text{d}\mu_{0,K}(u) \ .$$ 
Next, we define $\Gamma_{\lambda,P}$ as the $P$-localization of the state $\Gamma_\lambda$ as in Definition~\ref{def:Ploc}. Using Theorem \ref{thmLower symbols as de Finetti measures}, we have the crucial expressions 
\begin{equation} \label{eq:husimi12}
    \begin{split}
        \lambda \Gamma_{\lambda, P}^{(1)} & = \int_{P \mathfrak{H}} \ket{u}\bra{u}\text{d}\mu_{\lambda, P}(u) - \lambda P \\
        \lambda^2 \Gamma_{\lambda, P}^{(2)} & = \frac{1}{2}\int_{P \mathfrak{H}} \ket{u^{\otimes 2}}\bra{u^{\otimes 2}}\text{d}\mu_{\lambda, P}(u) -  \lambda^2\Gamma_{\lambda,P}^{(1)}\otimes_sP - 2 \lambda^2 P\otimes_sP \ , \\
    \end{split}
\end{equation}
where $\mu_{\lambda, P}= \mu_{P,\Gamma_\lambda}^\lambda$ is the lower symbol associated to the state $\Gamma_{\lambda,P}$, at scale $\lambda$. This will allow us to deal with the terms $\left \langle \dgamma{e_{-k}^- \otimes e_k^-} \right \rangle_\lambda$ and $\tr{e_{-k}^- \Gamma_\lambda^{(1)}}$. Indeed, using \eqref{eq:husimi12}, we have
\begin{equation}
    \begin{split}
        \lambda e_{-k}\Gamma_{\lambda,P}^{(1)} = \int_{P \mathfrak{H}} \ket{e_{-k}u}\bra{u}\text{d}\mu_{\lambda, P}(u) - \lambda e_{-k}P \ . \\
    \end{split}
\end{equation}
Then, we write 
\begin{equation} \label{eq : first rdm de finetti}
    \begin{split}
        \lambda \tr{e_{-k}^- \Gamma_\lambda^{(1)}} & = \lambda \tr{e_{-k}\Gamma_{\lambda, P}^{(1)}} \\
        & =  \int_{P \mathfrak{H}} \ps{u}{e_{-k}u} \text{d}\mu_{\lambda, P}(u) + \mathcal{O}(\lambda \Lambda^{1+}) \ . \\
    \end{split}
\end{equation}
In a similar manner, we treat the second density matrix as follows
\begin{equation}
    \begin{split}
        \lambda^2 (e_{-k}^- \otimes e_k^-) \Gamma_{\lambda,P}^{(2)} &= \frac{1}{2}\int_{P\mathfrak{H}} \ket{(e_{-k}^- u) \otimes (e_k^-u)} \bra{u\otimes u} \text{d}\mu_{\lambda, P}(u) - \lambda^2 (e_{-k}^- \otimes e_k^-)\left(\Gamma_{\lambda,P}^{(1)}\otimes_sP\right) \\
        & \quad \quad - 2\lambda^2 (e_{-k}^- \otimes e_k^-) \left( P\otimes_sP \right) \ . \\
    \end{split}
\end{equation}
Taking the trace yields 
\begin{equation}
    \begin{split}
        \lambda^2\left \langle \dgamma{e_{-k}^- \otimes e_k^-} \right \rangle_\lambda & = \lambda^2\tr{\left(e_{-k}^- \otimes e_k^-\right) \Gamma_{\lambda}^{(2)}} \\
        & = \lambda^2 \tr{(e_{-k}\otimes e_k)\Gamma_{\lambda,P}^{(2)}} \\
        & = \frac{1}{2}\int_{P\mathfrak{H}} \ps{u}{e_{-k}u} \ps{u}{e_ku} \text{d}\mu_{\lambda, P}(u) - \lambda^2 \tr{(e_{-k}^- \otimes e_k^-)\left(\Gamma_{\lambda,P}^{(1)}\otimes_sP\right)} \\
        & \quad \quad - 2\lambda^2 \tr{(e_{-k}^- \otimes e_k^-) \left( P\otimes_sP \right)} \ . \\
    \end{split}
\end{equation}
Remark also that
\begin{equation}
    \begin{split}
        \tr{(e_{-k}^- \otimes e_k^-) \left( P\otimes_sP \right)} & \leq \tr{e_{-k}P} \tr{e_kP} \\
        & \leq \tr{P}^2 \\
        & \lesssim \Lambda^{2+} \ , \\
    \end{split}
\end{equation}
and 
\begin{equation}
    \begin{split}
        \tr{(e_{-k}^- \otimes e_k^-)\left(\Gamma_{\lambda,P}^{(1)}\otimes_sP\right)} & \leq \tr{e_{-k}\Gamma_{\lambda,P}^{(1)}} \tr{e_k^-P} \\
        & \leq \tr{e_{-k}P \Gamma_{\lambda}^{(1)}P}\tr{P} \\
        & \lesssim \Lambda^{1+} \tr{e_{-k}^-\Gamma_{\lambda}^{(1)}} \\
        & \lesssim \lambda^{-1}\Lambda^{1+}|\log \varepsilon|^2 \ . \\
    \end{split}
\end{equation}
In the last line we used the previously derived estimate on $ \lambda \left \langle \text{d}\Gamma(e_k^-)   \right \rangle_\lambda$, and the fact that 
$$\lambda^{-1}\Lambda^{1+} \underset{\lambda \rightarrow 0^+}{=} o(\lambda^{-1}\Lambda^{1+}|\log \varepsilon|^2) \ .$$
Finally, we get  
\begin{equation}
    \begin{split}
        \lambda^2\left \langle \dgamma{e_{-k}^- \otimes e_k^-} \right \rangle_\lambda & \geq \frac{1}{2}\int_{P\mathfrak{H}} \ps{u}{e_{-k}u} \ps{u}{e_ku} \text{d}\mu_{\lambda, P}(u) - C \left( \lambda \Lambda^{1+}|\log \varepsilon|^2 + \lambda^2 \Lambda^{2+} \right) \ . \\
    \end{split}
\end{equation}
We then plug the above and \eqref{eq : first rdm de finetti} into \eqref{eq:ek-} to find
\begin{equation} \label{eq:quantideF}
    \begin{split}
        \frac{\lambda^2}{2}\left \langle {\left \lvert \text{d}\Gamma(e_k^-) - \left \langle \text{d}\Gamma(e_k^-) \right \rangle_0 \right \lvert}^2 \right \rangle_\lambda & \geq \frac{1}{2}\int_{P\mathfrak{H}} \ps{u}{e_{-k}u} \ps{u}{e_ku} \text{d}\mu_{\lambda, P}(u)  \\
        & \quad \ - \Re \int_{P \mathfrak{H}} \ps{u}{e_{-k}u} \text{d}\mu_{\lambda, P}(u)\int_{P \mathfrak{H}} \ps{u}{e_k u}\text{d}\mu_{0,K}(u)  \\
        & \quad \ + \frac{1}{2} \left \lvert \int_{P \mathfrak{H}} \ps{u}{e_k u}\text{d}\mu_{0,K}(u) \right \lvert^2 - C \left( \lambda \Lambda^{1+}|\log \varepsilon|^2 - \lambda^2 \Lambda^{2+} \right) \ . \\
    \end{split}
\end{equation} 
In the above we used that there holds, for any $\alpha >0$,
\begin{equation}
\begin{split}
	\left \lvert \int \ps{u}{e^-_k u} \textnormal{d} \mu_0(u) \right \lvert & \leq \int \| Pu \| ^2 \textnormal{d} \mu_0(u) \\
	& \leq \Lambda^\alpha \int_{P \mathfrak{H}} \|u \|^2_{\mathfrak{H}^{-\alpha}} \textnormal{d} \mu_0(u) \\
	& \leq \Lambda^\alpha \tr{h^{- (1+\alpha)}} \ , \\
\end{split}
\end{equation}
We then insert \eqref{eq:quantideF} and \eqref{eq:quantitativeesti} in the initial Cauchy-Schwarz inequality \eqref{eq:icsvarterms}. We start with \eqref{eq:quantitativeesti}. We have, for all $\delta>0$, 
\begin{equation}
    \begin{split}
        \lambda^2 \left \langle {\left \lvert \text{d}\Gamma(e_k^-) - \left \langle \text{d}\Gamma(e_k^-) \right \rangle_0 \right \lvert}^2 \right \rangle_\lambda & \lesssim (1+\delta)\lambda^2\left \langle {\left \lvert \text{d}\Gamma(e_k) - \left \langle \text{d}\Gamma(e_k) \right \rangle_0 \right \lvert}^2 \right \rangle_\lambda \\
        & \quad + (1+\delta^{-1})\left( |\log \varepsilon|^2 \Lambda^{-1/2+} - \lambda^2 \log \lambda (|k|^2 + |k| \Lambda^{1/2}) + \varepsilon^{-2} \lambda^2 |\log \lambda|^2 \right) \ . \\
    \end{split}
\end{equation}
Now, summing this inequality against $\widehat{w^\varepsilon}(k)$, $k \in 2\pi\mathbb{Z}^2$ (note that for $\lambda$ small enough, the right-hand side is positive), and using the fact that $\sum_{k \in 2\pi\mathbb{Z}^2}\widehat{w^\varepsilon}(k) = w^\varepsilon(0) \simeq \varepsilon^{-2}$, we find  
\begin{equation}
    \begin{split}
        \frac{\lambda^2}{2} \sum_{k \in 2\pi\mathbb{Z}^2}\widehat{w^\varepsilon}(k)\left \langle {\left \lvert \text{d}\Gamma(e_k^-) - \left \langle \text{d}\Gamma(e_k^-) \right \rangle_0 \right \lvert}^2 \right \rangle_\lambda & \lesssim (1+\delta)\lambda^2 \sum_{k \in 2\pi\mathbb{Z}^2}\widehat{w^\varepsilon}(k)\left \langle {\left \lvert \text{d}\Gamma(e_k) - \left \langle \text{d}\Gamma(e_k) \right \rangle_0 \right \lvert}^2 \right \rangle_\lambda \\
        & \quad + (1+\delta^{-1})\left( \varepsilon^{-2}|\log \varepsilon|^2 \Lambda^{-1/2+} + \varepsilon^{-4} \lambda^2 |\log \lambda|^2 \right) \\
        & \quad - (1+\delta^{-1})\lambda^2\log \lambda\sum_{k \in 2\pi\mathbb{Z}^2}\widehat{w^\varepsilon}(k)(|k|^2 + |k| \Lambda^{1/2}) \ . \\
    \end{split}
\end{equation}
Note that from the a priori bound and the fact that the von-Neumann relative entropy is positive, we have $\lambda^2 \left \langle \mathbb{W}^{\text{Ren}}_\varepsilon \right \rangle_\lambda \lesssim |\log \varepsilon|^2 + \varepsilon^{-2} \lesssim \varepsilon^{-2}$. Hence
\begin{equation} \label{eq: ICS LB before opt delta}
    \begin{split}
        \frac{\lambda^2}{2} \sum_{k \in 2\pi\mathbb{Z}^2}\widehat{w^\varepsilon}(k)\left \langle {\left \lvert \text{d}\Gamma(e_k^-) - \left \langle \text{d}\Gamma(e_k^-) \right \rangle_0 \right \lvert}^2 \right \rangle_\lambda & \lesssim \lambda^2 \left \langle \mathbb{W}^{\text{Ren}}_\varepsilon \right \rangle_\lambda + \delta \varepsilon^{-2} + (1+\delta^{-1})g(\varepsilon,\lambda, \Lambda) \ , \\
    \end{split}
\end{equation}
where we defined
\begin{equation}
    g(\varepsilon,\lambda,\Lambda)= \varepsilon^{-2}|\log \varepsilon|^2 \Lambda^{-1/2+} + \varepsilon^{-4} \lambda^2 |\log \lambda|^2 + \lambda^2\log \lambda\sum_{k \in 2\pi\mathbb{Z}^2}\widehat{w^\varepsilon}(k)(|k|^2 + |k| \Lambda^{1/2}) \ . \\
\end{equation}
Now, optimizing over $\delta>0$, we find
\begin{equation}
    \begin{split}
        \lambda^2 \left \langle \mathbb{W}^{\text{Ren}}_\varepsilon \right \rangle_\lambda & \gtrsim \frac{\lambda^2}{2} \sum_{k \in 2\pi\mathbb{Z}^2}\widehat{w^\varepsilon}(k)\left \langle {\left \lvert \text{d}\Gamma(e_k^-) - \left \langle \text{d}\Gamma(e_k^-) \right \rangle_0 \right \lvert}^2 \right \rangle_\lambda \\
        & \quad - \varepsilon^{-1}g(\varepsilon,\lambda,\Lambda)^{1/2} - g(\varepsilon,\lambda,\Lambda) \ . \\
    \end{split}
\end{equation}
Let us then define 
\begin{equation}
    \Tilde{g}(\varepsilon, \lambda, \Lambda)= \varepsilon^{-1}g(\varepsilon,\lambda,\Lambda)^{1/2} + g(\varepsilon,\lambda,\Lambda) + \varepsilon^{-2}\lambda \Lambda^{1+}|\log \varepsilon|^2 + \varepsilon^{-2}\lambda^2 \Lambda^{2+} \ . \\
\end{equation}
From \eqref{eq:quantideF}, we thus have
\begin{equation} \label{eq:wickinter1}
    \begin{split}
        \lambda^2 \left \langle \mathbb{W}^{\text{Ren}}_\varepsilon \right \rangle_\lambda & \geq \frac{1}{2}\sum_{k \in 2\pi\mathbb{Z}^2}\widehat{w^\varepsilon}(k)\int_{P\mathfrak{H}} \ps{u}{e_{-k}u} \ps{u}{e_ku} \text{d}\mu_{\lambda, P}(u)  \\
        & \quad \ - \sum_{k \in 2\pi\mathbb{Z}^2}\widehat{w^\varepsilon}(k)\Re \int_{P \mathfrak{H}} \ps{u}{e_{-k}u} \text{d}\mu_{\lambda, P}(u)\int_{P \mathfrak{H}} \ps{u}{e_k u}\text{d}\mu_{0,K}(u)  \\
        & \quad \ + \frac{1}{2} \sum_{k \in 2\pi\mathbb{Z}^2}\widehat{w^\varepsilon}(k)\left \lvert \int_{P \mathfrak{H}} \ps{u}{e_k u}\text{d}\mu_{0,K}(u) \right \lvert^2 \\
        & \quad - C \Tilde{g}(\varepsilon, \lambda, \Lambda) \\
        & = \frac{1}{2} \int_{P\mathfrak{H}} \sum_{k \in 2\pi\mathbb{Z}^2}\widehat{w^\varepsilon}(k) \left \lvert \ps{u}{e_ku} - \left \langle \ps{u}{e_ku} \right \rangle_{\mu_{0,K}} \right \lvert^2 \text{d}\mu_{\lambda, P}(u) \\
        & \quad - C \Tilde{g}(\varepsilon, \lambda, \Lambda) \\
        & = \frac{1}{2} \int_{P\mathfrak{H}} \left( \iint_{\mathbb{T}^2 \times \mathbb{T}^2} w^\varepsilon(x-y) \wick{|u(x)|^2}\wick{|u(y)|^2} \text{d}x\text{d}y\right) \text{d}\mu_{\lambda, P}(u) \\
        & \quad - C \Tilde{g}(\varepsilon, \lambda, \Lambda) \ . \\ 
    \end{split}
\end{equation}
The first term in the regularized renormalized interaction of the $\Phi^4_2$ theory appears in the above. Notice that here everything is well defined, since we are on the finite-dimensional Hilbert space $P\mathfrak{H}$. In order to obtain the second term of the interaction, we need to look at the remaining non constant term in the quantum variational principle, that is $-\lambda \tau^\varepsilon\left \langle \mathcal{N} - N_0 \right \rangle_\lambda$. To control this term, we write
\begin{equation}
    \begin{split}
        \left \langle \mathcal{N} - N_0 \right \rangle_\lambda& = \tr{\mathcal{N}\left( \Gamma_\lambda - \Gamma_0 \right)} \\
        & = \tr{\Gamma_\lambda^{(1)} - \Gamma_0^{(1)}} \\
        & = \tr{P\left(\Gamma_\lambda^{(1)} - \Gamma_0^{(1)}\right)P} + \tr{Q\left(\Gamma_\lambda^{(1)} - \Gamma_0^{(1)}\right)} \\
        & = \tr{\Gamma_{\lambda,P}^{(1)} - \Gamma_{0,P}^{(1)}} + \tr{Q\left(\Gamma_\lambda^{(1)} - \Gamma_0^{(1)}\right)} \ . \\
    \end{split}
\end{equation}
Here we used the fact that $P+Q = \mathbf{1}$, $P^2 =P$ and the cyclicity of the trace. Now, in the spirit of a previously used argument, we control
\begin{equation}
    \begin{split}
        \lambda \tr{Q\left(\Gamma_\lambda^{(1)} - \Gamma_0^{(1)}\right)} & \leq \lambda \left \lVert Qh^{-1} \right \lVert_{\text{HS}} \left \lVert \sqrt{h} \left(\Gamma_\lambda^{(1)} - \Gamma_0^{(1)} \right)\sqrt{h} \right \lVert_{\text{HS}} \\
        & \lesssim|\log \varepsilon|^2 \Lambda^{-1/2+} \ . \\
    \end{split}
\end{equation}
Furthermore, using \eqref{eq:husimi12}, we compute
\begin{equation}
    \lambda \tr{\Gamma_{\lambda,P}^{(1)}} = \int_{P\mathfrak{H}}\|u\|^2 \text{d}\mu_{\lambda, P}(u) - \lambda\tr{P} \ . \\
\end{equation}
Next, using that $\Gamma_0^{(1)} = \left( e^{\lambda h}-1 \right)^{-1}$ and the operator bound $\left(e^{\lambda h}-1\right)^{-1} - (\lambda h)^{-1} \gtrsim -1$, we have
\begin{equation} \label{eq: trPGamma01}
    \begin{split}
        \lambda \tr{\Gamma_{0,P}^{(1)}} & = \tr{Ph^{-1}} + \lambda \tr{P \left( \frac{1}{e^{\lambda h}-1} - \frac{1}{\lambda h} \right)} \\
        & \geq \int \ps{u}{Pu}\text{d}\mu_0(u) - C \lambda \tr{P} \\
        & \geq \int_{P\mathfrak{H}} \|u\|^2\text{d}\mu_{0,K}(u) - C \lambda \Lambda^{1+} \ . \\
    \end{split}
\end{equation}
Hence, using that $\mu_{\lambda, P}$ is a probability measure, we obtain
\begin{equation}
    \lambda \tr{\Gamma_{\lambda,P}^{(1)} - \Gamma_{0,P}^{(1)}} \leq \int_{P\mathfrak{H}}\left(\|u\|^2 - \int_{P\mathfrak{H}} \|u\|^2\text{d}\mu_{0,K}(u)\right)\text{d}\mu_{\lambda, P}(u) - C \lambda \Lambda \ . \\
\end{equation}
Thus, we get the following estimate
\begin{equation}\label{eq:wickinter2}
    \lambda\tau^\varepsilon\left \langle \mathcal{N} - N_0 \right \rangle_\lambda \leq \tau^\varepsilon \int_{P\mathfrak{H}}\left(\int_{\mathbb{T}^2}\wick{|u(x)|^2} \text{d}x\right)\text{d}\mu_{\lambda, P}(u) - C \left( \lambda \Lambda - |\log \varepsilon|^2 \Lambda^{-1/2+} \right) \ . \\
\end{equation}
Now, combining \eqref{eq:wickinter1}, \eqref{eq:wickinter2} and the quantum variational principle, and adding the constant term $-E^\varepsilon$ on both sides of the obtained inequality, we get
\begin{equation} \label{eq:quantumvarpple1}
    - \log \frac{Z(\lambda)}{Z_0(\lambda)} \geq \mathcal{H}(\Gamma_\lambda, \Gamma_0) + \int_{P\mathfrak{H}} W^\varepsilon_K[u] \text{d}\mu_{\lambda, P}(u) - C\rho(\varepsilon,\lambda,\Lambda) \ , \\
\end{equation}
where $W^\varepsilon_K[u]$ is the (well-defined) regularized renormalized interaction, and 
\begin{equation}
    \rho(\varepsilon,\lambda,\Lambda) = \Tilde{g}(\varepsilon,\lambda,\Lambda)- \lambda \Lambda + |\log \varepsilon|^2 \Lambda^{-1/2+} \ . \\
\end{equation}
Furthermore, by Theorem \ref{thm:Relative entropyquantum to classical}, we have
\begin{equation}
    \mathcal{H}(\Gamma_\lambda,\Gamma_0) \geq \mathcal{H}(\Gamma_{\lambda,P},\Gamma_{0,P}) \geq \mathcal{H}_{\text{cl}}(\mu_{\lambda, P},\mu_{P,0}) \ , \\
\end{equation}
which is the final step we needed to recover all of the classical elements. Here we used the shorthand notation $\mu_{P,0} = \mu_{P,\Gamma_0}^{\lambda}$. Recall also that
\begin{equation} \label{eq:classicalre}
    \mathcal{H}_{\text{cl}}(\mu_{\lambda, P},\mu_{P,0}) = \int_{P\mathfrak{H}} \frac{\text{d}\mu_{\lambda, P}}{\text{d}\mu_{P,0}}(u) \log \frac{\text{d}\mu_{\lambda, P}}{\text{d}\mu_{P,0}}(u) \text{d}\mu_{P,0}(u) \ . \\
\end{equation}
Now, recall that the measures we are considering are all absolutely continuous with respect to $\text{d}u$ the Lebesgue measure on $P\mathfrak{H} \simeq \mathbb{C}^K$, since they are defined as the following lower symbol
\begin{equation}
    \text{d}\mu_{P,\Gamma}^\delta(u) = (\delta \pi)^{-K} \ps{\xi\left(\frac{u}{\sqrt{\delta}} \right)}{\Gamma_P \xi\left(\frac{u}{\sqrt{\delta}} \right)}_{\mathfrak{F}(P\mathfrak{H})} \text{d}u
\end{equation}
for all $\delta>0$ and all state $\Gamma$. In particular, the classical relative entropy \eqref{eq:classicalre} reduces to
\begin{equation}
    \mathcal{H}_{\text{cl}}(\mu_{\lambda, P},\mu_{P,0}) = \int_{P\mathfrak{H}} \log \text{d}\mu_{\lambda, P}(u) \text{d}\mu_{\lambda, P}(u) - \int_{P\mathfrak{H}}\log\text{d}\mu_{P,0}(u)\text{d}\mu_{\lambda, P}(u) \ . \\
\end{equation}
The goal now is to replace the symbol $\mu_{P,0}$ of the $P$-localized free Gibbs state  by the cylindrical projection $\mu_{0,K}$ of the free measure. We would indeed like to obtain the term 
$$\int_{P\mathfrak{H}}W^\varepsilon_K[u] \text{d}\mu_{0,K}(u) \ ,$$ 
which is exactly the one we seek for our lower bound. To do so, since $\mu_{0,K} = z_0^{-1}e^{-\ps{u}{hu}}\text{d}u$ is also absolutely continuous with respect to the Lebesgue measure $\text{d}u$ on $P\mathfrak{H}$, remark first that
\begin{equation} \label{eq: relative entropy computation lowerb}
    \begin{split}
        \mathcal{H}_{\text{cl}}(\mu_{\lambda, P},\mu_{0,K}) & = \int_{P\mathfrak{H}} \frac{\text{d}\mu_{\lambda, P}}{\text{d}\mu_{0,K}}(u) \log \frac{\text{d}\mu_{\lambda, P}}{\text{d}\mu_{0,K}}(u) \text{d}\mu_{0,K}(u) \\
        & = \int_{P\mathfrak{H}} \log \text{d}\mu_{\lambda, P}(u) \text{d}\mu_{\lambda, P}(u) - \int_{P\mathfrak{H}}\log\text{d}\mu_{0,K}(u)\text{d}\mu_{\lambda, P}(u) \\
        & = \int_{P\mathfrak{H}} \log \text{d}\mu_{\lambda, P}(u) \text{d}\mu_{\lambda, P}(u) + \log z_0 + \int_{P\mathfrak{H}}\ps{u}{hu}\text{d}\mu_{\lambda, P}(u) \ . \\
    \end{split}
\end{equation}
Hence we find
\begin{equation} \label{eq:relativeentcomp}
    \mathcal{H}_{\text{cl}}(\mu_{\lambda, P},\mu_{P,0}) = \mathcal{H}_{\text{cl}}(\mu_{\lambda, P},\mu_{0,K}) - \int_{P\mathfrak{H}}\log\text{d}\mu_{P,0}(u)\text{d}\mu_{\lambda, P}(u) - \log z_0 - \int_{P\mathfrak{H}}\ps{u}{hu}\text{d}\mu_{\lambda, P}(u) \ . \\
\end{equation}
Next, we compute the lower symbol of the free Gibbs state in Lemma~\ref{lemma:Lower symbol for the Free gibbs state} below, finding
\begin{equation}
    \text{d}\mu_{P,0}(u) = \frac{1}{\Tilde{z}_0}e^{-\ps{u}{\Tilde{h}u}}\text{d}u \ , \\
\end{equation}
with
\begin{equation*}
    \begin{split}
        \Tilde{z}_0 &= (\lambda \pi)^{K}\prod_{j=1}^K\left(1-e^{-\lambda \lambda_j} \right)^{-1} \\
        \Tilde{h} &:= \lambda^{-1}P\left(1-e^{-\lambda h} \right)P \ . \\
    \end{split}
\end{equation*}
Coming back to \eqref{eq:relativeentcomp}, we write
\begin{equation} \label{eq: relative ent muP0 mu0K}
    \mathcal{H}_{\text{cl}}(\mu_{\lambda, P},\mu_{P,0}) = \mathcal{H}_{\text{cl}}(\mu_{\lambda, P},\mu_{0,K})  + \log \frac{\Tilde{z}_0}{z_0} + \int_{P\mathfrak{H}}\ps{u}{\left(\Tilde{h}-h\right)u}\text{d}\mu_{\lambda, P}(u) \ . \\
\end{equation}
Hence, the variational principle \eqref{eq:quantumvarpple1} gives us
\begin{equation} \label{eq:twotermsleftqvpple}
    \begin{split}
        - \log \frac{Z(\lambda)}{Z_0(\lambda)} & \geq \mathcal{H}_{\text{cl}}(\mu_{\lambda, P},\mu_{0,K}) + \int_{P\mathfrak{H}}W^\varepsilon_K[u] \text{d}\mu_{\lambda, P}(u) \\
        & \quad  + \log \frac{\Tilde{z}_0}{z_0} + \int_{P\mathfrak{H}}\ps{u}{\left(\Tilde{h}-h\right)u}\text{d}\mu_{\lambda, P}(u) \\
        & \quad - C\rho(\varepsilon,\lambda,\Lambda) \\
        & \geq - \log \left( \int_{P\mathfrak{H}} e^{-W^\varepsilon_K[u]} \text{d}\mu_{0,K}(u)\right) \\
        & \quad + \log \frac{\Tilde{z}_0}{z_0} + \int_{P\mathfrak{H}}\ps{u}{\left(\Tilde{h}-h\right)u}\text{d}\mu_{\lambda, P}(u) \\
        & \quad - C\rho(\varepsilon,\lambda,\Lambda) \ . \\
    \end{split}
\end{equation}
In the second inequality we used the classical variational principle \eqref{eqclassical variational problem 1}. Let us now treat the two remaining terms in \eqref{eq:twotermsleftqvpple}. First, we have
\begin{equation}
    \begin{split}
        \frac{\Tilde{z}_0}{z_0} & \ \ = \prod_{j=1}^K \frac{\lambda \lambda_j}{1 - e^{-\lambda \lambda_j}} \\
        & \underset{\lambda \rightarrow 0^+}{\simeq} \prod_{j=1}^K \frac{1}{1 - \lambda \lambda_j} \ , \\
    \end{split}
\end{equation}
so that
\begin{equation}
   \begin{split}
        \log \frac{\Tilde{z}_0}{z_0}  &\underset{\lambda \rightarrow 0^+}{\simeq} - \sum_{j=1}^K \log(1-\lambda \lambda_j) \\
        &\underset{\lambda \rightarrow 0^+}{\simeq} \lambda \sum_{j=1}^K\lambda_j \\
        & \ \ = \lambda \tr{Ph} \\
        &\underset{\lambda \rightarrow 0^+}{\simeq} \lambda \Lambda^2 \ . \\
   \end{split}
\end{equation}
Furthermore, we have
\begin{equation} \label{eq:1-e^-lambda h asymptotics}
    1-e^{-\lambda h} \underset{\lambda \rightarrow 0^+}{=} \lambda h - \lambda^2 h^2 + o(\lambda^2) \ . \\
\end{equation}
Indeed, write 
\begin{equation}
    1-e^{-\lambda h} = \lambda h - \lambda^2 h^2 + \sum_{n=3}^\infty (-1)^n \frac{(\lambda h)^n}{n!} \ . \\
\end{equation}
We certainly have 
$$\|Ph\|_{\text{op}} = \sup_{u \in P\mathfrak{H}, \|u\| = 1} \|h u\| \leq \Lambda \ , $$
and thus
\begin{equation}
    \begin{split}
        \left \lVert P \sum_{n=3}^\infty (-1)^n \frac{(\lambda h)^n}{n!}\right \lVert_{\text{op}}  & \leq \lambda^3 \Lambda^3 \sum_{n=3}^\infty \frac{(\lambda \Lambda)^{n-3}}{n!} \\
        & \leq \frac{\lambda^3 \Lambda^3}{6} \sum_{n=3}^\infty \frac{(\lambda \Lambda)^{n-3}}{(n-3)!} \\
        & \lesssim \lambda^3 \Lambda^3 e^{\lambda \Lambda} \ . \\
    \end{split}
\end{equation}
Recalling that $\Lambda = \lambda^{-\nu}$ for some $\nu >0$, we get
\begin{equation} \label{eq:first cond on nu}
    \begin{split}
        \frac{\lambda^3 \Lambda^3 e^{\lambda \Lambda}}{\lambda^2} & = \lambda^{1-3\nu} e^{\lambda \Lambda} \underset{\lambda \rightarrow 0^+}{\longrightarrow} 0 \quad \textnormal{iff} \ \nu < \frac{1}{3} \ . \\
    \end{split}
\end{equation}
Hence, \eqref{eq:1-e^-lambda h asymptotics} yields
\begin{equation}
    \begin{split}
        \int_{P\mathfrak{H}}\ps{u}{\left(\Tilde{h}-h\right)u}\text{d}\mu_{\lambda, P}(u) & \underset{\lambda \rightarrow 0^+}{\simeq} - \lambda \int_{P\mathfrak{H}}\ps{u}{h^2u}\text{d}\mu_{\lambda, P}(u) \\
        & \underset{\lambda \rightarrow 0^+}{\gtrsim} -\lambda \Lambda^2 \int_{P\mathfrak{H}} \|u\|^2\text{d}\mu_{\lambda, P}(u) \ . \\
    \end{split}
\end{equation}
Now, recall that
\begin{equation}
    \begin{split}
        \int_{P\mathfrak{H}} \|u\|^2\text{d}\mu_{\lambda, P}(u) & = \lambda \tr{\Gamma_{\lambda,P}^{(1)}} + \lambda \tr{P} \\
        & \lesssim - \log \lambda + \lambda \Lambda^{1+} \ . \\
    \end{split}
\end{equation}
Here we used from previous computations that $\tr{\Gamma_{\lambda,P}^{(1)}} \lesssim -\lambda^{-1} \log \lambda$ and $\tr{P} \lesssim \Lambda^{1+}$. Finally, plugging these estimates back into \eqref{eq:twotermsleftqvpple}, we arrive at
\begin{equation}
    \begin{split}
        - \log \frac{Z(\lambda)}{Z_0(\lambda)} 
        & \geq - \log \left( \int_{P\mathfrak{H}} e^{-W^\varepsilon_K[u]} \text{d}\mu_{0,K}(u)\right) - C \Tilde{\rho}(\varepsilon, \lambda, \Lambda) \ , \\
    \end{split}
\end{equation}
where $\Tilde{\rho}(\varepsilon, \lambda, \Lambda)= \rho(\varepsilon,\lambda,\Lambda) + \lambda^2 \Lambda^{3+} - \lambda \Lambda^2 - \lambda \log \lambda \Lambda^2$. \\

Let us now look at the full expression for $\Tilde{\rho}(\varepsilon, \lambda, \Lambda)$, and deduce the right conditions on $\eta$ and $\nu$ for our lower bound to hold in the limit. Gathering all the previous expressions together, we have
\begin{equation}
    \begin{split}
        \Tilde{\rho}(\varepsilon, \lambda, \Lambda) & = \varepsilon^{-1}\left( \varepsilon^{-2}|\log \varepsilon|^2 \Lambda^{-1/2+} + \varepsilon^{-4} \lambda^2 |\log \lambda|^2 + \lambda^2\log \lambda\sum_{k \in 2\pi\mathbb{Z}^2}\widehat{w^\varepsilon}(k)(|k|^2 + |k| \Lambda^{1/2}) \right)^{1/2} \\
        & \quad + \varepsilon^{-2}|\log \varepsilon|^2 \Lambda^{-1/2+} + \varepsilon^{-4} \lambda^2 |\log \lambda|^2 + \lambda^2\log \lambda\sum_{k \in 2\pi\mathbb{Z}^2}\widehat{w^\varepsilon}(k)(|k|^2 + |k| \Lambda^{1/2}) \\
        & \quad + \varepsilon^{-2}\lambda \Lambda^{1+}|\log \varepsilon|^2 + \varepsilon^{-2}\lambda^2 \Lambda^{2+} + \varepsilon^{-2}\lambda^2 \Lambda^{1+} |\log \varepsilon|^2 - \lambda \Lambda + |\log \varepsilon|^2 \Lambda^{-1/2+} \\
        & \quad + \lambda^2 \Lambda^{3+} - \lambda \Lambda^2 - \lambda \log \lambda \Lambda^2 \ . \\
    \end{split}
\end{equation}
Recall that we seek $\varepsilon$ and $\Lambda$ of the form $\varepsilon = \lambda^{\eta}$, $\eta > 0$ and $\Lambda = \lambda^{-\nu}$, $\nu > 0$.  With elementary manipulations,
\begin{equation} \label{eq:rho simplified 1}
    \begin{split}
        \Tilde{\rho}(\varepsilon, \lambda, \Lambda) & \leq  \varepsilon^{-2}|\log \varepsilon| \Lambda^{-1/4+} + \varepsilon^{-3} \lambda |\log \lambda| + \varepsilon^{-1} \lambda|\log \lambda|^{1/2} \left( \sum_{k \in 2\pi\mathbb{Z}^2}\widehat{w^\varepsilon}(k)(|k|^2 + |k| \Lambda^{1/2}) \right)^{1/2} \\
        & \quad + \varepsilon^{-2}|\log \varepsilon|^2 \Lambda^{-1/2+} + \varepsilon^{-4} \lambda^2 |\log \lambda|^2 + \lambda^2\log \lambda\sum_{k \in 2\pi\mathbb{Z}^2}\widehat{w^\varepsilon}(k)(|k|^2 + |k| \Lambda^{1/2}) \\
        & \quad + \varepsilon^{-2}\lambda \Lambda^{1+}|\log \varepsilon|^2 + \varepsilon^{-2}\lambda^2 \Lambda^{2+} + \varepsilon^{-2}\lambda^2 \Lambda^{1+} |\log \varepsilon|^2 - \lambda \Lambda + |\log \varepsilon|^2 \Lambda^{-1/2+} \\
        & \quad + \lambda^2 \Lambda^{3+} - \lambda \Lambda^2 - \lambda \log \lambda \Lambda^2 \\
        & \lesssim \varepsilon^{-2}|\log \varepsilon| \Lambda^{-1/4+} + \varepsilon^{-3} \lambda |\log \lambda| + \varepsilon^{-1} \lambda|\log \lambda|^{1/2} \left( \sum_{k \in 2\pi\mathbb{Z}^2}\widehat{w^\varepsilon}(k)(|k|^2 + |k| \Lambda^{1/2}) \right)^{1/2} \\
        & \quad + \lambda^2\log \lambda\sum_{k \in 2\pi\mathbb{Z}^2}\widehat{w^\varepsilon}(k)(|k|^2 + |k| \Lambda^{1/2}) + \varepsilon^{-2}\lambda \Lambda^{1+}|\log \varepsilon|^2 + \varepsilon^{-2}\lambda^2 \Lambda^{2+}  + \lambda^2 \Lambda^{3+} + \lambda |\log \lambda| \Lambda^2 \ , \\
    \end{split}
\end{equation}
where we assumed $\eta <1$ , so that e.g. $\varepsilon^{-4} \lambda^2 |\log \lambda|^2 \ll \varepsilon^{-3} \lambda |\log \lambda|$. Moreover, by our assumption on $w$, we have
\begin{equation}
\begin{split}
    \sum_{k \in 2\pi\mathbb{Z}^2}\widehat{w^\varepsilon}(k)|k|^2 & \lesssim \varepsilon^{-4} \\
\sum_{k \in 2\pi\mathbb{Z}^2}\widehat{w^\varepsilon}(k)|k| & \lesssim \varepsilon^{-3} \ . \\
\end{split}
\end{equation}
Thus, \eqref{eq:rho simplified 1} reduces to
\begin{equation}
        \begin{split}
        \Tilde{\rho}(\varepsilon, \lambda, \Lambda)
        & \lesssim \varepsilon^{-2}|\log \varepsilon| \Lambda^{-1/4+} + \varepsilon^{-3} \lambda |\log \lambda| + \varepsilon^{-3} \lambda|\log \lambda|^{1/2} + \varepsilon^{-5/2} \lambda \Lambda^{1/4}|\log \lambda|^{1/2}  \\
        & \quad + \varepsilon^{-4} \lambda^2\log \lambda + \varepsilon^{-3} \lambda^2\Lambda^{1/2}\log \lambda +  \varepsilon^{-2}\lambda \Lambda^{1+}|\log \varepsilon|^2 + \varepsilon^{-2}\lambda^2 \Lambda^{2+}  + \lambda^2 \Lambda^{3+} + \lambda |\log \lambda| \Lambda^2 \ . \\
    \end{split}
\end{equation}
Looking at all these terms separately as powers of $\lambda$ times powers of $\log \lambda$, and combining this with \eqref{eq:first cond on nu}, we deduce that $\nu$ and $\eta$ must satisfy the following inequalities in order to have $\Tilde{\rho}(\varepsilon, \lambda, \Lambda) = o(1)$
\begin{equation}
    \begin{cases}
        \eta \in \left(0,\frac{1}{3}\right) , \  \nu \in \left(0, \frac{1}{3} \right) \\
        \nu > 8 \eta \\
        \nu < 1 - \eta
    \end{cases} \ . \\
\end{equation}
The last condition being redundant since $\nu < 1/3$, we must have 
\begin{equation}
    8 \eta < \nu < \frac{1}{3} \ . \\
\end{equation}
Hence if $\eta < 1/24$ we may always find an appropriate $\nu$ to meet the above conditions.
\end{proof}

\subsection{Free energy upper bound}

We turn to the the relative free energy upper bound. We will use several ingredients presented in \cite{LewNamRou-20}, and~\cite{DinRou-24}. In \cite{LewNamRou-20} and \cite{NamZhuZhu-25}, the classical measure appears thanks to a trial state argument that allows to reduce the problem to a finite dimensional estimate, and the use of the Peierls-Bogoliubov inequality once this is done. In \cite{DinRou-24}, the main difference is that once the problem is projected in finite dimension, another trial state argument is used, which allows to choose a suitable upper symbol that immediately inserts the desired limiting measure in the estimate.

\begin{proposition}[\textbf{Free energy upper bound}] \label{propFree energy upper bound}\mbox{}\\
    Let $\lambda > 0$,  $\eta \in (0,1/18)$, and $\nu \in (8 \eta, 1/2 - \eta)$. Let $\varepsilon = \lambda^\eta$, $\Lambda = \lambda^{-\nu}$. Let $P = \mathbf{1}_{h \leq \Lambda} = \sum_{k=1}^K \ket{e_k}\bra{e_k}$ where $K = \tr{P}$. Consider the renormalized regularized interaction $W^\varepsilon_K$ as in \eqref{eq: regularized renormalized Int}, and $\mu_{0,K}$ the cylindrical projection of the Gaussian measure $\mu_0$ on $P\mathfrak{H}$.
    Then, we have
    \begin{equation} \label{equpper bound in prop}
        -\log \frac{Z(\lambda)}{Z_0(\lambda)} \leq - \log \left( \int_{P \mathfrak{H}}e^{-W^\varepsilon_K[u]} \textnormal{d}\mu_{0,K}(u) \right) + o_{\lambda \rightarrow0^+}(1) \ . \\
    \end{equation}
\end{proposition}

In order to prove the above proposition, we will construct a suitable \textit{projected} trial state, which will allow us to work in finite dimension. We are free to take any convenient $\Tilde{\Gamma} \in \mathcal{S}(\mathfrak{F}(\mathfrak{H}))$ in the variational principle \eqref{eqvariational problem 2} to obtain the upper bound
\begin{equation} \label{eq:vpple upper bound}
    \begin{split}
        - \log \frac{Z(\lambda)}{Z_0(\lambda)} & \leq  \mathcal{H}(\Tilde{\Gamma}, \Gamma_0) + \lambda^2 \text{Tr}_{\mathfrak{F}(\mathfrak{H})} \left[ \mathbb{W}^\text{Ren}_\varepsilon \Tilde{\Gamma}\right] - \lambda \tau^\varepsilon \text{Tr}_{\mathfrak{F}(\mathfrak{H})} \left[ \left( \mathcal{N} - N_0\right) \Tilde{\Gamma}\right] - E^\varepsilon \\
        &= \mathcal{H}(\Tilde{\Gamma}, \Gamma_0) + \lambda^2 \text{Tr}_{\mathfrak{F}(\mathfrak{H})} \left[ \Tilde{\mathbb{W}}_\varepsilon \Tilde{\Gamma}\right] \ , \\
    \end{split}
\end{equation}
where we defined
\begin{equation}
    \Tilde{\mathbb{W}}_\varepsilon = \mathbb{W}^\text{Ren}_\varepsilon - \lambda^{-1}\tau^\varepsilon (\mathcal{N} - N_0) - \lambda^{-2}E^\varepsilon \ . \\
\end{equation}
In order to define our initial trial state, let us first consider the following interacting Gibbs state in $\mathfrak{F}(P \mathfrak{H)}$
\begin{equation}
    \Tilde{\Gamma}_{\lambda,P} = \frac{e^{-  \lambda \mathbb{H}_{\lambda,P}^\varepsilon}}{\textnormal{Tr}_{\mathfrak{F}(P \mathfrak{H)}} \left[e^{-  \lambda \mathbb{H}_{\lambda,P}^\varepsilon}\right]} \ , \\
\end{equation}
where $\mathbb{H}_{\lambda,P}^\varepsilon$ is the projected interacting Hamiltonian
\begin{equation}
    \begin{split}
        \mathbb{H}_{\lambda,P}^\varepsilon &= \mathbb{H}_{0,P} + \lambda \Tilde{\mathbb{W}}_{\varepsilon,P} \\
        \mathbb{H}_{0,P} &= \dgamma{Ph} \ , \\
    \end{split}
\end{equation}
$\Tilde{\mathbb{W}}_{\varepsilon,P}$ is the $P$-localized interaction
\begin{equation}
    \Tilde{\mathbb{W}}_{\varepsilon,P} = \frac{1}{2} \sum_{k \in \left( 2 \pi \mathbb{Z} \right)^2} \widehat{w^\varepsilon}(k) \left \lvert \dgamma{e_k^-} - \left \langle \dgamma{e_k^-} \right \rangle_0 \right \lvert^2 - \lambda^{-1}\tau^\varepsilon \left( \dgamma{P} - \left \langle \dgamma{P} \right \rangle_0 \right) - \lambda^{-2}E^\varepsilon \ , \\
\end{equation}
and we recall that $e_k^- = Pe_kP$. Accordingly, define the associated partition functions
\begin{equation}
    \begin{split}
        Z_P(\lambda) &=  \textnormal{Tr}_{\mathfrak{F}(P \mathfrak{H)}} \left[e^{-  \lambda \mathbb{H}_{\lambda,P}^\varepsilon}\right] \\
         Z_{0,P}(\lambda) & = \textnormal{Tr}_{\mathfrak{F}(P \mathfrak{H)}} \left[e^{-  \lambda \mathbb{H}_{0,P}}\right] \ . \\
    \end{split}
\end{equation}
Then, the following lemma links the initial problem to its projected counterpart :

\begin{lemma}[\textbf{Reduction to a finite dimensional estimate}]\label{lem:up fin dim}\mbox{}\\
    Let $\lambda>0$, $\varepsilon = \lambda^{-\eta}$ and $\Lambda = \lambda^\nu$ for some $\eta, \nu >0$. We have 
    \begin{equation} \label{eq: Reduction to a finite dimensional estimate}
        - \log \frac{Z(\lambda)}{Z_0(\lambda)} \leq- \log \frac{Z_P(\lambda)}{Z_{0,P}(\lambda)} + C\Tilde{r}(\varepsilon, \lambda, \Lambda) \ , \\
    \end{equation}
    where 
    \begin{equation} \label{eq:tilde g in the reduction to finite dim}
        \Tilde{r}(\varepsilon, \lambda, \Lambda) = \varepsilon^{-2}\Lambda^{-1/4+}|\log \varepsilon| + \varepsilon^{-3}\lambda \Lambda^{1/4}|\log \lambda|^{1/2} + \varepsilon^{-4}\lambda^2 \Lambda^{1/2}|\log \lambda| \ , \\
    \end{equation}
    and the constant $C >0$ depends only on $\tr{h^{-1-}} < \infty$, $w$ and the parameters $\eta, \nu$.
\end{lemma}

\begin{proof}
    We start by defining the following trial state
    \begin{equation}
        \Tilde{\Gamma} = \mathcal{U}^*\left( \Tilde{\Gamma}_{\lambda,P} \otimes \left( \Gamma_0 \right)_Q \right) \mathcal{U} \ , \\
    \end{equation}
    where $\mathcal{U}$ is the unitary defined in \eqref{eq:unitary fock spaces}, and $\left( \Gamma_0 \right)_Q$ is the $Q$-localization of the non interacting Gibbs state $\Gamma_0$, defined as in \eqref{eq:P loc}, which in fact satisfies 
    \begin{equation}
        \left( \Gamma_0 \right)_Q = \Gamma_{0,Q} = \frac{e^{-\lambda \dgamma{Qh}}}{\text{Tr}_{\mathfrak{F}(Q\mathfrak{H})} \left[ e^{-\lambda \dgamma{Qh}} \right]} \ . \\
    \end{equation}
    Using \cite{LewNamRou-20}, Lemma 10.3, we know that the following holds
    \begin{equation} \label{eq:relative entropy + unitary}
        \mathcal{H}(\Tilde{\Gamma}, \Gamma_0) = \mathcal{H}(\Tilde{\Gamma}_{\lambda,P}, \Gamma_{0,P}) \ , \\
    \end{equation}
    which accounts for the first term in the variational principle of the projected free-energy. Next, we have
    \begin{equation} \label{eq:var term vpple}
        \lambda^2 \text{Tr}_{\mathfrak{F}(\mathfrak{H})} \left[ \Tilde{\mathbb{W}}_\varepsilon \Tilde{\Gamma}\right] = \frac{\lambda^2}{2} \sum_{k \in (2\pi \mathbb{Z})^2}\widehat{w^\varepsilon}(k) \textnormal{Tr}_{\mathfrak{F}(\mathfrak{H})} \left[ \left \lvert \dgamma{e_k} - \left \langle \dgamma{e_k} \right \rangle_0 \right \lvert^2 \Tilde{\Gamma} \right] - \lambda \tau^\varepsilon \textnormal{Tr}_{\mathfrak{F}(\mathfrak{H})} \left[ \left( \mathcal{N} - N_0  \right) \Tilde{\Gamma} \right] - E^\varepsilon \ . \\
    \end{equation}
    In order to control the first term, we use
    \begin{equation} \label{eq: Second ICS UB LB}
        \begin{split}
            \textnormal{Tr}_{\mathfrak{F}(\mathfrak{H})} \left[ \left \lvert \dgamma{e_k} - \left \langle \dgamma{e_k} \right \rangle_0 \right \lvert^2 \Tilde{\Gamma} \right] & \leq (1+\delta) \textnormal{Tr}_{\mathfrak{F}(\mathfrak{H})} \left[ \left \lvert \dgamma{e^-_k} - \left \langle \dgamma{e^-_k} \right \rangle_0 \right \lvert^2 \Tilde{\Gamma} \right] \\
            & \quad \quad + (1+\delta^{-1})\textnormal{Tr}_{\mathfrak{F}(\mathfrak{H})} \left[ \left \lvert \dgamma{e^+_k} - \left \langle \dgamma{e^+_k} \right \rangle_0 \right \lvert^2 \Tilde{\Gamma} \right] \\
        \end{split}
    \end{equation}
    for all $\delta>0$. From \eqref{eq: equivalent def P-loc}, and since $\left \lvert \dgamma{e^-_k} - \left \langle \dgamma{e^-_k} \right \rangle_0 \right \lvert^2$ is a self-adjoint operator on $\mathfrak{F}(P\mathfrak{H})$, hence acting trivially as the identity on $\mathfrak{F}(Q\mathfrak{H})$, we have
    \begin{equation}
        \begin{split}
            \textnormal{Tr}_{\mathfrak{F}(\mathfrak{H})} \left[ \left \lvert \dgamma{e^-_k} - \left \langle \dgamma{e^-_k} \right \rangle_0 \right \lvert^2 \Tilde{\Gamma} \right] & = \textnormal{Tr}_{\mathfrak{F}(P\mathfrak{H})} \left[ \left \lvert \dgamma{e^-_k} - \left \langle \dgamma{e^-_k} \right \rangle_0 \right \lvert^2 (\Tilde{\Gamma})_P \right] \\
            & = \textnormal{Tr}_{\mathfrak{F}(P\mathfrak{H})} \left[ \left \lvert \dgamma{e^-_k} - \left \langle \dgamma{e^-_k} \right \rangle_0 \right \lvert^2 \Tilde{\Gamma}_{\lambda,P} \right] \ . \\
        \end{split}
    \end{equation}
    Indeed, by definition of the $P$-localization of a state, and by cyclicity of the trace, we have
    \begin{equation}
        \begin{split}
            (\Tilde{\Gamma})_P & = \textnormal{Tr}_{\mathfrak{F}(Q\mathfrak{H})} \left[  \Tilde{\Gamma} \right] \\
            & = \textnormal{Tr}_{\mathfrak{F}(Q\mathfrak{H})} \left[  \mathcal{U}^*\left( \Tilde{\Gamma}_{\lambda,P} \otimes \left( \Gamma_0 \right)_Q \right) \mathcal{U} \right] \\
            & = \textnormal{Tr}_{\mathfrak{F}(Q\mathfrak{H})} \left[ \Tilde{\Gamma}_{\lambda,P} \otimes \left( \Gamma_0 \right)_Q \right] \\
            & = \Tilde{\Gamma}_{\lambda,P} \ . \\
        \end{split}
    \end{equation}
    Now, we would like to apply the correlation estimate \eqref{eq:quantitativeesti in thm} to control the term
    \begin{equation*}
        \textnormal{Tr}_{\mathfrak{F}(\mathfrak{H})} \left[ \left \lvert \dgamma{e^+_k} - \left \langle \dgamma{e^+_k} \right \rangle_0 \right \lvert^2 \Tilde{\Gamma} \right] \ . \\
    \end{equation*}
    To do so, remark that $\Tilde{\Gamma}$ is in fact a Gibbs state:
    \begin{equation} \label{eq: trial state as a gibbs state}
        \Tilde{\Gamma} = \frac{e^{-\lambda \Tilde{\mathbb{H}}^\varepsilon_{\lambda,P}}}{\textnormal{Tr}_{\mathfrak{F}(\mathfrak{H})} \left[ e^{-\lambda \Tilde{\mathbb{H}}^\varepsilon_{\lambda,P}} \right]} \ , \\
    \end{equation}
    where 
    \begin{equation}
        \Tilde{\mathbb{H}}^\varepsilon_{\lambda,P} = \dgamma{h} + \lambda \Tilde{\mathbb{W}}_{\varepsilon,P} \ . \\
    \end{equation}
    This comes from the fact that we can write every term in this Hamiltonian as quadratic and quartic polynomials in the creation and annihilation operators, then expand the exponential and use \eqref{eq:action of the unitary on creators and annihilators}. Hence, we are almost exactly in the right situation to apply \eqref{eq:quantitativeesti in thm}, the only difference being that we consider the projected interaction term $\Tilde{\mathbb{W}}_{\varepsilon,P}$ instead of the usual one $\Tilde{\mathbb{W}}_{\varepsilon}$. However, we can apply the proof in a very similar manner to \eqref{eq: trial state as a gibbs state}, and obtain eventually the following quantitative estimate
    \begin{equation}
        \lambda^2\textnormal{Tr}_{\mathfrak{F}(\mathfrak{H})} \left[ \left \lvert \dgamma{e^+_k} - \left \langle \dgamma{e^+_k} \right \rangle_0 \right \lvert^2 \Tilde{\Gamma} \right] \lesssim |\log \varepsilon|^2 \Lambda^{-1/2+} - \lambda^2 \log \lambda (|k|^2 + |k| \Lambda^{1/2}) + \varepsilon^{-2} \lambda^2 |\log \lambda| \ . \\
    \end{equation}
    Now, using the fact that we can express~\cite[Equation~(10.6)]{LewNamRou-20}  the first reduced density matrix of $\Tilde{\Gamma}$ as
    \begin{equation}
        \Tilde{\Gamma}^{(1)} = P\Tilde{\Gamma}_{\lambda,P}^{(1)}P + Q \Gamma_0^{(1)}Q \ , \\
    \end{equation}
    we have
    \begin{equation}
        \begin{split}
            \tr{\left( \mathcal{N} -N_0 \right) \Tilde{\Gamma}} & = \tr{P\Tilde{\Gamma}_{\lambda,P}^{(1)}P + Q \Gamma_0^{(1)}Q} - N_0 \\
            & = \tr{P\Tilde{\Gamma}_{\lambda,P}^{(1)}} + \tr{\left( 1-P \right)\Gamma_0^{(1)}} - N_0 \\
            & = \tr{\left(\dgamma{P} - \left \langle \dgamma{P}  \right \rangle_0 \right)\Tilde{\Gamma}_{\lambda,P}} \ . \\
        \end{split}
    \end{equation}
    Thus we find from \eqref{eq: Second ICS UB LB} that
    \begin{equation}
        \lambda^2 \textnormal{Tr}_{\mathfrak{F}(\mathfrak{H})} \left[ \Tilde{\mathbb{W}}_\varepsilon \Tilde{\Gamma} \right] \leq \lambda^2\textnormal{Tr}_{\mathfrak{F}(P \mathfrak{H})} \left[ \Tilde{\mathbb{W}}_{\varepsilon, P} \Tilde{\Gamma}_{\lambda, P}\right] + \delta C \varepsilon^{-2} + (1 + \delta^{-1}) C r(\varepsilon, \lambda, \Lambda) \ , \\
    \end{equation}
    where 
    \begin{equation}
        \begin{split}
            r(\varepsilon, \lambda, \Lambda) &= \sum_{k \in 2\pi\mathbb{Z}^2} \widehat{w^\varepsilon}(k) \left( |\log \varepsilon|^2 \Lambda^{-1/2+} - \lambda^2 \log \lambda (|k|^2 + |k| \Lambda^{1/2}) + \varepsilon^{-2} \lambda^2 |\log \lambda| \right) \\
            & \simeq \varepsilon^{-2}|\log \varepsilon|^2 \Lambda^{-1/2+} - \varepsilon^{-4} \lambda^2 \log \lambda -\varepsilon^{-4} \lambda^2\Lambda^{1/2} \log \lambda + \varepsilon^{-4} \lambda^2 |\log \lambda| \\
            & \lesssim \varepsilon^{-2}|\log \varepsilon|^2 \Lambda^{-1/2+}+\varepsilon^{-4} \lambda^2\Lambda^{1/2} |\log \lambda| \ . \\
        \end{split}
    \end{equation}
    Here we argued as when deriving~\eqref{eq: ICS LB before opt delta}. Hence, optimizing over $\delta >0$ yields
    \begin{equation}
        \lambda^2 \textnormal{Tr}_{\mathfrak{F}(\mathfrak{H})} \left[ \Tilde{\mathbb{W}}_\varepsilon \Tilde{\Gamma} \right] \leq \lambda^2\textnormal{Tr}_{\mathfrak{F}(P \mathfrak{H})} \left[ \Tilde{\mathbb{W}}_{\varepsilon, P} \Tilde{\Gamma}_{\lambda, P}\right] + C\left( \varepsilon^{-1}r(\varepsilon, \lambda, \Lambda)^{1/2} + r(\varepsilon, \lambda, \Lambda) \right) \ . \\
    \end{equation}
    Then, using the initial variational principle and \eqref{eq:relative entropy + unitary}, one obtains 
    \begin{equation}\label{eq:truc}
        \begin{split}
            - \log \frac{Z(\lambda)}{Z_0(\lambda)} & \leq \mathcal{H}(\Tilde{\Gamma}_{\lambda,P}, \Gamma_{0,P}) +\lambda^2 \text{Tr}_{\mathfrak{F}(P\mathfrak{H})} \left[ \Tilde{\mathbb{W}}_{\varepsilon,P} \Tilde{\Gamma}_{\lambda,P}\right] + C \left( \varepsilon^{-1}r(\varepsilon, \lambda, \Lambda)^{1/2} + r(\varepsilon, \lambda, \Lambda) \right) \\
        & = - \log \frac{Z_P(\lambda)}{Z_{0,P}(\lambda)} + C \left(  \varepsilon^{-1}r(\varepsilon, \lambda, \Lambda)^{1/2} + r(\varepsilon, \lambda, \Lambda) \right) \ , \\
        \end{split}
    \end{equation}
    which is the desired estimate. In the last line we used the fact that the projected Gibbs state $\Tilde{\Gamma}_{\lambda,P}$ is the unique minimizer for the variational principle
    \begin{equation}
        - \log \frac{Z_P(\lambda)}{Z_{0,P}(\lambda)} = \inf_{\Gamma \in \mathcal{S}(\mathfrak{F}(P\mathfrak{H}))} \left\{ \mathcal{H}(\Gamma, \Gamma_{0,P}) + \lambda^2 \text{Tr}_{\mathfrak{F}(P\mathfrak{H})} \left[ \Tilde{\mathbb{W}}_{\varepsilon,P} \Gamma\right] \right\} \ . \\
    \end{equation}
\end{proof}

We are now reduced to studying the projected relative free energy :
\begin{lemma}[\textbf{Finite dimensional semiclassics}]\label{lem:finite dim semi}\mbox{}\\
    Let $\lambda>0$, $\varepsilon = \lambda^{\eta}$ and $\Lambda = \lambda^{-\nu}$ for some $\eta, \nu >0$. We have 
    \begin{equation} \label{eq: Finite dimensional semiclassics}
        \begin{split}
            - \log \frac{Z_P(\lambda)}{Z_{0,P}(\lambda)} & \leq - \log \left( \int_{P \mathfrak{H}} e^{- W^\varepsilon_K[u]} \textnormal{d}\mu_{0,K}(u) \right) \\
            & \quad +  C \left( \frac{\lambda \varepsilon^{-2} \Lambda^{1+}}{\mathcal{Z}^\varepsilon_K} \int_{P \mathfrak{H}} \|u\|^2e^{- W^\varepsilon_K[u]} \textnormal{d}\mu_{0,K}(u)  + \lambda \varepsilon^{-2}\Lambda^{2+} - \lambda \Lambda^{1+} |\log \varepsilon| + \lambda \Lambda^{2+} \right) \ , \\
        \end{split}
    \end{equation}
    where $W^\varepsilon_K$ is the regularized renormalized interaction term, and $\mathcal{Z}^\varepsilon_K = \int_{P \mathfrak{H}} e^{- W^\varepsilon_K[u]} \textnormal{d}\mu_{0,K}(u)$.
\end{lemma}

\begin{proof}
   We follow a strategy similar to~\cite[Proof of Lemma~5.4]{DinRou-24}: we control separately $- \log Z_P(\lambda)$ and $- \log Z_{0,P}(\lambda)$. Let us remind that the Gibbs state $\Tilde{\Gamma}_{\lambda, P}$ is the unique minimizer to the following variational principle
    \begin{equation} \label{eq:vpple Gammalambda proj}
        \begin{split}
            - \log Z_P(\lambda) & = \inf_{\Gamma \in \mathcal{S}(\mathfrak{F}(P\mathfrak{H}))} \left\{ \lambda \tr{\mathbb{H}^\varepsilon_{\lambda,P}\Gamma} + \tr{\Gamma \log \Gamma} \right\} \ , \\
        \end{split}
    \end{equation}
    and that similarly, $\Gamma_{0,P}$ is the unique minimizer to the following variational principle
    \begin{equation} \label{eq:vpple Gamma0 proj}
        \begin{split}
            - \log Z_{0,P}(\lambda) & = \inf_{\Gamma \in \mathcal{S}(\mathfrak{F}(P\mathfrak{H}))} \left\{ \lambda \tr{\mathbb{H}_{0,P}\Gamma} + \tr{\Gamma \log \Gamma} \right\} \ . \\
        \end{split}
    \end{equation}
    In order to control these terms separately, we will then have to derive an upper bound on $-\log Z_P(\lambda)$, and a lower bound on $- \log Z_{0,P}(\lambda)$. \\
    
    \noindent \textbf{Upper bound on $-\log Z_P(\lambda)$.} Here, we are free to insert any trial state in \eqref{eq:vpple Gammalambda proj}. Since we are now in finite dimension, we can define the following state
    \begin{equation} \label{eq:trial state}
        \Gamma_{\nu} = \int_{P\mathfrak{H}} \nu(u) \ket{\xi(u/\sqrt{\lambda})} \bra{\xi(u/\sqrt{\lambda})} \textnormal{d}u \in \mathcal{S}(\mathfrak{F}(P\mathfrak{H})) \ , \\
    \end{equation}
    where $\nu(u) = \Gamma^{\textnormal{up}}(u)$ is an \textit{upper symbol} for $\Gamma_\nu$. Notice that as long as $\nu$ is a probability measure which is absolutely continuous with respect to the Lebesgue measure on $\mathbb{C}^K$, $\Gamma_\nu$ will indeed be a state, for the coherent states are normalized on the Fock space. Let us also note that the $k^{\textnormal{th}}$ reduced density matrix of $\Gamma_\nu$ can be explicitly computed as
    \begin{equation} \label{eq:Gamma nu density matrix}
        \begin{split}
            \Gamma_\nu^{(k)} & = \frac{1}{k!\lambda^k} \int_{P \mathfrak{H}} \nu(u) \ket{u^{\otimes k}} \bra{u^{\otimes k}} \textnormal{d}u \ . \\
        \end{split}
    \end{equation}
    Using \eqref{eq:vpple Gammalambda proj} with $\Gamma_\nu$, we thus have
    \begin{equation}
        \begin{split}
            -\log Z_P(\lambda) \leq \lambda \tr{\mathbb{H}^\varepsilon_{\lambda,P} \Gamma_\nu} + \tr{\Gamma_\nu \log \Gamma_\nu} \ . \\
        \end{split}
    \end{equation}
    Let us start with the first term. We compute
    \begin{equation} \label{eq:upper bound upper bound trace}
        \begin{split}
            \tr{\mathbb{H}^\varepsilon_{\lambda,P} \Gamma_\nu}  = \frac{1}{\lambda}\int_{P \mathfrak{H}} \nu(u) \ps{u}{Phu} \textnormal{d}u & + \frac{\lambda}{2} \sum_{k \in (2\pi \mathbb{Z})^2}\widehat{w^\varepsilon}(k) \tr{\left \lvert \dgamma{e^-_k} - \left \langle \dgamma{e^-_k} \right \rangle_0 \right \lvert^2 \Gamma_\nu}  \\
            & - \tau^\varepsilon \tr{ \left( \dgamma{P} - \left \langle \dgamma{P} \right \rangle_0 \right) \Gamma_\nu} - \lambda^{-1}E^\varepsilon \ . \\
        \end{split}
    \end{equation}
    Treating the second term in the above expression is done in a very similar fashion as for the lower bound, meaning that we expand the square, and test the one and two body operators that appear against the trial state $\Gamma_\nu$. The argument starts from an identity similar to \eqref{eq:expanding the square in trace}. We have
    \begin{equation}
        \begin{split}
            \tr{\left \lvert \dgamma{e^-_k} - \left \langle \dgamma{e^-_k} \right \rangle_0 \right \lvert^2 \Gamma_\nu} & = 2 \tr{\left(e_{-k}^- \otimes e_{k}^- \right) \Gamma_{\nu}^{(2)}} + \tr{|e_k^-|^2 \Gamma_\nu^{(1)}} \\
            & \quad - 2 \Re \left \langle \dgamma{e_k^-} \right \rangle_0 \tr{e_{-k}^- \Gamma_\nu^{(1)}} + \left \lvert \left \langle \dgamma{e_k^-} \right \rangle_0 \right \lvert^2 \ . \\
        \end{split}
    \end{equation}
    Next, using \eqref{eq:Gamma nu density matrix} and \eqref{eqlambda mean dgamma e_k^-0}, we have
    \begin{equation}
        \begin{split}
            \tr{\left(e_{-k}^- \otimes e_{k}^- \right) \Gamma_{\nu}^{(2)}} & = \frac{1}{2 \lambda^2} \int_{P \mathfrak{H}} \nu(u) \ps{u}{e_{-k}^-u}\ps{u}{e_{k}^-u} \textnormal{d}u \\
            \tr{|e_k^-|^2 \Gamma_\nu^{(1)}} & = \frac{1}{\lambda}\int_{P \mathfrak{H}} \nu(u) \|e_k^-u\|^2 \textnormal{d}u \\
            \tr{e_{-k}^- \Gamma_\nu^{(1)}} & = \frac{1}{\lambda} \int_{P \mathfrak{H}} \nu(u) \ps{u}{e_{-k}^-u}\textnormal{d}u \\
            \lambda \left \langle \dgamma{e_k^-} \right \rangle_0 & = \int \ps{u}{e_{k}^-u} \textnormal{d}\mu_0(u) + \mathcal{O}(\lambda \Lambda^{1+}) \ . \\
        \end{split}
    \end{equation}
    Hence, we obtain
    \begin{equation}
        \begin{split}
            \lambda^2 \tr{\left \lvert \dgamma{e^-_k} - \left \langle \dgamma{e^-_k} \right \rangle_0 \right \lvert^2 \Gamma_\nu} & = \int_{P \mathfrak{H}} \nu(u) \left \lvert \ps{u}{e_{k}^-u} \right \lvert^2\textnormal{d}u - 2\Re \left \langle \ps{u}{e_{k}u}\right \rangle_{\mu_0} \int_{P \mathfrak{H}} \nu(u) \overline{\ps{u}{e_{k}u}}\textnormal{d}u \\
            & \quad + \left \lvert \left \langle \ps{u}{e_{k}u}\right \rangle_{\mu_0} \right \lvert^2 + 
            \lambda\int_{P \mathfrak{H}} \nu(u) \|e_k^-u\|^2 \textnormal{d}u + \mathcal{O}(\lambda \Lambda^{2+}) \\
            & = \int_{P\mathfrak{H}} \nu(u) {\left \lvert \ps{u}{e_{k}u} - \left \langle \ps{u}{e_{k}u}\right \rangle_{\mu_0} \right \lvert}^2 \textnormal{d}u  \\
            & \quad +  
            \lambda\int_{P \mathfrak{H}} \nu(u) \|e_k^-u\|^2 \textnormal{d}u + \mathcal{O}(\lambda \Lambda^{2+}) \ . \\
        \end{split}
    \end{equation}
    Here we used several times the fact that, since the integrals are taken over $P \mathfrak{H}$, we have $u = Pu$ for all $u \in P \mathfrak{H}$, hence $\ps{u}{e_{k}^-u} = \ps{u}{Pe_{k}Pu} = \ps{Pu}{e_k Pu} = \ps{u}{e_{k}u}$. Now, summing against $\widehat{w^\varepsilon}(k)$, one finds
    \begin{equation}
        \begin{split}
            \frac{\lambda^2}{2} \sum_{k \in (2\pi \mathbb{Z})^2}\widehat{w^\varepsilon}(k) &\tr{\left \lvert \dgamma{e^-_k} - \left \langle \dgamma{e^-_k} \right \rangle_0 \right \lvert^2 \Gamma_\nu} \leq \\
            &\int_{P \mathfrak{H}} \left( \frac{1}{2} \iint_{\mathbb{T}^2 \times \mathbb{T}^2}w^\varepsilon(x-y) \wick{|u(x)|^2} \wick{|u(y)|^2} \textnormal{d}x \textnormal{d}y \right) \nu(u) \textnormal{d}u \\
            & \ +  \frac{\lambda}{2} \int_{P \mathfrak{H}} \sum_{k \in (2\pi \mathbb{Z})^2} \widehat{w^\varepsilon}(k)\|e_k^-u\|^2 \nu(u) \textnormal{d}u + C \lambda \varepsilon^{-2}\Lambda^{2+} \ . \\
        \end{split}
    \end{equation}
    The third term in \eqref{eq:upper bound upper bound trace} is handled using \eqref{eq:Gamma nu density matrix} and \eqref{eq: trPGamma01}
    \begin{equation}
        \begin{split}
            \lambda \tr{ \left( \dgamma{P} - \left \langle \dgamma{P} \right \rangle_0 \right) \Gamma_\nu} & = \lambda \tr{P \Gamma_\nu^{(1)}} - \lambda \tr{P \Gamma_0^{(1)}} \\
            & = \int_{P \mathfrak{H}}\ps{u}{Pu} \nu(u) \textnormal{d}u - \int \ps{u}{Pu}\textnormal{d}\mu_0(u) + \mathcal{O}(\lambda \Lambda^{1+}) \\
            & = \int_{P \mathfrak{H}} \left( \|u\|^2 - \left \langle \|u\|^2 \right \rangle_{\mu_0}  \right) \nu(u) \textnormal{d}u + \mathcal{O}(\lambda \Lambda^{1+}) \\
            & = \int_{P \mathfrak{H}} \left( \int_{\mathbb{T}^2} \wick{|u(x)|^2} \right) \nu(u)\textnormal{d}u + \mathcal{O}(\lambda \Lambda^{1+}) \ . \\
        \end{split}
    \end{equation}
    Thus, we obtain the following estimate
    \begin{equation}
        \begin{split}
            \lambda \tr{\mathbb{H}^\varepsilon_{\lambda,P} \Gamma_\nu} &\leq \int_{P \mathfrak{H}} W^\varepsilon_K[u] \nu(u) \textnormal{d}u + \int_{P \mathfrak{H}}\ps{u}{hu}\nu(u) \textnormal{d}u + C \lambda \varepsilon^{-2} \Lambda^{1+} \int_{P \mathfrak{H}} \|u\|^2 \nu(u) \textnormal{d}u \\
            & \quad  + C \left( \lambda \varepsilon^{-2}\Lambda^{2+} - \lambda \Lambda^{1+} |\log \varepsilon| \right) \ . \\
        \end{split}
    \end{equation}
    Now, in order to control the term $\tr{\Gamma_\nu \log \Gamma_\nu}$, we are going to use the second Berezin-Lieb inequality recalled in~\eqref{eq:Berezinlieb 2} below. Let us first note that, in view of our trial state \eqref{eq:trial state}, we have to consider a different partition of unity. Indeed, changing variables $u \mapsto u/\sqrt{\lambda}$ in  the partition of unity \eqref{eq: cs partition unity}, we find the following rescaled version
    \begin{equation} \label{eq: cs partition unity rescaled}
        (\lambda \pi)^{-K} \int_{\mathfrak{K}} \ket{\xi(u/\sqrt{\lambda})} \bra{\xi(u/\sqrt{\lambda})} \textnormal{d}u = \mathbf{1}_{\mathfrak{F}(\mathfrak{K})} \ . \\
    \end{equation}
    With this in hand, we can find, by using the same scheme of proof as in Appendix~\ref{app:Berezin}, that for every convex function $f: \mathbb{R}^+ \rightarrow \mathbb{R}$, we have
    \begin{equation}
        \tr{f \left(\Gamma_\nu \right)} \leq (\lambda \pi)^{-K} \int_{P \mathfrak{H}} f \left( \frac{\nu(u)}{(\lambda \pi)^{-K}} \right) \textnormal{d}u \ . \\
    \end{equation}
    Hence, applying this with $f(x) = x \log x$, we find
    \begin{equation}
        \begin{split}
            \tr{\Gamma_\nu \log \Gamma_\nu} & \leq (\lambda \pi)^{-K} \int_{P \mathfrak{H}} \frac{\nu(u)}{(\lambda \pi)^{-K}} \log \left(\frac{\nu(u)}{(\lambda \pi)^{-K}} \right) \textnormal{d}u \\
            & = \log (\lambda \pi)^{K} + \int_{P\mathfrak{H}} \nu(u) \log \nu(u) \textnormal{d}u \ . \\
        \end{split}
    \end{equation}
    Gathering the previous estimates together, we finally obtain the upper bound
    \begin{equation} \label{eq:final upperb upperb}
        \begin{split}
            - \log Z_{P}(\lambda) & \leq  \int_{P \mathfrak{H}} W^\varepsilon_K[u] \nu(u) \textnormal{d}u + \int_{P\mathfrak{H}} \nu(u) \log \nu(u) \textnormal{d}u + \int_{P \mathfrak{H}}\ps{u}{hu}\nu(u) \textnormal{d}u \\
            & \quad + C \left( \lambda \varepsilon^{-2} \Lambda^{1+} \int_{P \mathfrak{H}} \|u\|^2 \nu(u) \textnormal{d}u+ \lambda \varepsilon^{-2}\Lambda^{2+} - \lambda \Lambda^{1+} |\log \varepsilon| + \log (\lambda \pi)^{K} \right) \ , \\
        \end{split}
    \end{equation}
    which holds for any probability measure $\nu$ on $P\gH$.
    
    \medskip
    
    \noindent \textbf{Lower bound on $- \log Z_{0,P}(\lambda)$.} Starting from \eqref{eq:vpple Gamma0 proj}, we have
    \begin{equation}
        - \log Z_{0,P}(\lambda) = \lambda \tr{\mathbb{H}_{0,P}\Gamma_{0,P}} + \tr{\Gamma_{0,P}\log \Gamma_{0,P}} \ . \\
    \end{equation}
    Let us define the lower symbol of $\Gamma_{0,P}$ at scale $\lambda$ as in \eqref{eq: lower symbol def}
    \begin{equation} \label{eq: lower symbol proof LB UB}
        \textnormal{d}\Tilde{\nu}(u) = (\lambda \pi)^{-K} \ps{\xi\left(u/\sqrt{\lambda}\right)}{\Gamma_{0,P}\xi\left(u/\sqrt{\lambda}\right)} \textnormal{d}u \ . \\
    \end{equation}
    Next, using \eqref{eq: lower symbols as de Finetti measures estimate} with $k=1$, we have
    \begin{equation}
        \left \lVert  \lambda \Gamma_{0,P}^{(1)} - \int_{P \mathfrak{H}} \ket{u}\bra{u}\textnormal{d}\Tilde{\nu}(u) \right \lVert_{\mathfrak{S}^1(P\mathfrak{H})} \leq \lambda K \tr{\Gamma_{0,P}} \ . \\
    \end{equation}
    Hence, we can write
    \begin{equation}
        \begin{split}
            \lambda \tr{\mathbb{H}_{0,P}\Gamma_{0,P}} &= \lambda \tr{Ph \Gamma_{0,P}^{(1)}} \\
            & = \tr{Ph\int_{P \mathfrak{H}} \ket{u}\bra{u}\textnormal{d}\Tilde{\nu}(u)} + \tr{Ph \left( \lambda \Gamma_{0,P}^{(1)} - \int_{P \mathfrak{H}} \ket{u}\bra{u}\textnormal{d}\Tilde{\nu}(u)\right)} \\
            & \geq \int_{P \mathfrak{H}} \ps{u}{Phu}\textnormal{d}\Tilde{\nu}(u) - \left \lVert Ph \right \lVert_{\textnormal{op}} \left \lVert  \lambda \Gamma_{0,P}^{(1)} - \int_{P \mathfrak{H}} \ket{u}\bra{u}\textnormal{d}\Tilde{\nu}(u) \right \lVert_{\mathfrak{S}^1(P\mathfrak{H})} \\
            & \geq \int_{P \mathfrak{H}} \ps{u}{hu}\textnormal{d}\Tilde{\nu}(u) - C \lambda \Lambda^{2+} \ . \\
        \end{split}
    \end{equation}
    Next, we use the first Berezin-Lieb inequality \eqref{eq:berezinlieb 1}. Similarly as for the upper bound on $- \log Z_P(\lambda)$, we slightly modify this inequality, using \eqref{eq: lower symbol proof LB UB} instead of \eqref{eq:lower symbol BL1}, to find that, for every convex function $f: \mathbb{R}^+ \rightarrow \mathbb{R}$, it holds that
    \begin{equation}
        \tr{f \left(\Gamma_{0,P} \right)} \geq (\lambda \pi)^{-K} \int_{P \mathfrak{H}}f \left( \frac{{\Tilde\nu}(u)}{(\lambda \pi)^{-K}} \right) \textnormal{d}u \ . \\
    \end{equation}
    Hence, using $f(x) = x \log x$ yields
    \begin{equation}
        \tr{\Gamma_{0,P} \log \Gamma_{0,P}} \geq \int_{P \mathfrak{H}} \Tilde{\nu}(u) \log \Tilde{\nu}(u) \textnormal{d}u + \log (\lambda \pi)^{K} \ . \\
    \end{equation}
    Thus, 
    \begin{equation} \label{eq:final lowerb upperb}
        \begin{split}
            - \log Z_{0,P}(\lambda) & \geq \int_{P \mathfrak{H}} \ps{u}{hu}\textnormal{d}\Tilde{\nu}(u) + \int_{P \mathfrak{H}} \Tilde{\nu}(u) \log \Tilde{\nu}(u) \textnormal{d}u + \log (\lambda \pi)^{K} - C \lambda \Lambda^{2+} \\
            & \geq \int_{P \mathfrak{H}} \ps{u}{hu}\textnormal{d}\mu_{0,K}(u) + \int_{P \mathfrak{H}} \mu_{0,K}(u) \log \mu_{0,K}(u) \textnormal{d}u + \log (\lambda \pi)^{K} - C \lambda \Lambda^{2+} \ , \\
        \end{split}
    \end{equation}
    where 
    \begin{equation}
        \textnormal{d}\mu_{0,K}(u) = \frac{1}{z_{0,K}}e^{-\ps{u}{hu}}\textnormal{d}u
    \end{equation}
    is the cylindrical projection of the infinite dimensional Gaussian measure $\mu_0$. We used the fact that $\mu_{0,K}$ is the unique minimizer of the following variational problem
    \begin{equation}
        \inf_{\mu \in \mathcal{P}(P\mathfrak{H})} \left\{ \int_{P \mathfrak{H}}\ps{u}{hu}\textnormal{d}\mu(u) + \int_{P \mathfrak{H}}\textnormal{d}\mu(u) \log \textnormal{d}\mu(u) \right\} \ . \\
    \end{equation}
    \medskip
    
    \noindent \textbf{Conclusion.} Combining \eqref{eq:final upperb upperb} and \eqref{eq:final lowerb upperb}, we obtain, for any probability measure $\nu$ on $P\gH$,
    \begin{equation}
        \begin{split}
            - \log \frac{Z_P(\lambda)}{Z_{0,P}(\lambda)} & \leq \int_{P \mathfrak{H}} W^\varepsilon_K[u] \nu(u) \textnormal{d}u + \int_{P \mathfrak{H}} \ps{u}{hu} \left( \nu(u) - \mu_{0,K}(u) \right) \textnormal{d}u \\
            & \quad + \int_{P \mathfrak{H}} \left( \nu(u) \log \nu(u) - \mu_{0,K}(u) \log \mu_{0,K}(u) \right) \textnormal{d}u \\
            & \quad + C \left( \lambda \varepsilon^{-2} \Lambda^{1+} \int_{P \mathfrak{H}} \|u\|^2 \nu(u) \textnormal{d}u+ \lambda \varepsilon^{-2}\Lambda^{2+} - \lambda \Lambda^{1+} |\log \varepsilon| + \lambda \Lambda^{2+} \right) \ . \\
        \end{split}
    \end{equation}
    A similar computation as in \eqref{eq: relative entropy computation lowerb} shows that
    \begin{equation}
        \begin{split}
            \mathcal{H}_{\textnormal{cl}}(\nu, \mu_{0,K}) &:= \int_{P \mathfrak{H}} \frac{\textnormal{d}\nu}{\textnormal{d}\mu_{0,K}}(u) \log \frac{\textnormal{d}\nu}{\textnormal{d}\mu_{0,K}}(u) \textnormal{d}\mu_{0,K}(u)\\
            & =\int_{P \mathfrak{H}} \ps{u}{hu}\left( \textnormal{d}\nu(u) - \textnormal{d}\mu_{0,K}(u)\right) + \int_{P \mathfrak{H}} \left( \textnormal{d}\nu(u) \log\textnormal{d} \nu(u) - \textnormal{d}\mu_{0,K}(u) \log \textnormal{d}\mu_{0,K}(u) \right)  \\
            & \quad +  \log z_{0,K} \int_{P \mathfrak{H}} \left( \textnormal{d}\nu(u) - \textnormal{d}\mu_{0,K}(u)\right) \ , \\
        \end{split} 
    \end{equation}
    and the last term vanishes, since $\nu$ is a probability measure. This yields
    \begin{equation}
        \begin{split}
            - \log \frac{Z_P(\lambda)}{Z_{0,P}(\lambda)} & \leq \int_{P \mathfrak{H}} W^\varepsilon_K[u] \nu(u) \textnormal{d}u + \mathcal{H}_{\textnormal{cl}}(\nu, \mu_{0,K}) \\
            & \quad + C \left( \lambda \varepsilon^{-2} \Lambda^{1+} \int_{P \mathfrak{H}} \|u\|^2 \nu(u) \textnormal{d}u+ \lambda \varepsilon^{-2}\Lambda^{2+} - \lambda \Lambda^{1+} |\log \varepsilon| + \lambda \Lambda^{2+} \right) \ . \\
        \end{split}
    \end{equation}
    We now choose $\nu$ to be the truncated counterpart of the nonlinear Gibbs measure \eqref{eqnonlinear gibbs meas def in classical model}
    \begin{equation}
        \textnormal{d}\nu(u) = \frac{1}{\mathcal{Z}^\varepsilon_K} e^{- W^\varepsilon_K[u]} \textnormal{d}\mu_{0,K}(u) \quad , \quad \mathcal{Z}^\varepsilon_K = \int_{P \mathfrak{H}} e^{- W^\varepsilon_K[u]} \textnormal{d}\mu_{0,K}(u)
    \end{equation}
   and, by the classical variational principle we obtain
    \begin{equation}
        \int_{P \mathfrak{H}}W^\varepsilon_K[u] \nu(u) \textnormal{d}u + \mathcal{H}_{\textnormal{cl}}(\nu, \mu_{0,K}) = - \log \left( \int_{P \mathfrak{H}} e^{- W^\varepsilon_K[u]} \textnormal{d}\mu_{0,K}(u) \right) \ , \\
    \end{equation}
    thereby concluding the proof.
\end{proof}

Finally we give the 

\begin{proof}[Proof of Proposition \ref{propFree energy upper bound}]
    Combining \eqref{eq: Reduction to a finite dimensional estimate} and \eqref{eq: Finite dimensional semiclassics}, we obtain
    \begin{equation}
        - \log \frac{Z(\lambda)}{Z_0(\lambda)} \leq - \log \left( \int_{P\mathfrak{H}} e^{-W^\varepsilon_K[u]} \text{d}\mu_{0,K}(u)\right) + C {\sigma}(\varepsilon, \lambda, \Lambda) \ , \\
    \end{equation}
    where 
    \begin{equation}
        \sigma(\varepsilon, \lambda, \Lambda) = \Tilde{g}(\varepsilon, \lambda, \Lambda) + \frac{\lambda \varepsilon^{-2} \Lambda^{1+}}{\mathcal{Z}^\varepsilon_K} \int_{P \mathfrak{H}} \|u\|^2e^{- W^\varepsilon_K[u]} \textnormal{d}\mu_{0,K}(u)  + \lambda \varepsilon^{-2}\Lambda^{2+} - \lambda \Lambda^{1+} |\log \varepsilon| + \lambda \Lambda^{2+} \ . \\
    \end{equation}
    Let us now deal with the second term in the above. To do so, we start by noting that, for all $\alpha > 0$
    \begin{equation} \label{eqfernique thm}
        \|u\|_{\mathfrak{H}^{-\alpha}} < + \infty \ \mu_0 \textnormal{-as }\Longleftrightarrow  \tr{h^{-(1+\alpha)}} < + \infty \ . \\
    \end{equation}
    This is \textit{Fernique's theorem}, see \cite[Equation~(3.4)]{LewNamRou-15}. Hence, since we indeed have 
    $$\tr{h^{-(1+\alpha)}} < + \infty \mbox{ for all } \alpha >0 \ , $$
    we control
    \begin{equation}
        \begin{split}
            \int_{P \mathfrak{H}} \|u\|^2e^{- W^\varepsilon_K[u]} \textnormal{d}\mu_{0,K}(u) & = \int_{h \leq \Lambda} \braket{u|u}e^{- W^\varepsilon_K[u]} \textnormal{d}\mu_{0,K}(u) \\
            & \leq \Lambda^{\alpha} \int_{h \leq \Lambda} \braket{u|h^{-\alpha}u}e^{- W^\varepsilon_K[u]} \textnormal{d}\mu_{0,K}(u) \\
            & = \Lambda^{\alpha} \int_{h \leq \Lambda} \|u\|^2_{\mathfrak{H}^{-\alpha}}e^{- W^\varepsilon_K[u]} \textnormal{d}\mu_{0,K}(u) \\
            & \leq \Lambda^{\alpha} \left \lVert \|u\|_{\mathfrak{H}^{-\alpha}} \right \lVert_{L^4(\textnormal{d}\mu_{0})}^2\left \lVert e^{- W^\varepsilon_K} \right \lVert_{L^2(\textnormal{d}\mu_{0})} \ . \\
        \end{split}
    \end{equation}
    In the last line, we used Hölder's inequality. Now, Fernique's theorem actually ensures that \eqref{eqfernique thm} is equivalent to  
\begin{equation}
\int_{\mathfrak{H}^{-\alpha}} e^{\kappa \|u\|^2_{\mathfrak{H}^{-\alpha}}} \textnormal{d} \mu_0(u) < +\infty \quad \textnormal{ for some } \kappa > 0 \ , \\
\end{equation}    
which in particular implies that $\left \lVert \|u\|_{\mathfrak{H}^{-\alpha}} \right \lVert_{L^4(\textnormal{d}\mu_{0})} < +  \infty $. Next, using \cite[Proposition~4.3]{FroKnoSchSoh-22},  we also know that for all $q \geq 1$, there exists a constant $C_q > 0$ such that 
    \begin{equation}
        \left \lVert e^{- W^\varepsilon_K} \right \lVert_{L^q(\textnormal{d}\mu_{0})} \leq C_q \ \ \textnormal{ uniformly in } \varepsilon \ . \\
    \end{equation}
    Now, by Proposition \ref{propConvergence regularized renormalized partition function} below, we have $\mathcal{Z}^\varepsilon_K \longrightarrow \mathcal{Z}$ as $\lambda \rightarrow 0^+$, where $\mathcal{Z}$ is the partition function of the renormalized $\Phi^4_2$ measure \eqref{eqphi42 meas def in classical model}. Hence, we know that for all $\delta > 0$, we have $|\mathcal{Z}^\varepsilon_K - \mathcal{Z}| < \delta$, when $\lambda$ is small enough. Hence, using that $\mathcal{Z}^\varepsilon_K \geq 0$ and the reversed triangle inequality $ |a - b| \geq \left \lvert |a| - |b| \right \lvert$, we have
    \begin{equation}
        \begin{split}
            \mathcal{Z}^\varepsilon_K & = |\mathcal{Z} - (\mathcal{Z} - \mathcal{Z}^\varepsilon_K)| \\
            & \geq |\mathcal{Z}| - |\mathcal{Z} - \mathcal{Z}^\varepsilon_K| \\
            & \geq |\mathcal{Z}| - \delta \ . \\
        \end{split}
    \end{equation}
    Hence choosing $\alpha, \delta >0$ small enough and using \eqref{eq:tilde g in the reduction to finite dim}, we have
    \begin{equation}
        \begin{split}
            \sigma(\varepsilon, \lambda, \Lambda) & \leq  \varepsilon^{-2}\Lambda^{-1/4+}|\log \varepsilon| + \varepsilon^{-3}\lambda \Lambda^{1/4}|\log \lambda|^{1/2} + \varepsilon^{-4}\lambda^2 \Lambda^{1/2}|\log \lambda| \\
            & \quad + \frac{C}{|\mathcal{Z}| - \delta}  \lambda \varepsilon^{-2} \Lambda^{1+}  + \lambda \varepsilon^{-2}\Lambda^{2+} - \lambda \Lambda^{1+} |\log \varepsilon| + \lambda \Lambda^{2+} \\
            & \lesssim \varepsilon^{-2}\Lambda^{-1/4+}|\log \varepsilon| + \varepsilon^{-3}\lambda \Lambda^{1/4}|\log \lambda|^{1/2} + \varepsilon^{-4}\lambda^2 \Lambda^{1/2}|\log \lambda| + \lambda \varepsilon^{-2}\Lambda^{2+} \ . \\
        \end{split}
    \end{equation}
Thus, we deduce that $\nu$ and $\eta$ must satisfy the following conditions in order to have $\sigma(\varepsilon, \lambda, \Lambda) = o(1)$ when $\lambda \rightarrow 0^+$
    \begin{equation}
            \begin{cases}
                \nu > 8 \eta \\
                \nu < 4 - 12 \eta \\
                \nu < 1/2 - \eta
            \end{cases} \ . \\
    \end{equation}
    Notice that we have $1/2 - \eta < 4 - 12 \eta$ if and only if $\eta < 7/22$, which will eventually be the case, since in the lower bound we chose $\eta < 1/24$. Hence we must have
    \begin{equation}
        8 \eta < \nu < \frac{1}{2} - \eta \ . \\
    \end{equation}
    If $8 \eta < \frac{1}{2} - \eta$, ie $\eta < 1/18$ we may always find a suitable $\nu$ to conclude the proof.
\end{proof}

\subsection{Proof of the main result}
Let us now give the 
\begin{proof}[Proof of Theorem \ref{thmConvergence of the relative free-energy}]           First, note that 
    \begin{equation}
        \begin{cases}
            \eta \in (0,1/24) \\
            \nu \in (8 \eta, 1/3)
        \end{cases}
        \Longrightarrow \ \
        \begin{cases}
            \eta \in (0,1/18) \\
            \nu \in (8 \eta, 1/2 - \eta)
        \end{cases} \ . \\
    \end{equation}
    Hence it is sufficient to assume the conditions from Proposition \ref{propfree energy lower bound}. Combining \eqref{eqlower bound in prop} and \eqref{equpper bound in prop}, we obtain
    \begin{equation} \label{eqfinal eq before limit}
        -\log \frac{Z(\lambda)}{Z_0(\lambda)} = - \log \left( \int_{P \mathfrak{H}}e^{-W^\varepsilon_K[u]} \textnormal{d}\mu_{0,K}(u) \right) + o_{\lambda \rightarrow0^+}(1) \ . \\
    \end{equation}
    Now using \eqref{eqconvergence Phi42 in prop} from Proposition \ref{propConvergence regularized renormalized partition function}, and the continuity of the logarithm, we find that
    \begin{equation}
        - \log \left(\int_{P\mathfrak{H}}e^{-W^\varepsilon_K(u)}\textnormal{d}\mu_{0,K}(u)\right) \underset{\lambda \rightarrow 0^+}{\longrightarrow}  - \log \left(\int e^{-V(u)}\textnormal{d}\mu_{0}(u) \right)
    \end{equation}
    Hence, passing to the limit in \eqref{eqfinal eq before limit}, we get
    \begin{equation}
         -\log \frac{Z(\lambda)}{Z_0(\lambda)}\underset{\lambda \rightarrow 0^+}{\longrightarrow}  - \log \left(\int e^{-V(u)}\textnormal{d}\mu_{0}(u) \right)
    \end{equation}
    as desired.
\end{proof}

\newpage
\appendix

\section{Lower symbol for the free Gibbs state}
Here we compute the lower symbol (in the sense of Definition~\ref{def:lowsymb}) associated to the free Gibbs state
\begin{equation} \label{eq:free gibbs state appendix}
    \Gamma_0 = \frac{e^{-\lambda\dgamma{h}}}{\textnormal{Tr}_{\mathfrak{F}(\mathfrak{H})} \left[ e^{-\lambda\dgamma{h}} \right]}
\end{equation}
on $P \mathfrak{H}$. One important fact that will simplify the calculations is that 
\begin{equation}\label{eq:bidule}
    \left( \Gamma_0 \right)_P = \Gamma_{0,P} = \frac{e^{-\lambda\dgamma{Ph}}}{\textnormal{Tr}_{\mathfrak{F}(P\mathfrak{H})} \left[ e^{-\lambda\dgamma{Ph}} \right]} \ . \\
\end{equation}
We then claim the following

\begin{lemma}[\textbf{Lower symbol for the free Gibbs state}]\label{lemma:Lower symbol for the Free gibbs state}\mbox{}\\
    Let $\Gamma_0$ be as in \eqref{eq:free gibbs state appendix}. Let $h = -\Delta + 1$ and $P = \mathbf{1}_{h \leq \Lambda}$, for some $\Lambda>0$. Consider the lower symbol of $\Gamma_0$ on $P\mathfrak{H}$ at scale $\lambda > 0$, defined by
    \begin{equation}
        \textnormal{d} \mu^\lambda_{P,\Gamma_0}(u) = (\lambda \pi)^{-K} \ps{\xi(u/\sqrt{\lambda})}{\Gamma_{0,P} \xi(u/\sqrt{\lambda})}_{\mathfrak{F}(P\mathfrak{H})} \textnormal{d}u \ , \\
    \end{equation}
    where $\textnormal{d}u$ is the Lebesgue measure on $P\mathfrak{H} \simeq \mathbb{C}^K$, with $K = \tr{P}$ and the coherent state $\xi(u/\sqrt{\lambda})$ is as in Definition~\ref{def:coh}. Then
    \begin{equation}
        \textnormal{d} \mu^\lambda_{P,\Gamma_0}(u) = \frac{1}{\Tilde{z}_0}e^{-\ps{u}{\Tilde{h}u}}\text{d}u \ , \\
    \end{equation}
    where $\Tilde{z}_0= (\lambda \pi)^{K}\prod_{j=1}^K\left(1-e^{-\lambda \lambda_j} \right)^{-1}$, and $\Tilde{h}= \lambda^{-1}\left(1-e^{-\lambda h} \right)$.
\end{lemma}

\begin{proof}
    We start from~\eqref{eq:bidule} and recall the spectral decomposition $h = \sum_{j\geq1}\lambda_j \ket{u_j}\bra{u_j}$. We can write $\mathfrak{H} = \bigoplus_{j\geq1} \left(\mathbb{C}u_j \right)$, which, using the factorization property of the Fock space, implies that
    \begin{equation}
        \mathfrak{F}(\mathfrak{H}) \simeq \bigotimes_{j=1}^\infty \mathfrak{F}(\mathbb{C}u_j) \ . \\
    \end{equation}
    Quite similarly, since $Ph = \sum_{j=1}^K\lambda_j \ket{u_j}\bra{u_j}$, we have for the finite dimensional Hilbert space $P \mathfrak{H}$, the corresponding Fock space
    \begin{equation}
        \mathfrak{F}(P\mathfrak{H}) \simeq \bigotimes_{j=1}^K \mathfrak{F}(\mathbb{C}u_j)
    \end{equation}
    Denoting $a_j = a(u_j)$ and $a^\dagger_j = a^\dagger(u_j)$ we have
    \begin{equation}
        \dgamma{Ph} = \sum_{j=1}^K\lambda_j a^\dagger_ja_j \ . \\
    \end{equation}
    Hence, using the CCR, we can write
    \begin{equation}
        e^{-\lambda \dgamma{Ph}} = \bigotimes_{j=1}^Ke^{-\lambda\lambda_ja^\dagger_ja_j} \ . \\
    \end{equation}
    Thus
    \begin{equation}
        \begin{split}
            \text{Tr}_{\mathfrak{F}(\mathfrak{H})} \left[ e^{-\lambda \dgamma{Ph}}\right] & = \prod_{j=1}^K\text{Tr}_{\mathfrak{F}(\mathbb{C}u_j)} \left[ e^{-\lambda\lambda_ja^\dagger_ja_j}\right] \\
            & = \prod_{j=1}^K \sum_{n=1}^\infty e^{-\lambda\lambda_jn} \\
            & = \prod_{j=1}^{K}\left(1-e^{-\lambda \lambda_j} \right)^{-1} \ . \\
        \end{split}
    \end{equation}
    Hence, we obtain the followig expression for $\Gamma_{0,P}$
    \begin{equation}\label{eq:gamma0P}
        \Gamma_{0,P} = \bigotimes_{j=1}^K\left(1-e^{-\lambda \lambda_j} \right)e^{-\lambda\lambda_ja^\dagger_ja_j} \ . \\
    \end{equation}
    Writing a generic vector as $v = \sum_{j=1}^K\alpha_ju_j$, where $\alpha_j = \ps{u_j}{v}$ we obtain, using again the CCR, that the Weyl operator from Definition~\ref{def:coh} acts as
    \begin{equation}
        W(v) = \bigotimes_{j=1}^K \exp \left(\alpha_ja^\dagger_j - \overline{\alpha_j} a_j \right) \ . \\
    \end{equation}
    Hence, using the Baker-Campbell-Haussdorff formula and the properties of the creation and annihilation operators, we get
    \begin{equation} \label{eq:bch}
         \begin{split}
             \xi\left(\frac{u}{\sqrt{\lambda}} \right) & = \bigotimes_{j=1}^K \exp \left(\frac{\alpha_j}{\sqrt{\lambda}}a^\dagger_j - \frac{\overline{\alpha_j}}{\sqrt{\lambda}} a_j \right) \ket{0} \\
             & = \bigotimes_{j=1}^Ke^{-\frac{|\alpha_j|^2}{2\lambda}}e^{\frac{\alpha_j}{\sqrt{\lambda}}a^\dagger_j}e^{- \frac{\overline{\alpha_j}}{\sqrt{\lambda}} a_j} \ket{0}_j \\
             & = \bigotimes_{j=1}^Ke^{-\frac{|\alpha_j|^2}{2\lambda}}\bigoplus_{n=0}^\infty \frac{\alpha_j^n}{\sqrt{n!}}u_j^{\otimes n} \ . \\
         \end{split}
    \end{equation}
    Combining \eqref{eq:gamma0P} and \eqref{eq:bch}, we  find
    \begin{equation}
        \Gamma_{0,P} \xi\left(\frac{u}{\sqrt{\lambda}} \right) = \bigotimes_{j=1}^K \left(1-e^{-\lambda \lambda_j} \right)e^{-\frac{|\alpha_j|^2}{2\lambda}}  \bigoplus_{n=0}^\infty \frac{\alpha_j^n}{\sqrt{n!}}e^{-\lambda\lambda_ja^\dagger_ja_j}u_j^{\otimes n} \ . \\
    \end{equation}
    Now, since the creation and annihilation operators are bounded operators on each sector of the Fock space, we can expand the quantity $e^{-\lambda\lambda_ja^\dagger_ja_j}$ without any domain considerations. Using their definition, we first notice that for all $n \geq 1$, 
    \begin{equation}
        \begin{split}
            a^\dagger_ju_j^{\otimes n} & = \sqrt{n+1}u_j^{\otimes (n+1)} \\
            a_ju_j^{\otimes n} & = \sqrt{n} u_j^{\otimes n}
        \end{split}
    \end{equation}
    so that, using the CCR, we have
    \begin{equation}
        \begin{split}
            e^{-\lambda\lambda_ja^\dagger_ja_j}u_j^{\otimes n} & = \sum_{k=0}^\infty (-1)^k \frac{(\lambda \lambda_j)^k}{k!}(a^\dagger_j a_j)^ku_j^{\otimes n} \\
            & = \sum_{k=0}^\infty (-1)^k \frac{(\lambda \lambda_j)^k}{k!}n^ku_j^{\otimes n} \\
            & = e^{-\lambda \lambda_j n}u_j^{\otimes n} \ . \\
        \end{split}
    \end{equation}
    Now we can compute explicitely $\text{d}\mu^\lambda_{P,\Gamma_0}$. Using the definition of the scalar product on the Fock space, we have
    \begin{equation}
        \begin{split}
            \text{d}\mu^\lambda_{P,\Gamma_0}(u) & = (\lambda\pi)^{-K} \prod_{j=1}^K \left \langle e^{-\frac{|\alpha_j|^2}{2\lambda}}\bigoplus_{n=0}^\infty \frac{\alpha_j^n}{\sqrt{n!}}u_j^{\otimes n} , \left(1-e^{-\lambda \lambda_j} \right)e^{-\frac{|\alpha_j|^2}{2\lambda}}  \bigoplus_{n=0}^\infty \frac{\alpha_j^n}{\sqrt{n!}}e^{-\lambda \lambda_j n}u_j^{\otimes n} \right \rangle_{\mathfrak{F}(\mathbb{C}u_j)} \text{d}u \\
            & = (\lambda\pi)^{-K} \prod_{j=1}^K e^{-\frac{|\alpha_j|^2}{\lambda}}\left(1-e^{-\lambda \lambda_j} \right) \sum_{n=0}^\infty \frac{\left( \frac{|\alpha_j|^2}{\lambda} e^{-\lambda \lambda_j}\right)^n}{n!} \ps{u_j^{\otimes n}}{u_j^{\otimes n}} \text{d}u\\
            & = (\lambda\pi)^{-K} \prod_{j=1}^K \left(1-e^{-\lambda \lambda_j} \right)e^{-\frac{|\alpha_j|^2}{\lambda} \left( 1-e^{-\lambda \lambda_j} \right)} \text{d}u\\
            & = \frac{1}{\Tilde{z}_0}e^{-\ps{u}{\Tilde{h}u}}\text{d}u
        \end{split}
    \end{equation}
    with $\Tilde{z}_0$ and $\Tilde{h}$ as in the statement of the lemma. This gives the result
\end{proof}

\section{Berezin-Lieb inequalities on the Fock space}\label{app:Berezin}
We follow closely the lines of \cite[Appendix~B]{Rougerie-LMU} to prove two Berezin-Lieb~\cite{Berezin-72,Lieb-73b,Simon-80} inequalities on the bosonic Fock space $\mathfrak{F}(\mathfrak{K})$ of a finite dimensional Hilbert space $\mathfrak{K}$. The main difference here is that we start with a coherent state partition of unity, whereas the Berezin-Lieb inequalities in \cite{Rougerie-LMU} use a Hartree state partition of unity.

We  start by recalling the coherent state partition of unity, implied by \eqref{eq: coherent state} (see e.g.~\cite{Rougerie-EMS})
\begin{lemma}[\textbf{Coherent state partition of unity}]\mbox{}\\
    Let $\mathfrak{K}$ be a finite dimensional Hilbert space, and denote $K = \dim \mathfrak{H}$. Let $\left\{u_k\right\}_{k=1}^K$ be an orthonormal basis of $\mathfrak{K}$. Let $\textnormal{d}u$ be the Lebesgue measure on $\mathfrak{K} \simeq \mathbb{C}^K$. Let $\xi(u)$ be the coherent state defined in \eqref{eq: coherent state}. Then, the following holds
    \begin{equation} \label{eq: cs partition unity}
        \pi^{-K} \int_{\mathfrak{K}} \ket{\xi(u)} \bra{\xi(u)} \textnormal{d}u = \mathbf{1}_{\mathfrak{F}(\mathfrak{K})} \ . \\
    \end{equation}
\end{lemma}

We can now state the Berezin-Lieb inequalities. We recall that the set of states on any Hilbert space $\mathfrak{X}$ is defined by
\begin{equation}
    \mathcal{S}(\mathfrak{X})= \left\{ \Gamma = \Gamma^* \geq 0 \ \vert \ \text{Tr}_{\mathfrak{X}} \left[ \Gamma \right] = 1 \right\} \ . \\
\end{equation}

\begin{lemma}[\textbf{First Berezin-Lieb inequality}]\mbox{}\\
    Let $\mathfrak{K}$ be a finite dimensional Hilbert space of dimension $K$. Let $\Gamma \in \mathcal{S}(\mathfrak{F}(\mathfrak{K}))$. Define its \textnormal{lower symbol} $\nu^{\textnormal{low}}$ as
    \begin{equation} \label{eq:lower symbol BL1}
        \nu^{\textnormal{low}}(u) =  \pi^{-K} \ps{\xi(u)}{\Gamma \xi(u)}_{\mathfrak{F}(\mathfrak{K})} \ . \\
    \end{equation}
    Then, for every convex function $f:\mathbb{R}^+ \rightarrow \mathbb{R}$, we have
    \begin{equation} \label{eq:berezinlieb 1}
        \tr{f(\Gamma)} \geq \pi^{-K} \int_{\mathfrak{K}}f \left( \frac{\nu^{\textnormal{low}}(u)}{\pi^{-K}} \right) \text{d}u \ . \\
    \end{equation}
\end{lemma}

\begin{proof}
 Mimick that of~\cite{Simon-80} along the lines of~\cite[Proof of Lemma~B.3]{Rougerie-LMU}.     
\end{proof}

\begin{lemma}[\textbf{Second Berezin-Lieb inequality}]\mbox{}\\
    Let $\mathfrak{K}$ be a finite dimensional Hilbert space of dimension $K$. Let $\Gamma \in \mathcal{S}(\mathfrak{F}(\mathfrak{K}))$ having an \textnormal{upper symbol} $\nu^{\textnormal{up}} \geq 0$, i.e
    \begin{equation} \label{eq:upper symbol BL1}
        \Gamma = \int_{\mathfrak{K}} \nu^{\textnormal{up}}(u) \ket{\xi(u)} \bra{\xi(u)} \textnormal{d}u \ . \\
    \end{equation}
Then, for every convex function $f:\mathbb{R}^+ \rightarrow \mathbb{R}$, we have
    \begin{equation} \label{eq:Berezinlieb 2}
        \tr{f(\Gamma)} \leq \pi^{-K} \int_{\mathfrak{K}}f \left( \frac{\nu^{\textnormal{up}}(u)}{\pi^{-K}} \right) \text{d}u \ . \\
    \end{equation}
\end{lemma}

\begin{proof}
 Mimick that of~\cite{Simon-80} along the lines of~\cite[Proof of Lemma~B.4]{Rougerie-LMU}.     
\end{proof}

\section{Regularization of the $\Phi^4_2$ measure by convolution}

We have used that the finite dimensionally truncated partition function of the $\Phi^4_2$ measure can be uniformly approximated by that of the corresponding Hartree measure, with smeared non-linearity. This follows from results of~\cite{FroKnoSchSoh-22}, as we explain below, using their notation for simplicity. Let us define
\begin{equation} \label{eqdef W V in FKSS}
    \begin{split}
        W^\varepsilon(u) &= \frac{1}{2}\iint_{\mathbb{T}^2 \times \mathbb{T}^2} w^\varepsilon(x-y) \wick{|u(x)|^2}\wick{|u(y)|^2} \textnormal{d}x\textnormal{d}y - \tau^\varepsilon \int_{\mathbb{T}^2}\wick{|u(x)|^2} \textnormal{d}x - E^\varepsilon \\
        V^\varepsilon(u) & =  \frac{1}{2}\int_{\mathbb{T}^2} w^\varepsilon(x-y) \wick{|u(x)|^2|u(y)|^2} \textnormal{d}x \textnormal{d}y \\
        V(u) & = \frac{1}{2}\int_{\mathbb{T}^2}  \wick{|u(x)|^4} \textnormal{d}x \ . \\
    \end{split}
\end{equation}
These are understood as $L^2(\textnormal{d}\mu_0)$ limits of their truncated counterparts $W^\varepsilon_K$, $V^\varepsilon_K$ and $V_K(u)$, that are well-defined on $P\mathfrak{H}$. A crucial result of~\cite{FroKnoSchSoh-22} is that $W^\varepsilon,V^\varepsilon \to V$ when $\varepsilon \to 0$. In the main text, we have commuted the $\varepsilon \to 0$ and $K\to \infty$ limit, a step for which we now provide justifications. \\

The following is obtained by a straightforward adaptation of the proof of \cite[Lemma 4.5]{FroKnoSchSoh-22}:

\begin{lemma}[\textbf{$W^\varepsilon$ versus $V^\varepsilon$}] \label{lemmacontrol of V W K eps}\mbox{}\\
    Let $\varepsilon >0$. Consider $W^\varepsilon_K$ and $V^\varepsilon_K$ the truncated counterparts of $W^\varepsilon$ and $V^\varepsilon$, as defined above in \eqref{eqdef W V in FKSS}. Then, for all $0 < K \leq + \infty$, there exists a function $f$ such that, uniformly in $K$, we have
    \begin{equation}
        \left \lVert V^\varepsilon_K - W^\varepsilon_K \right \lVert_{L^{2}(\textnormal{d}\mu_0)} \leq f(\varepsilon) \underset{\varepsilon \rightarrow 0^+}{\longrightarrow 0} \ . \\
    \end{equation}
\end{lemma}

The convergence result for this section is then the following
\begin{proposition}[\textbf{Joint $K\to \infty$ and $\varepsilon \to 0$ limit}]\label{propConvergence regularized renormalized partition function}\mbox{}\\
    Let $\varepsilon > 0$  and $\Lambda = \varepsilon^{-\nu}$ for some $\nu > 0$. Consider the Laplacian $h = - \Delta+1$ acting on $\mathfrak{H} = L^2_\textnormal{s}(\mathbb{T}^2)$ and the orthogonal projector on low modes $P = \mathbf{1}_{h\leq \Lambda}$. Denote $K = \tr{P}$. Define the truncated renormalized interaction
    \begin{equation}
        W^\varepsilon_K(u) = \frac{1}{2}\iint_{\mathbb{T}^2 \times \mathbb{T}^2} w^\varepsilon(x-y) \wick{|Pu(x)|^2}\wick{|Pu(y)|^2} \textnormal{d}x\textnormal{d}y - \tau^\varepsilon \int_{\mathbb{T}^2}\wick{|Pu(x)|^2} \textnormal{d}x - E^\varepsilon
    \end{equation}
    and the cylindrical projection of the Gaussian measure $\mu_0$ on $P\mathfrak{H}$
    \begin{equation}
        \textnormal{d}\mu_{0,K}(u) = \frac{1}{\mathcal{Z}_{0,K}}e^{-\ps{u}{hu}} \textnormal{d}u \ . \\
    \end{equation}
    We have
    \begin{equation} \label{eqconvergence Phi42 in prop}
        \int_{P\mathfrak{H}}e^{-W^\varepsilon_K(u)}\textnormal{d}\mu_{0,K}(u) \underset{\varepsilon \rightarrow 0^+}{\longrightarrow}  \int e^{-V(u)}\textnormal{d}\mu_{0}(u) \ . \\
    \end{equation}
\end{proposition}

\begin{proof}
    By definition of $\mu_{0,K}$, we have
    \begin{equation}
        \begin{split}
            \left \lvert \int_{P\mathfrak{H}}e^{-W^\varepsilon_K(u)}\textnormal{d}\mu_{0,K}(u) -  \int e^{-V(u)}\textnormal{d}\mu_{0}(u) \right \lvert &= \left \lvert \int \left( e^{-W^\varepsilon_K(u)} -  e^{-V(u)} \right) \textnormal{d}\mu_0(u) \right \lvert \\
            & \leq \left \lVert e^{-W^\varepsilon_K} -  e^{-V} \right \lVert_{L^1(\textnormal{d}\mu_0)} \ . \\
        \end{split}
    \end{equation}
    Since $\mu_0$ is a probability measure, we can write, using Hölder's inequality
    \begin{equation*}
        \begin{split}
            \left \lVert e^{-W^\varepsilon_K} -  e^{-V} \right \lVert_{L^1(\textnormal{d}\mu_0)} & \leq \left \lVert e^{-W^\varepsilon_K} -  e^{-V} \right \lVert_{L^2(\textnormal{d}\mu_0)} \\
            & \leq \int_0^1 \left \lVert (W^\varepsilon_K - V)e^{-t W^\varepsilon_K - (1-t)V} \right \lVert_{L^2(\textnormal{d}\mu_0)} \textnormal{d}t \\
            & \leq \left \lVert W^\varepsilon_K - V\right \lVert_{L^{2r}(\textnormal{d}\mu_0)} \int_0^1  \left \lVert e^{-t W^\varepsilon_K - (1-t)V} \right \lVert_{L^{2q}(\textnormal{d}\mu_0)} \textnormal{d}t \\
            & \leq \left \lVert W^\varepsilon_K - V\right \lVert_{L^{2r}(\textnormal{d}\mu_0)} \int_0^1  \left \lVert e^{- W^\varepsilon_K} \right \lVert^t_{L^{2q}(\textnormal{d}\mu_0)} \left \lVert e^{-V} \right \lVert^{1-t}_{L^{2q}(\textnormal{d}\mu_0)} \textnormal{d}t \\
            & \leq \left \lVert W^\varepsilon_K - V\right \lVert_{L^{2r}(\textnormal{d}\mu_0)} \int_0^1  C^t {\Tilde{C}}^{1-t} \textnormal{d}t \ . \\
        \end{split}
    \end{equation*}
    In the third and fourth lines, we used H\"older's inequality with respectively $1/r + 1/q = 1$ and 
    $$1/(1/t) + 1/(1/(1-t)) = 1 \ .$$ 
    In the fifth line, we used the fact that $\left \lVert e^{- W^\varepsilon} \right \lVert_{L^{p}(\textnormal{d}\mu_0)}$ is uniformly bounded in $\varepsilon$ for all $p \geq 1$ (see \cite[Lemma 4.6]{FroKnoSchSoh-22}), and so is $\left \lVert e^{- W_K^\varepsilon} \right \lVert_{L^{p}(\textnormal{d}\mu_0)}$. Next, we have from \eqref{eqint ren with extra term I}
    \begin{equation}
        \begin{split}
            W^\varepsilon & =  \frac{1}{2}\int_{\mathbb{T}^2} w^\varepsilon(x-y) \wick{|u(x)|^2|u(y)|^2} \textnormal{d}x \textnormal{d}y \\
            & \quad \quad + \iint_{\mathbb{T}^2 \times \mathbb{T}^2}w^\varepsilon(x-y)G(x-y) \left( \wick{\overline{u(x)}u(y)} - \wick{|u(x)|^2} \right)\textnormal{d}x \textnormal{d}y \ , \\
        \end{split}
    \end{equation}
    with a similar expression for $W^\varepsilon_K$. Hence, in view of \eqref{eqdef W V in FKSS}, we see that $W^\varepsilon_K - V$ lives in the fourth polynomial chaos (see \cite[Appendix~A]{FroKnoSchSoh-22}, or~\cite{Hairer-16} for a definition). Hence, using \cite[Lemma~1.4]{FroKnoSchSoh-22} gives 
    \begin{equation}
        \left \lVert W^\varepsilon_K - V\right \lVert_{L^{2r}(\textnormal{d}\mu_0)} \leq 4Cr^2 \left \lVert W^\varepsilon_K - V\right \lVert_{L^{2}(\textnormal{d}\mu_0)}
    \end{equation}
    for some universal constant $C > 0$. Let us now deal with the term $\left \lVert W^\varepsilon_K - V\right \lVert_{L^{2}(\textnormal{d}\mu_0)}$. Let $L > 0$ that we will choose later. By the triangle inequality, we have
    \begin{equation}
        \begin{split}
            \left \lVert W^\varepsilon_K - V\right \lVert_{L^{2}(\textnormal{d}\mu_0)} & \leq \left \lVert W^\varepsilon_K - V^\varepsilon_K\right \lVert_{L^{2}(\textnormal{d}\mu_0)} + \left \lVert V^\varepsilon_K - V\right \lVert_{L^{2}(\textnormal{d}\mu_0)} \\
            & \leq \left \lVert W^\varepsilon_K - V^\varepsilon_K\right \lVert_{L^{2}(\textnormal{d}\mu_0)} + \left \lVert V^\varepsilon_K - V^\varepsilon\right \lVert_{L^{2}(\textnormal{d}\mu_0)} + \left \lVert V^\varepsilon - V\right \lVert_{L^{2}(\textnormal{d}\mu_0)} \\
            & \leq \left \lVert W^\varepsilon_K - V^\varepsilon_K\right \lVert_{L^{2}(\textnormal{d}\mu_0)} + \left \lVert V^\varepsilon_K - V^\varepsilon_L\right \lVert_{L^{2}(\textnormal{d}\mu_0)} + \left \lVert V^\varepsilon_L - V^\varepsilon\right \lVert_{L^{2}(\textnormal{d}\mu_0)} + \left \lVert V^\varepsilon - V\right \lVert_{L^{2}(\textnormal{d}\mu_0)} \ . \\
        \end{split}
    \end{equation}
    The last two terms of the right hand side satisfy
    \begin{equation*}
        \begin{split}
            \left \lVert V^\varepsilon_L - V^\varepsilon\right \lVert_{L^{2}(\textnormal{d}\mu_0)} & \underset{L \rightarrow + \infty}{\longrightarrow} 0 \\
            \left \lVert V^\varepsilon - V\right \lVert_{L^{2}(\textnormal{d}\mu_0)} & \underset{\varepsilon \rightarrow 0}{\longrightarrow} 0 \ . \\
        \end{split}
    \end{equation*}
    by construction, and by \cite[Lemma 4.2]{FroKnoSchSoh-22}, respectively. Now, for the second term, using \cite[Lemma 4.5]{FroKnoSchSoh-22}, we have, for all $0 < K \leq L < + \infty$ and for any fixed $\delta > 0$, 
    \begin{equation} \label{eqcauchy control on VKvarepsilon}
        \left \lVert V^\varepsilon_K - V^\varepsilon_L \right \lVert_{L^{2}(\textnormal{d}\mu_0)} \lesssim_\delta K^{-1+\delta} \ . \\
    \end{equation}
    Now, 
    Taking $L > 0$ large enough, \eqref{eqcauchy control on VKvarepsilon} does hold for all fixed $K \leq L$. Thus, using Lemma \ref{lemmacontrol of V W K eps} we get the following estimate
    \begin{equation}
        \left \lVert e^{-W^\varepsilon_K} -  e^{-V} \right \lVert_{L^1(\textnormal{d}\mu_0)} \lesssim_\delta  f(\varepsilon) +  K^{-1+\delta} + \left \lVert V^\varepsilon_L - V^\varepsilon \right \lVert_{L^{2}(\textnormal{d}\mu_0)} + \left \lVert V^\varepsilon - V\right \lVert_{L^{2}(\textnormal{d}\mu_0)} \ . \\
    \end{equation}
    Taking $L$ to infinity in the above, we obtain
    \begin{equation}
         \left \lVert e^{-W^\varepsilon_K} -  e^{-V} \right \lVert_{L^1(\textnormal{d}\mu_0)} \lesssim_\delta  f(\varepsilon) +  K^{-1+\delta} + \left \lVert V^\varepsilon - V\right \lVert_{L^{2}(\textnormal{d}\mu_0)} \ . \\
    \end{equation}
    It now remains to take $\varepsilon$ to $0$, and use the facts that $\Lambda \underset{\varepsilon \rightarrow 0}{\longrightarrow} + \infty$, and
    \begin{equation}
        \begin{split}
            K  = \tr{P} = \sum_{k \in 2 \pi \mathbb{Z}^2} \mathbf{1}_{|k|^2+1 \leq \Lambda} \underset{\Lambda \rightarrow + \infty}{ \longrightarrow} + \infty \ .
        \end{split}
    \end{equation}
    We hence obtain
    \begin{equation}
        \left \lVert e^{-W^\varepsilon_K} -  e^{-V} \right \lVert_{L^1(\textnormal{d}\mu_0)} \underset{\lambda \rightarrow 0^+}{\longrightarrow} 0
    \end{equation}
    and the result follows. \\
\end{proof}


\end{document}